\documentclass[11pt]{article}
\usepackage{amsmath}
\usepackage{amssymb}
\usepackage{amsfonts}
\usepackage{mathrsfs}
\usepackage{mathtools}
\usepackage{enumerate}
\usepackage[usenames]{color}
\usepackage[numbers]{natbib}
\usepackage{enumitem} 
\usepackage{amscd} 
\usepackage{mathrsfs}
\usepackage[normalem]{ulem} 
\usepackage{comment}
\usepackage{bbm}

\usepackage[pdftex]{graphicx}
\usepackage[all]{xy}

\usepackage{hyperref}
\hypersetup{colorlinks=true,linkcolor=blue,
citecolor=blue} 

\usepackage{graphicx}
\usepackage{amsthm}

\numberwithin{equation}{section}

\newtheorem{theorem}{Theorem}[section]
\newtheorem{lemma}[theorem]{Lemma}
\newtheorem{proposition}[theorem]{Proposition}
\newtheorem{cor}[theorem]{Corollary}
\newtheorem{rem}[theorem]{Remark}

\newtheorem{definition}[theorem]{Definition}

\newcommand{\bd}{\begin{displaymath}}
\newcommand{\ed}{\end{displaymath}}
\newcommand{\N}{{\mathbb{N}}}
\newcommand{\Z}{{\mathbb{Z}}}

\newcommand{\R}{{\mathbb{R}}}
\newcommand{{\rd}}{\R^d}

\newcommand{\lan}{\langle}
\newcommand{\ran}{\rangle}

\newcommand{\e}{\varepsilon}
\newcommand{\ve}{\epsilon}

\newcommand{\dd}{\text{\rm d}}             

\newcommand{\cE }{{\cal E}}
\newcommand{\cF }{{\cal F}}

\newcommand{\var}[1]{\mathrm{Var}\left({#1}\right)}

\newcommand{\dis}{\displaystyle}

\usepackage{makeidx}
\makeindex

\allowdisplaybreaks[4]

\author{Sergio Albeverio\thanks{Institute for Applied Mathematics and HCM, University of Bonn, Bonn, Germany}\and Seiichiro Kusuoka\thanks{Department of Mathematics, Graduate School of Science, Kyoto University, Kyoto, Japan} \and Song Liang\thanks{School of Education, Waseda University, Tokyo, Japan} \and Makoto Nakashima\thanks{nakamako@math.nagoya-u.ac.jp, Graduate School of Mathematics, Nagoya University, Furocho, Chikusaku, Nagoya, Japan
} }

\begin{document}
\title{Stochastic Quantization of the three-dimensional polymer measure via the Dirichlet form method}

\pagestyle{myheadings}
\markboth{Three dimensional polymer measure}{Three dimensional polymer measure}
\date{}

\maketitle


\begin{center}{\it
Dedicated to the Memory of Xian Yin Zhou, with great admiration and gratitude for his work, on which the present contribution builds
}\end{center}

\begin{abstract}
We prove that there exists a diffusion process whose invariant measure is the three dimensional polymer measure $\nu_\lambda$ for all $\lambda>0$.
We follow in part a previous incomplete unpublished work of the first named author with M. R{\"o}ckner and X.Y. Zhou \cite{ARZ96}. For the construction of $\nu _\lambda$ we rely on previous work by J. Westwater, E. Bolthausen and X.Y. Zhou.
Using $\nu _\lambda$, the diffusion is constructed by means of the theory of Dirichlet forms on infinite-dimensional state spaces. The closability of the appropriate pre-Dirichlet form which is of gradient type is proven, by using a general closability result in \cite{AR89a}. This result does not require an integration by parts formula  (which does not even hold for the two-dimensional polymer measure $\nu_\lambda$) but requires the quasi-invariance of $\nu_\lambda$ along a basis  of vectors in the classical Cameron-Martin space such that the Radon-Nikodym derivatives have versions which form a continuous process.

\end{abstract}

\vspace{1em}


\vspace{1em}{\bf Key words:} Three-dimensional polymer measure, closability, Dirichlet forms, diffusion processes, quasi-invariance.


\section{Introduction}

\subsection{Polymer measure}\label{subsec:1.1}

In the literature on polymer physics, random walks and Brownian motions can be regarded as realizations of polymer chains that consist of a huge number of monomers. Moreover, it is natural to think of self-interaction of polymer chains, e.g.~an attractive or repulsive effect due to Van der Waals force between monomers, and a repulsive effect due to the physical restriction such that different monomers should not occupy the same point in space.
Edwards' model (see \cite{Edw65}) is a probabilistic model for long polymer chains that takes into account the self-exclusive volume effects.   It is formally given by a probability measure on Wiener space:
\begin{align}\label{eq:polmea}
\nu_\lambda(\dd\omega)=N_\lambda^{-1}\exp\left(-\lambda \int_0^1\int_s^1 \delta_0(\omega_t-\omega_s)\dd s \dd t\right)\nu_0(\dd \omega),
\end{align}
where $\nu_0$ is the Wiener measure and\begin{align*}
J(\omega):= \int_0^1\int_s^1 \delta_0(\omega_t-\omega_s)\dd s \dd t
\end{align*} is the self-interaction local time of the $d$-dimensional Brownian motion $\{\omega_t\}_{t\geq 0}$ in the time interval $[0,1]$,  $\lambda\in [0,\infty)$ is a coupling constant, and $N_\lambda$ is the normalizing constant to assure that $\nu_\lambda$ is a probability measure. We will give their precise definitions below. 

In constructive quantum field theory (CQFT), Symanzik independently introduced a representation of $\phi^4$ theory in terms of a gas of Brownian paths (see \cite{Sym69}). In this representation, a self-intersection local time naturally appear. Later, Brydges, Fr\"ohlich, Sokal, and Spencer established the existence and nontriviality of $\phi^4_2$ and $\phi^4_3$ via a random walk representation (discrete counterpart to \eqref{eq:polmea}), Schwinger-Dyson equation, and skeleton inequality (see \cite{BFS82, BFS83a, BFS83b, BF84, FFS92}).    

The intersection properties of Brownian paths have been investigated since
the forties \cite{Lev40}. In \cite{DEK50,DEKT57} Dovorezcky, Erd\"{o}s, Kakutani, and Taylor  proved that the $d$-dimensional Brownian motion $\omega= \{\omega_t\}_{t\geq 0}$ has no multiple points when $d\geq 4$ almost surely so that $J=0$  and   $\nu_\lambda$ is trivial, whereas it has multiple points for $d\leq 3$ almost surely, that is $\{0\leq s< t\leq 1:\omega_s=\omega_t\}\not=\emptyset$. Hence, we focus on the case for $d\leq 3$. We remark that the critical dimension of multiple points is $d=6$ from the viewpoint of quasi-everywhere \cite{Lyo86}.  (In connection with the triviality of $\phi^4_4$ theory, intersection local time and Edwards' model in $d=4$ are discussed in \cite{Aiz85, ADC21, AK95, AFHKL86, AZ94, ABZ04, Fro82, Fro83, Hab23, Nel83}.)

In order to describe  $J$ and $\nu_\lambda$ more precisely, we start with recalling some facts about the $d$-dimensional polymer measure.  Let $X:=C_0([0,1],\R^d)$ be the set of all continuous paths in $\R^d$ indexed by $[0,1]$ and starting at zero. Let $\mathcal{B}$ be the $\sigma$-algebra generated by all maps $\omega\mapsto \omega_t$, $t\geq 0$, from $X$ to $\R^d$ and let $\nu_0$ denote the Wiener measure on $(X,\mathcal{B})$.

First, we ``approximate" $J(\omega)$ in the following way: 
For $\e\in (0,1)$, $a\in (0,1)$, $0\leq s\leq t\leq 1$, $0\leq u\leq v\leq 1$, and $\omega\in X$, let \begin{align}
J_{s,t;u,v}^{\e,a}:=\int_s^t\dd \sigma \int_{(u\vee (\sigma+\e))\wedge v}^v \dd \tau p_a(\omega_\sigma-\omega_\tau)\label{eq:Jestuvdfn}
\end{align}
where
\[
p_a(x) := \frac{1}{(2\pi a)^{d/2}}\exp \left( -\frac{|x|^2}{2a}\right) \quad \text{for $a>0$, $x\in \R^d$}.
\]
Taking the limit of $J_{s,t:u,v}^{\e,a}$ as $a \searrow 0$ formally,  
$J^{\e,0}_{s,t;u,v}$ describes the ``self-intersection local time" of Brownian paths $\omega[s,t]$ and $\omega[u,v]$ with removal of times near the diagonal 
by parameter $\e>0$. The probability measure in \eqref{eq:polmea} is defined by
\begin{align}
\nu_{\lambda}(\dd \omega)=\lim_{\e\searrow 0}\frac{1}{N_\lambda^\e}\exp\left(-\lambda J_{0,1;0,1}^{\e ,0}\right)\nu_0(\dd\omega) \label{eq:nulambda}
\end{align}
if the limit exists.
For $\lambda>0$, the realization of paths under $\nu_\lambda$ likes to avoid self-intersection other than those already present under $\nu_0$.
The difficulty to define $\nu_\lambda$  rigorously  depends on the dimension of the space.

For the case $d=1$, Westwater discussed the $L^1$-convergence of $J_{0,1;0,1}^{0,a}$ as $a\searrow 0$  and  the limit is  given by
\begin{align}
\frac{1}{2}\int_\R \ell_1(x)^2\dd x\label{eq:d=1SelfInLocal}
\end{align}
where $\{\ell_1(x)\}_{x\in\R}$ is the local time of Brownian motion $\omega$ (almost surely defined) up to time $1$, which are continuous in space and hence \eqref{eq:d=1SelfInLocal} makes sense (see \cite{Wes80}). Furthermore, \eqref{eq:polmea} does make sense. This explicit expression allows us to analyze the properties of the long-time behavior of $\omega$  under the polymer measures (see \cite{Kus85b,Wes85}). Westwater has shown that for $T \rightarrow \infty$, the law of $\frac{\omega_T}{T}$ converges under $\nu_{\lambda,T}:=\frac{1}{Z_{\lambda,T}}\exp\left(-\lambda J_{0,T;0,T}^{0,0}\right)\nu_0$ to $\frac{1}{2}(\delta _{\lambda^* \lambda ^{1/3}} + \delta _{-\lambda ^* \lambda ^{1/3}} )$, where $J_{0,T;0,T}^{0,0}$ is defined similarly to $J_{0,1;0,1}^{0,0}$ by replacing $[0,1]$ by $[0,T]$ and   $\lambda^* \sim 1.1$ is a constant. This result is improved by van der Hofstad, den Hollander, K\"onig in \cite{vdH98,HHK97,HHK03a,HHK03b} and Najnudel in \cite{Naj10}.
 
In the case that $d=2$, it is known that $J_{0,1;0,1}^{0,a}$ diverges almost surely as $a\searrow 0$ (see \cite{Var69,LeG85,Yor85}). But, there exists a random variable $Y$ such that $J_{0,1;0,1}^{0,a}-E[J_{0,1;0,1}^{0,a}]$ converges to $ Y$ in $L^2(\nu_0)$ as $a\searrow 0$ (Varadhan's renormalization) and $E[e^{-\lambda  Y}]<\infty$ for all $\lambda \geq 0$ (see \cite{Var69,LeG85,Ros86a}). Therefore, \eqref{eq:polmea} is  rigorously defined by $\nu_\lambda(\dd \omega)=\frac{1}{E[e^{-\lambda  Y}]}e^{-\lambda Y}\nu_0(\dd \omega)$.
Moreover, there exists $g_0>0$ such that $E[e^{gY}]<\infty$ for $g\in (-\infty,g_0)$ (see \cite{Sto89,LeG94}).
Then, Bass and Chen showed that the best constant $g_0$ coincides with $A^{-4}$, where
$A$ is the best constant in the Gagliardo-Nirenberg inequality \cite{BC04}.  
Here, we remark that for $\lambda\in (-g_0,\infty)$, $\nu_\lambda$ is absolutely continuous with respect to $\nu_0$. The $k$-multiple intersection local times for the planar Brownian motion are also discussed in \cite{Dyn88b, Dyn88, Ros90}. 

The case $d=3$ is more complicated, because Varadhan's renormalization does not work, since the variance of $J_{0,1;0,1}^{\e,a}$ diverges as $a\searrow 0$ and $\e\searrow 0 $ (see \eqref{eq:kappa2-2} below). However, Westwater constructed the polymer measure $\nu_\lambda$ in \cite{Wes80} and then Bolthausen gave a simpler construction in \cite{Bol93} for small values of the coupling constant $\lambda\geq 0$. In \cite{Wes82} it is shown that the approach in \cite{Wes80} is still suitable for all positive coupling constants, and in \cite{AZ98} it is proved that the polymer measure given by Bolthausen is actually the same as the one given by Westwater. It is also proved that the three-dimensional polymer measure is singular with respect to the Wiener measure (see \cite[Theorem 3]{Wes82}). More precisely, $\nu_{\lambda_1}$ and $\nu_{\lambda_2}$ are mutually singular for $0\leq \lambda_1<\lambda_2$. We remark that Bovier, Felder, and Fr\"ohlich proved the tightness of approximated polymer measures in \cite{BFF84}.
Also, Gubinelli and L\"orinczi discussed the construction of $\nu_\lambda$ (for $d=1$) with the general theory of Gibbs measures via rough path theory in \cite{GL09}.

The self-intersection local times are interesting objects by themselves,
and many studies are dedicated to them. For example, the Wiener-Chaos 
decomposition of regularizations of the self-intersection local times
are studied in \cite{BOS16, dFDS00, Hu96, IPV95}. On the other hand, nondifferentiability of the self-intersection local times in the sense of Meyer-Watanabe
is discussed in \cite{AHZ97,Hu96} in the case that $d=2$. The
self-intersection local times are generalized formally by letting
$J_x(\omega):=\int_0^1\int_0^t \delta _0(\omega_t-\omega_s-x)\dd s\dd
t$ for $x\in\R^d$. Then, it is known that for $x\not =0$,
$J_x(\omega)$ is well-defined as a function for $d=1,2,3$ and as a
generalized function in Sobolev space with negative order for $d\geq
4$ (see \cite{Ros83,LeG85,Yor85,IPV95,Ros05}).
We remark that for $d\geq 4 $, the self-intersection local time
$J_0(\omega)$ of Brownian motion can be defined as a generalized
functional of Brownian motion in the space of Hida distributions (see
\cite{Wat91,HYYW95}) with some suitable renormalization.

Now, we give a rigorous definition of $\nu_\lambda$ as in \cite{Bol93,AZ98}: For the rest of this paper, we fix $\lambda\in [0,\infty)$.
We recall the following result by Bolthausen.

\begin{proposition}\label{prop1}\cite[Proposition (2.1)]{Bol93}
There exists a version of $J_{s,t;u,v}^{\e,a}$ in $(X,\mathcal{B},\nu_0)$ which is jointly continuous in all the variables $0\leq s\leq t\leq 1$ and $0\leq u\leq v\leq 1$, and parameters $\e\in (0,1)$ and $a\in (0,1)$.
Moreover, $\dis\lim_{a\searrow 0}J_{s,t;u,v}^{\e,a}=:J_{s,t;u,v}^{\e,0}$ exists $\nu_0$-a.s.~and also it is jointly continuous in all variables $0\leq s\leq t\leq 1$ and $0\leq u\leq v\leq 1$ and the parameter $\e\in (0,1)$.
\end{proposition}

\begin{rem}\label{rem:Je}
We can prove in a similar way to \cite[Proposition (2.1)]{Bol93} that for each $\e>0$, $0\leq s\leq t\leq 1$, $0\leq u\leq v\leq 1$, and for any $p\geq 1$, $J_{s,t;u,v}^{\e,a}$ converges in $L^p(\nu_0)$ as $a\searrow 0$, which is not mentioned in the paper \cite{Bol93} explicitly (see Lemma \ref{lem:Jacauchy} and Corollary \ref{cor:JaJ}).
\end{rem}

From now on, we regard $J_{s,t;u,v}^{\e,a}$ as the version given in Proposition \ref{prop1}. In particular, we write $J_{s,t;u,v}^{\e}:=J_{s,t;u,v}^{\e,0}$ and we write \begin{align*}
J_{s,t;u,v}^{\e,0}=\int_s^t\dd \sigma \int_{(u\vee (\sigma+\e))\wedge v}^v \dd \tau \delta_0(\omega_\sigma-\omega_\tau)
\end{align*} instead of   equation \eqref{eq:Jestuvdfn}. Also, we denote $J_{0,1;0,1}^{\e,a}$ by $J_{0,1}^{\e,a}$ for $\e>0$ and $a\geq 0$ (see also Corollary \ref{cor:JaJ}).

For $\e\in (0,1)$, we set
\begin{align}
 \kappa_1(\e)=\int_\e^1 p_t(0)\dd t=2(2\pi)^{-\frac{3}{2}}(\e^{-\frac{1}{2}}-1)\label{eq:kappa1a}
 \end{align} and 
 \begin{align*}
\kappa_2(\e) &= (2\pi)^{-3}\int_0^1\dd s_1{\int_{s_1}^{1}\dd s_2\int_{s_2}^1\dd s_31_{\{\e\leq s_2\}}1_{\{\e\leq s_3-s_1\}}}\\
&\quad \hspace{4cm} \times (s_1(s_2-s_1)+s_1(s_3-s_2)+(s_3-s_2)(s_2-s_1))^{-\frac{3}{2}}.
\end{align*}
In Section \ref{kappa2}, we shall prove that there exist constants $K_1,K_2,K_3\in \R$ such that 
\begin{align}
&\lim_{\e\searrow 0}\left(E[J_{0,1}^\e]-\kappa_1(\e)\right)=K_1,\label{eq:kappa1}\\
&\lim_{\e\searrow 0}\left(\mathrm{Var}\left(J_{0,1}^\e\right)+\frac{2}{(2\pi)^2}\log \e\right)=K_2,\label{eq:kappa2-2}\\
&\lim_{\e\searrow 0}\left(\kappa_2(\e)+\frac{1}{(2\pi)^2}\log \e\right)=K_3.\label{eq:kappa2}
\end{align}
where we denote $J^{\e}_{0,1}$ by $J^{\e,0}_{0,1}$, and $E$ (resp.~$\mathrm{Var}$) is the expectation (resp.~variance) with respect to $\nu_0$.
These imply that $\kappa_1(\e)$ and $2\kappa_2(\e)$ are divergent parts of the expectation and variance of $J_{0,1}^\e$, respectively.

Let \begin{align}
&\bar{J}_{s,t}^{\e,\lambda}:=\lambda J_{s,t}^\e-\lambda(t-s)\kappa_1(\e)+\lambda^2(t-s)\kappa_2(\e),\notag
\intertext{and}
&\nu_{\e,\lambda}:=E\left[\exp\left(-\bar{J}_{0,1}^{\e,\lambda}\right)\right]^{-1}\exp\left(-\bar{J}_{0,1}^{\e,\lambda}\right)\nu_0.\quad \label{eq:nuelambdadfn}
\end{align}
 It was shown in \cite{Bol93} and \cite{AZ98} that the limit $\dis \lim_{\e\searrow 0}\nu_{\e,\lambda}$ exists in the weak sense on the Wiener space $(X,\mathcal{B})$. This limit is denoted by $\nu_\lambda$, which is a rigorous version of \eqref{eq:polmea}. We will recall Bolthausen's construction in Section \ref{sec:3}. See also Section \ref{sub:strategy}. The reader may refer to the lecture notes \cite[Chapter 1]{Bol02} treating Edwards' model.

As mentioned above, $\nu_\lambda$ is singular with respect to $\nu_0$ for $\lambda>0$ in the three-dimensional case, and hence it is hard to analyze path properties under $\nu_\lambda$. In \cite{Kus85a}, Shigeo Kusuoka proved the $\{\omega_{t}\}_{t\in[0,1]}$ under $\nu_\lambda$ is a Dirichlet process in the sense of F\"ollmer \cite{Fol81}. In \cite{Zho91,Zho92,Zho92b,Zho92c,Zho96}, Xianyin Zhou discussed the self-intersection local times, the Hausdorff dimensions of the set of double points, and other properties of the paths.

Also, the self-intersection local time and the associated polymer measures have been studied for $d$-dimensional fractional Brownian motions with Hurst parameter $H\in (0,1)$ \cite{BFS17,BFS21,BOS16,GOSS11,Hu01,HN05,HNS08,Ros87}. In \cite{HN05}, it is proved that if $dH<1$, $J_{0,1;0,1}^{0,a}$ converges in $L^2$-sense; and if $dH\in \left[1,\frac{3}{2}\right)$,  Varadhan's renormalization does also work for $J^{0,a}_{0,1;0,1}$. In \cite{GOSS11, BFS21}, a construction of the polymer measures for $dH\leq 1$ has been given by analyzing the corresponding self-intersection local times.

We can also consider a similar model in a discrete setting by replacing the Brownian motion by a simple random walk $\{S_n\}_{n=0}^\infty$, that is, the self-intersection local time is defined by $L_N(S)=\sum_{0\leq i<j\leq N}1_{\{S_i=S_j\}}$ and define Gibbs measures by \begin{align*}
P_N^g(\dd S)=\frac{1}{Z_N^g}\exp\left(-g L_N(S)\right)P(\dd S),\quad \text{for }g\geq 0,
\end{align*}  
where $P$ is the law of simple random walk $S$ and $Z_N^g$ is a normalizing constant.
This model is called the \textit{weakly self-avoiding walk} or the \textit{Domb-Joyce model}, and has been well-studied (see  \cite{denH09,Law13,MS13}).
Unlike Brownian motion, the self-intersection local time with respect to the simple random walk can be positive.
Hence, $L_N$ and $P_N^g$ do make sense. As discussed in \cite{Law13}, the interaction in Edwards' model is weaker than the one in the weakly-self avoiding walk. To approximate Edwards' model for $d=2,3$, it is suitable to renormalize $g$ as $g_N=2N^{\frac{d-2}{2}}$ with some time-space rescaling. Actually, Stoll in \cite{Sto89} constructed Edwards' model for $d=2$ via this rescaling and the use of Robinson's nonstandard analysis (see e.g. \cite{AFHKL86}).
In \cite{ABZ94}, the same method was discussed for $d=3$.
For $d=4$, it is known that paths of the simple random walk have infinitely many self-intersection points, thus behave in a different way with respect to Brownian motion (see \cite{Law13}).
The limit of the rescaled measures $P_{N}^{g_N}$ may in principle under certain restriction on $\lambda$ yield the polymer measure for $d=4$. In \cite{AZ94}, the question of the asymptotics of the probabilities of the event that two independent simple random walks in $\Z^4$ have no intersection points up to time $N$ is discussed.
Symanzik, Nelson and Aizenman pointed out that this quantity is related to the $\phi^4_4$ theory (see \cite{Sym69,Nel83, Aiz85, ADC21}). In the view point of self-intersection local times, the first named author and Zhou in \cite{AZ95} proved the central limit theorem for the renormalized self-intersection local times $L_N(S)$ for $d\geq 3$ as discussed in \cite{Yor85}.

\subsection{Stochastic quantization}\label{subsec:1.2}

Parisi and Wu introduced the idea for stochastic quantizations in the literature of quantum fields theory. For a given probability measure $\nu$ on a measurable space $\mathcal{X}$, the stochastic quantization of $\nu$ means a construction of a Markov process that has an invariant distribution $\nu$ (see \cite{PW81}).
In the case of finite-dimensions Kolmogorov had studied such a problem via Fokker-Plank equations many years before (see \cite{Kol37}).
In the present paper we discuss stochastic quantization of the three-dimensional polymer measure, which is a much more difficult case, because the measure is on the (infinite-dimensional) path space and is defined via renormalization.

 In the case of $\phi ^4$ Euclidean quantum field theory, the stochastic quantization has been a very hot topic starting many years ago and it remains so also now.
 In \cite{JLM85} stochastic quantization for dimension $2$ given by stochastic partial differential equations (SPDE) with regularized noise is studied.  Stochastic quantization of the two-dimensional $\phi ^4$ theory is
 studied by infinite-dimensional Dirichlet forms in \cite{AR91} (see also \cite{BCM88}) and also by SPDE approach in \cite{DPD03}.
 On the other hand, stochastic quantization for quantum field models in the three-dimensional case has been remained as a difficult problem for a long time.
 However, recently the theories of regularity structures introduced by Hairer in \cite{Hai14} and paracontrolled calculus introduced by
 Gubinelli, Imkeller and Perkowski in \cite{GIP15} appeared and these methods enable us  to solve nonlinear SPDEs with singular noise via renormalization.
 These methods work very well for stochastic quantization of quantum field theories, and are now developing remarkably.
 Stochastic quantization of the three-dimensional $\phi ^4$ theory via SPDEs was first achieved locally in time on torus by Hairer via regularity
 structures (see \cite[Section 1.5.2]{Hai14}), also via paracontrolled calculus locally in time on torus (see \cite{CC18a}), then globally in
 time on torus (see \cite{MW17}), and globally in time on the whole space (see \cite{GH19}).
 Construction of the three-dimensional $\phi ^4$-quantum field measure also has succeeded by the methods of singular SPDEs, e.g., on
 torus (see \cite{AK20}), on the whole space by discrete approximation \cite{GH21}, and on the whole space by continuous approximation (see
 \cite{AK25}).
 Besides, stimulated by the development of the methods of singular SPDEs, many mathematicians study similar but different approaches,
 e.g, by renormalization group (see \cite{Kup16}), by a variational method (see \cite{BG20}), by elliptic stochastic quantization (see
 \cite{GH19, ADVG21}), and by scaling argument (see \cite{Duc25}).
 For the details of the history of the stochastic quantization of $\phi ^4$ theory, see the introductions of \cite{AK20, AK25, GH21} and
 references therein. Here, we recall that, as mentioned in above, in \cite{Sym69} Symanzik
 suggested a close relation between the three-dimensional polymer measure and the three-dimensional $\phi ^4$ theory by formal
 calculations, and indeed many similarities of the models are known, e.g., the renormalization constants which appeared in both models
 have the same asymptotics, both measures are singular with respect to the base measures (the base measure is the free field measure in the case of the
 $\phi ^4$ theory, and the Wiener measure in the case of the polymer measure).
 The renormalization constants of the one-dimensional KPZ equation  also have the same asymptotics of those of the three-dimensional
 polymer measure. See \cite{Hai13} and \cite{GP17} for approaches from singular SPDEs.
 The relation between the one-dimensional KPZ equation and the $(1+1)$-dimensional directed polymers in random environment (DPRE) had already been
 known in the paper \cite{KPZ86}, which first introduced the KPZ equation. Recently, DPRE has focused on the connection between the KPZ equation 
 and the stochastic Burgers equation (see \cite{AZ96b, BC20, BC22, BL18, BL19, Bol89, Com17, DS88, FJ22, FJ23, Gia07, GP17, IS88, NN23, SZ96}).
 Furthermore, stochastic quantization techniques for random measures on a path space have been applied to other problems (see \cite{BC22}).
 Besides, the stochastic quantization of exponential interaction (H{\o}egh-Krohn model) has also been studied both from Dirichlet
 form theory and singular SPDEs (see \cite{AKMR23, HKK21, HKK23} and references therein).

Our main result in this paper is about the stochastic quantization of the three-dimensional polymer measure via the Dirichlet form method.

In \cite{AHRZ99}, the stochastic quantization of the two-dimensional
polymer measure was constructed using the theory of Dirichlet forms on
infinite-dimensional state spaces. In the present paper the same is done in
the case of the three-dimensional polymer measure, which is technically much
more difficult, since in three-dimensions the polymer measure is not absolutely continuous with respect to the Wiener measure (as opposed to the lower-dimensional cases after Varadhan's renormalization for $d=2$).
In the unpublished work \cite{ARZ96} a construction of the Dirichlet form associated with the three-dimensional polymer measure was attempted.
However, the proofs were not complete or contained gaps (e.g. \cite[Proposition 2.1]{ARZ96}).
This work in 1996 remained unfinished due to the sudden tragic departure of Xianyin Zhou in his early age.
The present paper is largely based on the previous work \cite{ARZ96}, that in turn was possible through deep in sights accumulated by X.Y.Zhou in his deep work (see also e.g. \cite{Zho91, Zho92c, Zho92b, Zho92, Zho96}).
In the present paper we present a self-contained full proof of the construction of the Dirichlet form associated with the three-dimensional polymer measure, which guarantees that the statement of the main theorem of the preprint \cite{ARZ96} is correct.

Let us remind that the construction of the polymer measure $\nu _\lambda$ in the cases that $d\leq 2$ are somewhat similar in difficulty with the construction of the Euclidean measures of models of quantum field theory in dimension $d\leq 2$, see e.g. the references  given in the introductions of \cite{AK20, AK25, GH21}.
For the construction of Dirichlet forms associated with $\phi^4_2$-measure see \cite{AR91, AMR15, RZZ17, AKMR23, HKK21, HKK23, BG23}.
For the case $d=3$, the construction of the polymer measure and associated Dirichlet form is comparable in difficulty with the construction of the $\phi ^4_3$-measure and and associated Dirichlet form (see \cite{ZZ18, BG23}).
A direct construction of a stochastic process with the $\phi ^4_3$-measure as an invariant measure has been provided in various recent works \cite{Hai13, GH19, AK20, ZZ18} and the references in \cite{AK25}.
See below after Remark \ref{rem:thm2} for further comments on this.

We formulate the main results of the present paper in the next subsection (Theorems \ref{thm:thm1}, \ref{thm:thm2} and \ref{thm:thm3}).
Let us note that the irreducibility of the Dirichlet form associated to $\nu _\lambda$ also remains an open problem (whereas it has been solved for the case that $d=2$ in \cite{AHRZ99}, see also \cite{AKR97,AKR97b} for relations between irreducibility and ergodicity of associated process).

\begin{rem}
When one replaces Brownian motion by a fractional Brownian motions, similar constructions of associated polymer measures and stochastic quantization have been discussed in \cite{BFS17}. In \cite{Hu01,HNS08}, it is proved that the self-intersection local times are Meyer-Watanabe differentiable for $dH<1$ where $H$ is the Hurst index of fractional white noise. In \cite{BFS17}, the authors used \cite{Pot00} and Meyer-Watanabe differentiability to obtain a diffusion process whose invariant measure is the fractional Edwards' model for $d$-dimensional fractional Brownian motion. The diffusion is constructed via Dirichlet form techniques in infinite-dimensional analysis, where the closability was shown by an integration by parts formula, and irreducibility of the constructed diffusion follows as in the case of planar Brownian motion (discussed in \cite{AHRZ99}). In  \cite{BFS18}, the infinite dimensional SDE associated with the Dirichlet form constructed in \cite{BFS17} is discussed. We will give a formal infinite dimensional SDE associated with our Dirichlet form $(\mathcal{E}_{\nu_\lambda},D(\mathcal{E}_{\nu_\lambda}))$ below in Remark \ref{rem:infiniteSDE}.
\end{rem}

\subsection{Main results}\label{sec:Main}
 
Recall that $X:=C_0([0,1],\R^d)$ and $X$ equipped with the supremum-norm is a separable real Banach space.
Let $X'$ be its dual. Let $H\subset X$ be the classical Cameron-Martin space, i.e., $H:=\{h\in X:\text{$h$ is absolutely continuous and $|h|_H^2=\int_0^1|h'(t)|^2\dd t<\infty$}\}$.
Furthermore, let \begin{align}
K:=\left\{h\in H: \sup_{0\leq t\leq 1}|h'(t)|<\infty\right\}.\label{eq:Kdfn}
\end{align}
Then, $(H,\langle \cdot , \cdot \rangle_{H})$ is a real separable Hilbert space which is densely and continuously embedded into $X$, and $K$ is a linear subspace of $(H,\langle \cdot , \cdot \rangle_H)$.
By identifying $H$ with its dual we obtain that $X'$ is densely embedded into $H$.
Hence, $X'\subset H\subset X$, and $\langle \cdot , \cdot \rangle_H$ restricted to $X'\times H$ coincides with the dualization between $X'$ and $X$.
Define the set  of bounded smooth cylinder functions by
\begin{align*}
\mathcal{F}C_b^\infty :=\{f(l_1,\cdots,l_m):m\in \N,f\in C_b^\infty(\R^m),l_1,\cdots,l_m\in X'\},
\end{align*}
where $C_b^\infty(\R^m)$ denotes the set of all bounded infinitely differentiable functions on $\R^m$ with all bounded continuous partial derivatives.

  Let $h\in X$. For $u\in \mathcal{F}C_b^\infty$, $\omega\in X$, define \begin{align*}
 \frac{\partial u}{\partial h}(\omega):=\left.\frac{\dd }{\dd s}u(\omega+sh)\right|_{s=0}
 \end{align*} 
 and let $\nabla u(\omega)$ denote the unique element in $H$ such that \begin{align*}
 \langle \nabla u(\omega), h\rangle_H=\frac{\partial u}{\partial h}(\omega)\quad \text{for all }h\in H.
 \end{align*}
Define for $u,v\in \mathcal{F}C_b^\infty$\begin{align}
\cE_{\nu_\lambda}(u,v):=\int \langle \nabla u(\omega),\nabla v(\omega)\rangle_H\nu_\lambda(\dd \omega)\label{eq:defcE}
\end{align}
and $\cE_{\nu_\lambda,1}:=\cE_{\nu_\lambda}+(\cdot , \cdot )_{L^2(X;\nu_\lambda)}$.

Now, we can formulate our first result.

\begin{theorem}\label{thm:thm1}\begin{enumerate}[label=(\roman*)]
\item\label{item:1-4-1} The support of $\nu _\lambda$, $\mathrm{supp}\nu_\lambda$, coincides with $X$ (i.e.~$\nu_\lambda(U)>0$ for all non-empty open $U\subset X$). In particular, we can identify each $u\in \mathcal{F}C_b^\infty$ with the corresponding class in $L^2(X;\nu_\lambda)$, and thus $(\cE_{\nu_\lambda},\mathcal{F}C_b^\infty)$ is a symmetric non-negative definite bilinear form on $L^2(X;\nu_\lambda)$.
\item\label{item:1-4-2} The symmetric bilinear form $(\cE_{\nu_\lambda},\mathcal{F}C_b^\infty)$ is closable on $L^2(X;\nu_\lambda)$ and the closure $(\cE_{\nu_\lambda},D(\cE_{\nu_\lambda}))$ is  a symmetric Dirichlet form, i.e., a closed non-negative definite symmetric bilinear form such that $u^\#:=(u\vee 0)\wedge 1\in D(\cE_{\nu_\lambda})$ and $\cE_{\nu_\lambda}(u^\#,v^\#)\leq \cE_{\nu_\lambda}(u,v) $ for all $u\in D(\cE_{\nu_\lambda})$.
\end{enumerate}
\end{theorem}

The closability of $(\cE_{\nu_\lambda},\mathcal{F}C_b^\infty)$ on $L^2(X;\nu_\lambda)$ means that the unique continuous extension $\bar{\iota}$ of the continuous embedding $\iota : \mathcal{F}C_b^\infty \hookrightarrow L^2(X;\nu_\lambda)$, where $\mathcal{F}C_b^\infty$ is equipped with norm $(\cE_{\nu_\lambda,1})^{\frac{1}{2}}$ and $L^2(X;\nu_\lambda)$ with $\| \cdot \| _{L^2(X;\nu_\lambda)}$, to the completion $\overline{\cF C_b^\infty}$ of $\cF C_b^\infty $ with respect to $(\cE_{\nu_\lambda,1})^{\frac{1}{2}}$ is still continuous embedding.
The \textit{closure} $(\cE_{\nu_\lambda},D(\cE_{\nu_\lambda}))$ is then the smallest closed extension of $(\cE_{\nu_\lambda},\cF C_b^\infty)$ on $L^2(X;\nu_\lambda)$. For more details on Dirichlet forms the reader is referred, for example,  to \cite{MR92,FOT11,BH91,Sil74}, and for the special type of Dirichlet form appearing here, namely \textit{classical infinite dimensional Dirichlet forms} to \cite{AR90a}.

Let $(L,D(L))$  be the generator of $(\cE_{\nu_\lambda},D(\cE_{\nu_\lambda}))$, i.e., the unique non-positive definite self-adjoint operator on $L^2(X;\nu_\lambda)$ such that \begin{align}
D(\sqrt{-L})=D(\cE_{\nu_\lambda})\quad \text{and}\quad (\sqrt{-L}u,\sqrt{-L}v)=\cE_{\nu_\lambda}(u,v)\quad \text{for all }u,v\in D(\cE_{\nu_\lambda}).
\end{align} 
Let $T_t:=e^{tL},$ $t\geq 0$.
Then, the general theory of Dirichlet form yields the following theorem.
\begin{theorem}\label{thm:thm2}
There exists a diffusion process $\mathbb{M}=(\Omega,\cF,(\cF_t)_{t\geq 0},(X_t)_{t\geq0},(P_\omega)_{\omega\in X})$ which is associated with $(\cE_{\nu_\lambda},D(\cE_{\nu_\lambda}))$, i.e., for all ($\nu_\lambda$-versions of ) $f\in L^2(X;\nu_\lambda)$ and all $t\geq 0$ the function
\begin{align*}
\omega\mapsto p_tf(\omega):=\int_\Omega f(X_t)\dd P_\omega,\quad \omega\in X,
\end{align*}
is a $\nu_\lambda$-version of $T_tf$. Moreover, $\mathbb{M}$ is conservative and $\nu_\lambda$-symmetric. In particular, $\nu_\lambda$ is an invariant measure for $\mathbb{M}$.
\end{theorem}

\begin{rem}\label{rem:thm2}
\begin{enumerate}[label=(\roman*)]
\item $\mathbb{M}$ is in fact even properly associated to $(\cE_{\nu_\lambda},D(\cE_{\nu_\lambda}))$ in the sense of \cite[Chap.IV, Definition 2.5]{MR92}.
\item $\Omega$ in Theorem \ref{thm:thm2} can, of course, always be chosen to be $C([0,\infty),X)$, since $\mathbb{M}$ has continuous sample paths and is conservative. 
\end{enumerate}
\end{rem}

As already emphasized in \cite{AHRZ99} the construction of the stochastic quantization of $\nu_\lambda$ by the theory of Dirichlet forms uses $\nu_\lambda$ and is based on some of its properties. It thus does not lead to an independent \textit{construction of $\nu_\lambda$} itself. This is similar to early works on the stochastic quantization of measures in Euclidean field theory (such as e.g. $P(\phi)_2$-fields, see for example \cite{JLM85,AR89a,AR89b,AR91, AMR15, ABR22}) where the respective measures have been used in an essential way.
A truly constructive approach of the invariant measures (more precisely by the stochastic quantization equations for Euclidean quantum fields) was not known (for dimension $d=2,3$) before the work initiated in \cite{DPD03} for $d=2$, and \cite{Hai13, GIP15, AK20, GH21, AK25} for $d=3$ .
For polynomial and exponential type models in $d=2$ even uniqueness results are known (see e.g. \cite{RZZ17, BDV21, DVGT22}).
For $d=3$ $\phi ^4_3$-models results are less complete (see \cite{ALZ06, ZZ18, BG20, AK20, AK25}).
To the best of our knowledge, the problem of a constructive approach of $\nu _\lambda$ for polymer measures in $d=2,3$ does not seem to have been discussed from this point of view, i.e. via solutions of SPDEs. 

For the proof of Theorem \ref{thm:thm1}, we need the following result, which is of its own interest and is in fact the heart of this paper. For $h\in X$, we define $\tau_h(\omega):=\omega+h$, $\omega\in X$, and let
\begin{align}
K_0:=\left\{h\in K\left|\sup_{0\leq t\leq 1}|h''(t)|<\infty\right.\right\},  \quad  \label{eq:Kodfn}
\end{align}
where $K$ is given in \eqref{eq:Kdfn}.

\begin{theorem}\label{thm:thm3}
Let $h\in K_0$. Then, $\nu_\lambda$ is $k$-quasi-invariant, i.e., $\nu_\lambda\circ \tau_{sh}^{-1}$ is absolutely continuous with respect  to $\nu_\lambda$ for all $s\in\R$. If $a_{sh}:=\dis \frac{\dd (\nu_{\lambda}\circ \tau_{sh}^{-1})}{\dd \nu_{\lambda}}$, $s\in \R$, then the process $(a_{sh})_{s\in\R}$ has a version which has $\nu_\lambda$-almost every continuous sample paths.
\end{theorem}

The proof of Theorem \ref{thm:thm3} is given in Section \ref{sec:5} below after the necessary preparations in Section \ref{sec:2}-\ref{sec:4}.

Once we prove Theorem \ref{thm:thm3}, Theorem \ref{thm:thm1} and Theorem \ref{thm:thm2} follow from the theory of the Dirichlet forms as in \cite[Theorem 1.2, Theorem 1.3]{AHRZ99}. 
\begin{proof}[Proof of Theorem \ref{thm:thm1}]
\ref{item:1-4-1} By Theorem \ref{thm:thm3}, $\nu_\lambda$ is $k$-quasi-invariant for each $k\in K_0$. Since  $K_0$  is a linear space and dense in $H$, hence also in $X$, it follows by  \cite[Proposition 2.7]{AR90b} that $\mathrm{supp}\nu_\lambda=X$.

\ref{item:1-4-2} Since $K_0$ is dense in $H$, we can find an orthogonal basis $\{h_n:n\in \N\}$ of $H$ in $K_0$. Then, we can find \begin{align*}
\sum_{n\geq 1}{}_{X'}\langle l,h_n\rangle_{X}^2<\infty\quad \text{for all }l\in X'.
\end{align*}
Suppose that for each $n\in \N$, $(\mathcal{E}_{h_n},\mathcal{F}C_b^\infty)$ on $L^2(X;\nu_\lambda)$ is closable, where  \begin{align*}
\mathcal{E}_{h_n}(u,v)=\int \frac{\partial u}{\partial h_n}\frac{\partial v}{\partial h_n }\dd \nu_\lambda,\quad u,v\in \mathcal{F}C_b^\infty.
\end{align*} 
Then, we get from \cite[Remark 3.7 in Chapter II]{MR92} that 
for $u,v\in \mathcal{F}C_b^\infty$,
\begin{align}
\mathcal{E}_{\nu_\lambda}(u,v)=\sum_{n=1}^\infty \int \frac{\partial u}{\partial h_n}\frac{\partial v}{\partial h_n }\dd \nu_\lambda,			\label{eq:Dirichletnug}
\end{align}
and from \cite[Theorem 3.8]{AR90a} or \cite[Proposition 3.5 in Chapter II ]{MR92} that $(\mathcal{E}_{\nu_\lambda},\mathcal{F}C_b^\infty)$ is closable on $L^2(\nu_\lambda)$ and its closure $(\mathcal{E}_{\nu_\lambda},D(\mathcal{E}_{\nu_\lambda}))$  is a symmetric Dirichlet form. 
Thus, it suffices to prove the closability of each $(\mathcal{E}_{h_n},\mathcal{F}C_b^\infty)$ on $L^2(X;\nu_\lambda)$. 

Since $s\mapsto a_{sh_n}$ is continuous  $\nu_\lambda$-almost everywhere by Theorem \ref{thm:thm3} for each $n\in \N$, $h_n$ is admissible \cite[Corollary 2.5]{AR89b} in the sense \cite[Definition 1.6]{AR89b}.    Then, the closability follows from \cite[Theorem 1.3]{AR89b}.
\end{proof}

The proof of Theorem \ref{thm:thm2} then follows by the standard machinery of Dirichlet forms on infinite-dimensional state spaces and is exactly the same as that of \cite[Theorem 1.2]{AHRZ99}. 

\begin{proof}[Proof of Theorem \ref{thm:thm2}]
Since we have verified the closability of $(\mathcal{E}_{\nu_\lambda},\mathcal{F}C_b^\infty)$ on $L^2(X;\nu_\lambda)$, it follows from \cite[Subsection 4 b) in Chapter IV]{MR92}   that  $(\cE_{\nu_\lambda},D(\cE_{\nu_\lambda}))$ is quasi-regular. Then, the existence of a $\nu_{\lambda}$-symmetric  $\mathbb{M}$ is a consequence of \cite[Theorem 3.5 in Chapter IV]{MR92}.

Also, by \cite[Example 1.12 (ii) in Chapter V]{MR92}, the quasi-regular Dirichlet form  $(\cE_{\nu_\lambda},D(\cE_{\nu_\lambda}))$ possesses the local property, and hence  sample paths of $\mathbb{M}$ are continuous by \cite[Theorem 1.11 in Chapter V]{MR92}. The conservativity of $\mathbb{M}$ is obvious since $1\in D(L)$ 
and $L1= 0$, hence $T_t1=1$ for all $ t\geq  0$. The fact that $\nu_\lambda$ is an invariant measure for $\mathbb{M}$ then follows immediately.
\end{proof}

 The rest of this paper  is devoted to proving Theorem \ref{thm:thm3}.

\begin{rem}
In contrast to the two-dimensional case we have not yet succeeded in proving the irreducibility (which is equivalent to ergodicity) of $(\mathcal{E}_{\nu_\lambda},D(\mathcal{E}_{\nu_\lambda}))$ in Theorem \ref{thm:thm1}.
\end{rem}

\begin{rem}\label{rem:infiniteSDE}

In \cite{AR91} a general theory of the infinite-dimensional SDE associated with a classical Dirichlet form is given.
At present, it is unknown whether this theory is applicable to the three-dimensional polymer measure $\nu _\lambda$ or not. 
Here we formally apply the general theory to the Dirichlet form $(\mathcal{E}_{\nu_\lambda},D(\mathcal{E}_{\nu_\lambda}))$ and the diffusion $\mathbb{M}=(\Omega,\cF,(\cF_t)_{t\geq 0},(X_t)_{t\geq0},(P_\omega)_{\omega\in X})$ for a heuristic observation and we clarify the problems.

For each $k\in K_0$, let $b_k$ be a functional on the path space $X$ which satisfies the integration by parts formula of $\nu_\lambda$ is formally written by
\begin{align*}
\int_X \frac{\partial u}{\partial k}\dd \nu_\lambda=-\int_X u b_k \dd\nu_\lambda
\end{align*}
for all cylindrical polynomial function $u$.
The functional $b_k$ is regarded as the log-derivative of $\nu_\lambda$ with respect to the direction $k$, and the existence of $b_k$ is not known.
As in \cite[Theorems~2.8 and 5.3]{AR91}
\begin{align}
{}_{X'}\langle k,X_t \rangle_{X}-{}_{X'}\langle k,X_0\rangle_{X}= \sqrt{2}W_t^k+\int_0^t b_k(X_s)\dd s \label{eq:Fdecom}
\end{align}
where for all $\omega\in X\backslash S_k$ for some $S_k\subset X$ with capacity zero, $\{W_t^k,\mathcal{F}_t,P_\omega\}_{t\geq 0}$ is an $(\mathcal{F}_t)_{t\geq 0}$-Brownian motion starting at zero.
Here, note that we did not put $\frac{1}{2}$ in \eqref{eq:defcE} and that \eqref{eq:Fdecom} is associated with the Fukushima decomposition.
Formally, as in \cite[Lemma~1]{BFS18}, $b_k$ is given by
\[
b_k={}_{X'}\langle k,\cdot \rangle_{X}+\lambda \frac{\partial \rho_\lambda}{\partial k}
\]
with the formal derivative of $\rho_\lambda$ in direction $k$,
\begin{align*}
\frac{\partial \rho_\lambda}{\partial k} =\lim_{s\to 0}\frac{\rho_{\lambda}(s,k)}{s},
\end{align*}
where $\rho_{\lambda}(s,k)$ is rigorously defined in Theorem~\ref{thm:rho} and Remark~\ref{rem:thmrho}.
We remark that $\rho_{\lambda}(s,k)$ is the rigorously defined version of the random variable heuristically given by $J^{0}_{0,1}(\cdot + sk)-J^{0}_{0,1}(\cdot)$.
However, we do not know whether $\frac{\partial \rho_\lambda}{\partial k}$ exists or not.

This heuristic argument is still not applicable to the two-dimensional, because the self-intersection local time $Y$ is not differentiable in the sense of Meyer-Watanabe \cite{AHZ97}. 
   In \cite{BFS18} this heuristic argument is rigorously applied to the case of the self-interaction polymer measure of fractional Brownian motions under the condition $dH<1$, where $d$ and $H$ are the dimension and the Hurst parameter of the fractional Brownian motion, respectively.
Here, we remark that in the case where $dH<1$ we do not need renormalization for the construction of the polymer measure.

\end{rem}

\subsection{Strategy of the proof of Theorem~\ref{thm:thm3}}\label{sub:strategy}

We first remark that the strategy of the proof of Theorem~\ref{thm:thm3} is similar to the argument in \cite{ARZ96}.
However, for the proofs in the present paper we need a lot of modifications from the argument in \cite{ARZ96}.
Recall that \cite{ARZ96} remains as an unfinished paper (preprint).

As mentioned above, we are able to construct the measure formally given by
\begin{align*}
\nu_\lambda(\dd \omega)=E[e^{-\lambda J_{0,1}^0}]^{-1}\exp\left(-\lambda {J_{0,1}^0}\right)\nu_0 (\dd \omega)
\end{align*}
where
\begin{align*}
J_{0,1}^0=\int_{0}^1\int_s^1 \delta_0(\omega_t-\omega_s)\dd t\dd s
\end{align*}
by approximations.
By using this formal expression we would see from Girsanov's transformation that
\begin{align}
\frac{\dd (\nu_{\lambda}\circ \tau_{uh}^{-1})}{\dd\nu_\lambda}=\exp\left(-\lambda \left(J_{0,1}^{0}(\omega-uh)-J_{0,1}^0(\omega)\right)+u\int_0^1 h_s'\dd \omega_s-\frac{u^2}{2}\int_0^1 \left(h_s'\right)^2\dd s\right)\label{eq:shiftder}
\end{align}
should corresponds with $a_{uk}$ in Theorem~\ref{thm:thm3}. The right-hand side of \eqref{eq:shiftder} has two issues to be dealt with carefully.

The first issue is to clarify how the stochastic integral $\int_0^1 h_s'\dd \omega_s$ is defined. It is not clear whether the integral is well-defined or not in It\^o's sense, because we do not know whether  $\omega$ is a semimartingale under $\nu_\lambda$ or not.
However, we can easily avoid this problem by interpreting the stochastic integral $\int_0^1 h_s'\dd \omega_s$ as
\begin{align}
\int_0^1 h_s'\dd \omega_s=h_1'\omega_1-\int_0^1 \omega_s \dd h'_s,\label{eq:stointrepre}
\end{align}
for $h\in K_0 $,
where the right-hand side of \eqref{eq:stointrepre} is well-defined for any continuous function $\omega$. This treatment of the stochastic integral has already been introduced in \cite{ARZ96}.

Next, we should make sense of the difference $J_{0,1}^{0}(\omega-uh)-J_{0,1}^0(\omega)$ under $\nu_\lambda$.
Moreover, for Theorem \ref{thm:thm3} we have to show the existence of the continuous modification of the random function $u\mapsto J_{0,1}^{0}(\cdot -uh)$ under $\nu_\lambda$ for each $h\in K_0$. For the proof we apply the Kolmogorov criterion.
Since $\nu_\lambda$ is singular to the Wiener measure $\nu _0$, we need a lot of approximations to obtain the desired moment estimate of $J_{0,1}^{0}(\omega-uh)-J_{0,1}^0(\omega)$ under $\nu_\lambda$.
We will prove the moment estimate (Theorem \ref{thm:rho}) in Section \ref{sec:4} by applying lemmas in Sections \ref{sec:2} and \ref{sec:3}.
Here, we remark that modifications of the proofs are needed from those in \cite{ARZ96} in the argument of Sections \ref{sec:2}--\ref{sec:4}.

Now we demonstrate the outline the argument in Sections \ref{sec:2}--\ref{sec:4}.
In Theorem~\ref{thm:rho} we will see that  for $\ve_0>0$ small, $\left\{\widetilde{J}_{2^{-\ve_0 n},2^{-n}}(-u,h):=J_{0,1}^{2^{-\ve_0n},2^{-n}}(\omega-uh)-J_{0,1}^{2^{-\ve_0n},2^{-n}}(\omega)\right\}$ is an $L^p(\nu_\lambda)$-Cauchy sequence for any $p\geq 1$ by showing that\begin{align}
E_{\nu_\lambda}\left[\left|\widetilde{J}_{2^{-\ve_0 (n+1)},2^{-(n+1)}}(-u,h)-\widetilde{J}_{2^{-\ve_0 n},2^{-n}}(-u,h)\right|^p\right]\leq C2^{-apn}\label{eq:JCauchyoutline}
\end{align}
for a suitable $a>0$, where $E_{\nu_\lambda}[\cdot]$ denotes the expectation with respect to $\nu_\lambda$ for $\lambda\geq 0$.
We remark that it is difficult to estimate the left-hand side of \eqref{eq:JCauchyoutline} directly, since $\nu_\lambda$ is  singular with respect to the Wiener measure $\nu_0$.
To deal with this difficulty, we approxiamte the left-side hand of \eqref{eq:JCauchyoutline} by
\begin{align}
E_{\nu_{\widetilde{\e}_n,\lambda}}\left[\left|\widetilde{J}_{2^{-\ve_0 (n+1)},2^{-(n+1)}}(-u,h)-\widetilde{J}_{2^{-\ve_0 n},2^{-n}}(-u,h)\right|^p\right]\label{eq:nuenlambdaJ}
\end{align}
for a suitable sequence $\{\widetilde{\e}_n\}$ such that the difference decays exponentially, where $\nu_{\e,\lambda}$ is an approximation of $\nu_\lambda$ given in \eqref{eq:nuelambdadfn} and $E_{\nu_{\e,\lambda}}[\cdot]$ denotes the expectation with respect to $\nu_{\e,\lambda}$. We remark that \eqref{eq:nuenlambdaJ} is easier to treat, since $\nu_{\e,\lambda}$ is absolutely continuous with respect to the Wiener measure $\nu_0$.

To prove the fact that the left-side hand of \eqref{eq:JCauchyoutline} is approximated by \eqref{eq:nuenlambdaJ}, we apply Lemma~\ref{lem:phiest}, which gives an estimate of $|\rho(\e)-\rho(\e')|$ for $\Phi$ belonging to a family set of suitable functions, where 
\begin{align}
\rho(\e):=E\left[e^{-\overline{J}_{0,1}^{\e,\lambda}}\Phi\prod_{i=1}^n \delta_{x_i}(\omega_{t_i})\right] \label{eq:defrho}
\end{align}
(the precise definition of the right-hand side is given in Remark~\ref{rem:deltaprod}).
In the proof, we prove that its derivative $\frac{\dd\rho(\e)}{\dd \e}$ is of  order $\e^{\delta-1}$ for some $\delta>0$ as $\e\searrow 0$ . The idea is basically the same as that used for the construction of $\nu_\lambda$ by Bolthausen \cite{Bol93}. However, the argument is much more delicate since we need to deal with the functions $\Phi$ in a wider class.
The differentiability of $\rho (\e)$ for $\e>0$ is proved in Lemma~\ref{lem:diffrho} with an explicit representation.
Section \ref{sec:3} is devoted to the proof of Lemma~\ref{lem:phiest}.
In the proof we give precise computations of integrals, which were omitted in \cite{ARZ96}.  

To obtain a sufficiently nice estimate for \eqref{eq:nuenlambdaJ} we will see in Lemma~\ref{lem:Jcauchy} that
\begin{align}
E_{\nu_0}\left[\left|\widetilde{J}_{2^{-\ve_0 (n+1)},2^{-(n+1)}}(-u,h)-\widetilde{J}_{2^{-\ve_0 n},2^{-n}}(-u,h)\right|^p\right]\leq C2^{-bpn}\label{eq:WieenlambdaJ2}
\end{align}
for a suitable $b>0$, where $E_{\nu_0}[\cdot]$ denotes the expectation with respect to the Wiener measure.
By the sufficiently quick convergence following from \eqref{eq:WieenlambdaJ2} and H\"older's inequality, we can show that \eqref{eq:nuenlambdaJ} decays exponentially.
The left-hand side of \eqref{eq:WieenlambdaJ2} is much easier to treat than \eqref{eq:nuenlambdaJ}, since it is the expectation of the functional of the Brownian motion. To prove \eqref{eq:WieenlambdaJ2}, we carefully use Rosen's method, which is a technique to compute and  estimate the moments of approximated self-intersection local times via the Fourier transformation.
J. Rosen originally applied his method in \cite{Ros83} for the moment estimates of off-diagonal self-intersection local time in the two and three dimensional cases.
However, in our  case we need more delicate estimates to apply Rosen's method, because of the singularity of the heat kernel at $t=0$.
This is the most different part from the unpublished paper \cite{ARZ96}.
Indeed, Lemma~\ref{lem:Jcauchy} completes the proof of a weaker version of \cite[Proposition~2.1]{ARZ96}, and the proof of \cite[Proposition~2.1]{ARZ96} seems to contain a serious error.
We give these arguments in Section~\ref{sec:2} as the proof of Lemma~\ref{lem:Jcauchy}.

In Section~\ref{sec:4} we prove Theorem~\ref{thm:rho} by applying Lemmas~\ref{lem:Jcauchy} and \ref{lem:phiest}.
In particular, we prove that $\widetilde{J}_{2^{-\ve_0 n},2^{-n}}(-u,h)$ is also an $L^p(\nu_\lambda)$-Cauchy sequence for all $\lambda>0$ and $p\geq 1$. We denote by $\rho_\lambda(-u,k)$ the $L^p(\nu_\lambda)$ limit, which corresponds to the quantity formally given by $J_{0,1}^{0}(\omega-uh)-J_{0,1}^0(\omega)$, which appears in the beginning of Section \ref{sub:strategy}.

In Section~\ref{sec:5} we prove Theorem~\ref{thm:thm3} by applying Theorem~\ref{thm:rho}.
Also in Section~\ref{sec:5} we give precise arguments, which were not sufficiently discussed in \cite{ARZ96}.

The steps above are the strategy for the proof of Theorem~\ref{thm:thm3}.

\subsection{Organization of the paper}

The organizations of the rest of the paper is as follows.

As explained in Section~\ref{sub:strategy}, in Section \ref{sec:2} we will estimate the $p$-th moments of difference between approximated self-intersection local times $\widetilde{J}_{2^{\ve_0 n},2^{-n}}(-u,h)$ with respect to the Wiener measure $\nu_0$, in particular \eqref{eq:WieenlambdaJ2}. In Section~\ref{subsec:Ronsemethod}, we review Rosen's method, which is an adjusted version to our case. In Section \ref{sec:proofLemJacauchy}, we demonstrate how to use Rosen's method in the proof of   Lemma \ref{lem:Jacauchy}, which is an improved version of  \cite[Proposition~(2.1)]{Bol93}. In Section \ref{sec:ProofLemJcauchy}, we give the proof of Lemma \ref{lem:Jcauchy}, which is the estimate for \eqref{eq:nuenlambdaJ} and the main result in Section~\ref{sec:2}.
This may be the most technical and hard part in the proofs of the
main theorems.
In Sections~\ref{subsec:delta} and \ref{subsec:diff}, we provide some lemmas which guarantee the treatment of Dirac functions and the
differentiability of a key quantity \eqref{eq:defrho} in the parameter of approximation $\e$, respectively.

In Section~\ref{sec:3}, we will show Lemma~\ref{lem:phiest}, which enables us to have a sufficiently nice estimate associated with approximation of the left-hand side of \eqref{eq:JCauchyoutline} by \eqref{eq:nuenlambdaJ}.
For the proof we need a lot of explicit calculations, because we have to check that the divergent terms are canceled out by the renormalization.

In Section~\ref{sec:4}, by combining the results  in Sections~\ref{sec:2} and \ref{sec:3}, we will prove that $\widetilde{J}_{2^{-\ve_0 n},2^{-n}}(-u,h)$ is also an $L^p(\nu_\lambda)$-Cauchy sequence for all $\lambda>0$ and $p\geq 1$.
The random variable $J_{0,1}^{0}(\omega-uh)-J_{0,1}^0(\omega)$, which appears in the beginning of Section~\ref{sub:strategy}, is given by the limit of the sequence.

Finally, in Section~\ref{sec:5} we provide the proof of Theorem~\ref{thm:thm3}, which yields the main theorems as mentioned in Section~\ref{subsec:1.2}.

We remark that some explicit calculations are put in Appendix (Section~\ref{app}).

\section{Some estimates of self-intersection local times in Wiener space}\label{sec:2}

In this section, we prepare some estimates which are needed for later arguments.
We remark that the proofs in this section are very technical.
We set
\begin{align*}
f_a(x)=(2\pi)^{-3}\exp\left(-\frac{a|x|^2}{2}\right),\qquad \text{for }a>0, x=(x_1,x_2,x_3)\in \R^3. 
\end{align*}
Then, the following holds in the weak sense \begin{align*}
\int_{\R^3} e^{i\langle x,y\rangle}f_a(y)\dd y=\frac{1}{(2\pi a)^\frac{3}{2}}e^{-\frac{|x|^2}{2a}}\to \delta_0(x),\quad a\searrow 0,
\end{align*}
where $\delta_0$ is the Dirac function at point $0\in\R^3$ , and $\lan x,y\ran=x_1y_1+x_2y_2+x_3y_3$ for $x=(x_1,x_2,x_3)$ and $y=(y_1,y_2,y_3)$ and we denote by $|x|=\sqrt{\lan x,x\ran}$ the Euclidean norm.
We remark that 
\begin{align*}
J_{0,1}^{\e,a}(\omega)=\int_{T_\e}\int_{\R^3}e^{i\lan y,\omega_t-\omega_s\ran}f_a(y)\dd y\dd s\dd t = \int _{T_\e} \frac{1}{(2\pi a )^{\frac 32}} e^{-\frac{|\omega _t -\omega _s|^2}{2a}} \dd s \dd t,
\end{align*}
where $T_\e:=\{(s,t)\in [0,1]^2:t-s\geq \e\}$.  
The following lemma has been given already in \cite[Proposition (2.1)]{Bol93}, but we will provide in Section \ref{sec:proofLemJacauchy} below its proof for self-containedness and for identification of the dependency in $\e$.
\begin{lemma}\label{lem:Jacauchy}
Let $u\in \R$, $h\in K$, and $p\geq 1$ be given. Then, for $\gamma\in \left(0,\frac{1}{4}\right)$, there exists a constant $C$ such that \begin{align*}
E\left[\left|J_{0,1}^{\e,a}(\cdot+uh)-J_{0,1}^{\e,a'}(\cdot+uh)\right|^p\right]\leq C|a-a'|^{p}a^{-\frac{5}{6}p}+C|a-a'|^{\gamma p} a^{\frac{1-4\gamma}{24}p}\e^{-\frac{3}{4}p}
\end{align*}
for $0<a'<a<1$ and for $\e \in (0,1)$.
\end{lemma}

As a corollary of Lemma \ref{lem:Jacauchy}, we obtain the following.

\begin{cor}\label{cor:JaJ}
Let $u\in \R$, $h\in K$, and $p\geq 1$ be given. Then, $\{J_{0,1}^{\e,a}(\omega+uh)\}_{a\in (0,1)}$ has a continuous version (we also denote it by $J_{0,1}^{\e,a}(\omega+uh)$) in $(X,\mathcal{B},\nu_0)$ and has the $L^p(\nu_0)$-and $\nu_0$-a.s.-limit $J_{0,1}^{\e}(\omega+uh)=\lim_{a\searrow 0}J_{0,1}^{\e,a}(\omega+uh)$.

Moreover, we have
\begin{align}
E\left[\left|J_{0,1}^{\e,a}(\cdot+uh)-J_{0,1}^{\e}(\cdot+uh)\right|^p\right]\leq Ca^{\frac{p}{6}}+C a^{\frac{1+20\gamma}{24}p}\e^{-\frac{3}{4}p}\label{eq:JaJbdd}
\end{align}
for $\gamma \in \left(0,\frac{1}{4}\right)$.
\end{cor}

\begin{proof}
Fix $\gamma\in (0,\frac{1}{4})$.
Then, we can apply the Kolmogorov continuity theorem to $\left\{J_{0,1}^{\e,a}(\omega+uh)\right\}_{a\in (0,1)}$ in view of Lemma \ref{lem:Jacauchy} for $p\geq 1$ with $p\gamma>1$. Also, we can see that $ \left\{J_{0,1}^{\e,a}(\omega+uh)\right\}_{a\in (0,1)}$ is an $L^p(\nu_0)$-Cauchy sequence as $a\searrow 0$, and hence it has the $L^p(\nu_0)$-limit.
Letting $a'\searrow 0$, \eqref{eq:JaJbdd} follows from Fatou's lemma.
\end{proof}

We define 
\begin{align}
&\widetilde{J}_{\e,a}(u,h):=J_{0,1}^{\e,a}(\omega+uh)-J_{0,1}^{\e,a}(\omega),\label{eq:Jtideuh}
\end{align}
and 
\begin{align}
\widehat{J}_{\e,a}(u_1,u_2,h):=\widetilde{J}_{\e,a}(u_1,h)-\widetilde{J}_{\e,a}(u_2,h),\label{eq:Jhatuuh}
\end{align}
for $a,\e\in (0,1)$,  $u,u_1,u_2\in \R$, and  $h\in K$.
For $\ve_0\in \left(0,\frac{1}{21}\right)$, we define
\begin{align}
\e_n&=2^{-\ve_0n} \label{eq:en}
\intertext{and}
a_n&=2^{-n} \label{eq:an}
\end{align}
for $n\geq 1$.

\begin{lemma}\label{lem:Jcauchy}
Let $\gamma\in (0,\frac{1}{2})$, $\ve_0 \in \left(0, \frac{1-2\gamma}{21}\right)$, and $h\in K$  be given. Then, there exists $\beta>0$ such that the following holds.  
For any given $p\geq 2$ and any interval $[-M,M]$ with $M>0$, there exists $C>0$ such that 
\begin{align*} 
&E\left[\left|\widehat{J}_{\e_n,a_n}(u_1,u_2,h)-\widehat{J}_{\e_m,a_m}(u_1,u_2,h)\right|^{p}\right]\leq C|u_1-u_2|^{p\gamma }2^{-\beta pm}
\end{align*}
for all $u_1,u_2\in [-M,M]$, $n\geq m\geq 1$.
\end{lemma}

We will prove this lemma in Section \ref{sec:ProofLemJcauchy} below.
Taking $u_2=0$ and combining with Corollary \ref{cor:JaJ}, we have the following.

\begin{cor}\label{cor:lemJcauchy}
$\left\{\widehat{J}_{\e_n,a_n}(u_1,u_2,h)\right\}_{n\geq 1}$, $\left\{ \widetilde{J}_{\e_n,a_n}(u,h)\right\}_{n\geq 1}$ and $\Big\{{J}_{0,1}^{\e_n}(\omega+uh)-J_{0,1}^{\e_n}(\omega)\Big\}_{n\geq 1}$ are $L^{p}(\nu_0)$-Cauchy sequences for each $u_1,u_2,u\in\R$, for any $h\in K$, and for any $p\geq 1$.
In particular, $\left\{\widetilde{J}_{\e_n,a_n}(u,h) \right\}_{n\geq 1}$ is bounded in $L^{p}(\nu_0)$.
\end{cor}

\begin{proof}
In view of Lemma \ref{lem:Jcauchy} $\left\{\widehat{J}_{\e_n,a_n}(u_1,u_2,h)\right\}_{n\geq 1}$ is an $L^p(\nu_0)$-Cauchy sequence for each $u_1,u_2\in \R$.
Since
\begin{align*}
\widehat{J}_{\e,a}(u,0,h) =\widetilde{J}_{\e,a}(u,h),
\end{align*}
$\left\{ \widetilde{J}_{\e_n,a_n}(u,h)\right\}_{n\geq 1}$ is also an $L^p(\nu_0)$-Cauchy sequence.
Furthermore, from this fact and Corollary \ref{cor:JaJ}, we have the assertion for $\Big\{{J}_{0,1}^{\e_n}(\omega+uh)-J_{0,1}^{\e_n}(\omega)\Big\}_{n\geq 1}$.
\end{proof}

\begin{rem}
In \cite[Proposition 2.1]{ARZ96}, a similar statement to Lemma \ref{lem:Jcauchy} for another approximation of $J_{0,1}^\e$ and $\eta_{n,\e}$, is discussed.  However, we can find serious errors in its proof.
Indeed, the authors of \cite{ARZ96} missed the problem in (Step 4) in the proof of Lemma \ref{lem:Jacauchy} in the present paper, which is applied in the proof of Lemma \ref{lem:Jcauchy}.
Moreover, they did not provide any proof of an analogous result to Proposition \ref{prop:Jea2} in the present paper. 
\end{rem}

\subsection{J. Rosen's method}\label{subsec:Ronsemethod}

To prove Lemmas \ref{lem:Jacauchy} and \ref{lem:Jcauchy}, we will modify the technique in \cite[Lemma 2]{Ros83}. Rosen's method applies the Fourier transformation to  moments of the approximated self-intersection local time. Then,  we will see at (Step 3) in the proof of Lemma \ref{lem:Jacauchy} that estimates of the integrals in time variables are converted into the  integrals in space variables.

For later use, we give some equations, notations, and some lemmas.

\begin{proposition}\label{prop:momentexp}
We set $T\subset [0,1]^2\cap \{0\leq s<t\leq 1\}$ and $p\in \mathbb{N}$.  Then, we find that for $g\in L^1(\R^3)$, $h\in K$, and $u\in \R$
\begin{align}
&E\left[\left(\int_{T}\dd s\dd t\int_{\R^3}e^{i\lan y,\omega_t-\omega_s+u(h_t-h_s)\ran}g(y)\right)^p\right]\notag\\
&=\int_{T^p}\dd {\mathbf{s}}\dd \mathbf{t}\int_{\R^{3p}}\dd \mathbf{y}\left(\prod_{j=1}^pe^{iu\lan y_j,h_{\mathbf{s},\mathbf{t}}^{(j)}\ran}g(y_j)\right)e^{-\frac{1}{2}\var{\sum_{j=1}^p \left\lan y_j,\omega_{t_j}-\omega_{s_j}\right\ran }}\label{eq:lpg}
\intertext{and}
&E\left[\left(\int_{T}\dd s\dd t\int_{\R^3}\left(e^{i\lan y,\omega_t+u_1h_t-\omega_s-u_1h_s\ran}-e^{i\lan y,\omega_t+u_2h_t-\omega_s-u_2h_s\ran}\right)g(y)\right)^p\right]\notag\\
&=\int_{T^p}\dd {\mathbf{s}}\dd \mathbf{t}\int_{\R^{3p}}\dd \mathbf{y}\left(\prod_{j=1}^p\left(e^{i\left\lan y_j,u_1h_{\mathbf{s},\mathbf{t}}^{(j)}\right\ran}-e^{i\left\lan y_j,u_2h_{\mathbf{s},\mathbf{t}}^{(j)}\right\ran}\right)g(y_j)\right)
e^{-\frac{1}{2}\var{\sum_{j=1}^p \left\lan y_j,\omega_{t_j}-\omega_{s_j}\right\ran }}\label{eq:lpetan}
\end{align}
where $(\mathbf{s},\mathbf{t})$ denotes an element in $[0,1]^{2p}$ with 
$\mathbf{s}=(s_1,\cdots,s_p)$, $\mathbf{t}=(t_1,\cdots,t_p)\in[0,1]^p$ such that $0\leq s_i<t_i\leq 1$ for each $i=1,\cdots,p$, $\mathbf{y}=(y_1,\cdots,y_p)\in \R^{3p}$ with $y_1,\cdots,y_p\in \R^3$, and we define $h_{\mathbf{s},\mathbf{t}}^{(j)}:=h_{t_j}-h_{s_j}$ for $(\mathbf{s},\mathbf{t})$ and $j=1,\dots,p$. 
\end{proposition}
\begin{proof}
Since $g\in L^1(\R^3)$, we can use Fubini's theorem to obtain
\begin{align*}
&E\left[\left(\int_{T}\dd s\dd t\int_{\R^3}\dd ye^{i\lan y,\omega_t-\omega_s+u(h_t-h_s)\ran}g(y)\right)^p\right]\\
&=\int_{T^p}\dd {\mathbf{s}}\dd \mathbf{t}\int_{\R^{3p}}\dd \mathbf{y}\left(\prod_{j=1}^pe^{\lan y_j,u(h_{t_j}-h_{s_j}) \ran }g(y_j)\right)E\left[\prod_{j=1}^p e^{i\lan y_j,\omega_{t_j}-\omega_{s_j}\ran}\right]\notag
\end{align*}
and that
\begin{align*}
&E\left[\left(\int_{T}\dd s\dd t\int_{\R^3}\dd y\left(e^{i\lan y,\omega_t+u_1h_t-\omega_s-u_1h_s\ran}-e^{i\lan y,\omega_t+u_2h_t-\omega_s-u_2h_s\ran}\right)g(y)\right)^p\right]\notag\\
&=\int_{T^p}\dd {\mathbf{s}}\dd \mathbf{t}\int_{\R^{3p}}\dd \mathbf{y}\left(\prod_{j=1}^p\left(e^{i\left\lan y_j,u_1\left(h_{t_j}-h_{s_j}\right)\right\ran}-e^{i\left\lan y_j,u_2\left(h_{t_j}-h_{s_j}\right)\right\ran}\right)g(y_j)\right)E\left[\prod_{j=1}^p e^{i\lan y_j,\omega_{t_j}-\omega_{s_j}\ran}\right].\notag
\end{align*}
The statement holds from the fact $\lan y_j,\omega_{t_j}-\omega_{s_j}\ran$ is a Gaussian random variable for any $\mathbf{y}$ and $\mathbf{s},\mathbf{t}$.
\end{proof}

Next, we focus on $ \textrm{Var}\left(\sum_{j=1}^p \lan y_j,\omega_{t_j}-\omega_{s_j}\ran\right)$ for $p\geq 1$ and $\mathbf{s},\mathbf{t}$.
As in the proof of \cite[Lemma 2]{Ros83}, we denote the pair $(s_j,t_j)$ by $S^j$ and set $(z_1,\cdots,z_{2p}):=(s_1,t_1,s_2,t_2,\cdots,s_p,t_p)$.  Then, we split the domain of the integral in $\mathbf{s},\mathbf{t}$ on the right-hand side of \eqref{eq:lpetan} into $T^p\cap \Delta_p(\pi)$ for some permutation $\pi$ of $\{1,\cdots,2p\}$, where
\begin{align}
\Delta_p(\pi):=\{(z_1,\cdots,z_{2p}):z_{\pi(1)}<\cdots<z_{\pi(2p)}\},\label{eq:labelingpermu}
\end{align}
Let us define a partition $\{R_i\}_{i=1}^{2p-1}$ of the interval $[z_{\pi(1)},z_{\pi(2p)}]$ by
\begin{align}
R_i:=[z_{\pi(i)},z_{\pi(i+1)}],\quad 1\leq i\leq 2p-1,\label{eq:Ri}
\end{align}
and let $W(R_i):=\omega_{z_{\pi(i+1)}}-\omega_{z_{\pi (i)}}$. Then it holds that
\begin{align*}
\omega_{t_j}-\omega_{s_j}=\sum_{i; R_i\subset [s_j,t_j]}W(R_i).
\end{align*}
Since $W(R_i)$, $i=1,2,\cdots,2p-1$, are independent, we have \begin{align}
\textrm{Var}\left(\sum_{j=1}^p \lan y_j,\omega_{t_j}-\omega_{s_j}\ran\right)&=\textrm{Var}\left(\sum_{j=1}^p\left\lan	y_j,\sum_{R_i\subset [s_j,t_j]}W(R_i)	\right\ran\right)\notag\\
&=\sum_{i}\left|\overline{y}^i\right|^2|R_i|,\label{eq:varexpand}
\end{align} 
where $\overline{y}^i=\dis \sum_{j;[s_j,t_j]\supset R_i}y_j$ and $|R_j|$ denotes the length of $R_j$. 
For convenience, we let $R_0:= [0,z_{\pi (1)}]$, $R_{2p}:= [z_{\pi(2p)}, 1]$, $\overline{y}^0:=0$ and $\overline{y}^{2p}:=0$. 
We define 
\begin{align}\label{eq:frdef}
f(\ell):= \min \left\{ i| \ R_i \subset [s_\ell ,t_\ell] \right\} , \quad r(\ell):= \max \left\{ i| \ R_i \subset [s_\ell ,t_\ell] \right\} +1
\end{align}
Then, it holds that
\begin{align*}
[s_\ell,t_\ell]=R_{f(\ell)}\cup R_{f(\ell)+1}\cup \cdots\cup R_{r(\ell)-1},
\end{align*}
and hence
\begin{align}
y_\ell=\overline{y}^{f(\ell)}-\overline{y}^{f(\ell)-1}, \quad -y_\ell=\overline{y}^{r(\ell)}-\overline{y}^{r(\ell)-1}.\label{eq:ybary}
\end{align}
We have the following three cases:
\begin{enumerate}[label=(T-\roman*)]
\item\label{item:T1} $\dis \bigcup_{j\in \{1,\dots,p\}}(s_j,t_j)$ is connected.
\item\label{item:T2} $\dis \bigcup_{j\in \{1,\dots,p\}}(s_j,t_j)$ is disconnected and there exists  $j\in \{1,\dots,p\}$ such that $(s_j,t_j)\cap \bigcup_{i\not= j}[s_i,t_i]=\emptyset$.
\item\label{item:T3} $\dis \bigcup_{j\in \{1,\dots,p\}}(s_j,t_j)$ is disconnected and for any $j\in \{1,\dots,p\}$, $(s_j,t_j)\cap \bigcup_{i\not= j}[s_i,t_i]\not=\emptyset$.
\end{enumerate}

\begin{rem}
We can find that for each case, $\left\{\overline{y}^j\right\}_{j=0}^{2p}$ has the following property.
\begin{enumerate} 
\item In the case of {\rm \ref{item:T1}}, there exists $j,k\in \{1,\dots,p\}$ such that $r(j)-1=r(k)$ and $\overline{y}^{r(k)}=y_j$. This is true for the largest $t_j$ and the second-largest $t_k$. 
\item In the case of {\rm \ref{item:T2}}, there exists $j\in \{1,\dots,p\}$ such that
\begin{itemize}
\item $\overline{y}^{f(j)-1}=\overline{y}^{r(j)}=0$.
\item $f(j)=r(j)-1$ and  $\overline{y}^{f(j)}=y_j$. 
\item $\overline{y}^k$ ($k\not= f(j)$) does not contain $y_j$.
\end{itemize}
\end{enumerate}
\end{rem}

Then, we have the following properties, which were first obtained in \cite{Ros83}.

\begin{lemma}\label{lem:RosLem3}\cite[Lemma 3]{Ros83}
For $p\in \N$, $S_1,\dots,S_p\in \{(s,t)\in [0,1]^2:t-s>0\}$, $\mathbf{y}\in \R^{3p}$, we set $\{f(\ell)\}_{\ell=1,\dots,p}$, $\{r(\ell)\}_{\ell=1,\dots,p}$, $\{\overline{y}^j\}_{j=0}^{2p}$ as above. Then,
$\overline{y}^{f(1)},\dots,\overline{y}^{f(p)}$ is a nonsingular set of coordinates for $\R^{3p}$. 
\end{lemma}

\begin{lemma}\label{lem:RosLem4}
Let $a\in (0,1)$. For $p\in \N$ with $p\geq 2$, $S_1,\dots,S_p\in \{(s,t)\in [0,1]^2:t-s\geq a\}$, $\mathbf{y}\in \R^{3p}$, we set $\{f(\ell)\}_{\ell=1,\dots,p}$, $\{r(\ell)\}_{\ell=1,\dots,p}$, $\{\overline{y}^j\}_{j=0}^{2p}$ as above. Then, for each $j$, $1\leq j\leq p$, there exists $k\in \{f(j),\dots,r(j)-1\}$ such that $|R_{k}|\geq \frac{a}{2p}$. We denote by $d(j)\in \{f(j),\cdots,r(j)-1\}$ the smallest integer in such $k$.  Moreover, we have 
 $\displaystyle \mathrm{rank}\left(\{\overline{y}^{r(1)},\dots,\overline{y}^{r(p)},\overline{y}^{d(1)},\dots,\overline{y}^{d(p)}\}\right)= 3p$. 
\end{lemma}

\begin{rem}
Lemma \ref{lem:RosLem4}  was given in \cite[Lemma 4]{Ros83}. We also remark that there might be the possibility that $R_{d(i)}=R_{d(j)}$ for $i\not=j$. 
\end{rem}

\subsection{Proof of Lemma \ref{lem:Jacauchy}}\label{sec:proofLemJacauchy}

Now we give the proof of Lemma \ref{lem:Jacauchy}.
We will split the proof into four steps.

\vspace{3mm}\noindent {\bf (Step 1)} We apply the Fourier transformation to the approximated self-intersection local times.

It is easy to see that \begin{align*}
&\left|J_{0,1}^{\e,a}(\omega+uh)-J_{0,1}^{\e,a'}(\omega+uh)\right|\\
&= \left|\int_{T_\e}\int_{\R^3}e^{i\lan y,\omega_t-\omega_s+u(h_t-h_s)\ran}\left(1-e^{-\frac{(a-a')|y|^2}{2}}\right)f_{a'}(y)\dd y\dd s\dd t\right|\\
&\leq \left|\int_{T_\e}\int_{\R^3}e^{i\lan y,\omega_t-\omega_s+u(h_t-h_s)\ran}\left(1-e^{-\frac{(a-a')|y|^2}{2}}\right)f_{a'}(y)1\{|y|\leq a^{-\frac{1}{6}}\}\dd y\dd s\dd t\right|\\
&\quad +\left|\int_{T_\e}\int_{\R^3}e^{i\lan y,\omega_t-\omega_s+u(h_t-h_s)\ran}\left(1-e^{-\frac{(a-a')|y|^2}{2}}\right)f_{a'}(y)1\{|y|>a^{-\frac{1}{6}}\}\dd y\dd s\dd t\right|.
\end{align*}
Since $1-e^{-x}\leq x$ for $x\geq 0$, we have 
\begin{align*}
& \left|\int_{T_\e}\int_{\R^3}e^{i\lan y,\omega_t-\omega_s+u(h_t-h_s)\ran}\left(1-e^{-\frac{(a-a')|y|^2}{2}}\right)f_{a'}(y)1\{|y|\leq a^{-\frac{1}{6}}\}\dd y\dd s\dd t\right|
\leq c(a-a')a^{-\frac{5}{6}}
\end{align*}
for some constant $c>0$. Hence,  it is sufficient to estimate \begin{align*}
\left|\int_{T_\e}\int_{\R^3}e^{i\lan y,\omega_t-\omega_s+u(h_t-h_s)\ran}\left(1-e^{-\frac{(a-a')|y|^2}{2}}\right)f_{a'}(y)1\{|y|>a^{-\frac{1}{6}}\}\dd y\dd s\dd t\right|
\end{align*}
for $p\in 2\N$.

\vspace{3mm}\noindent {\bf (Step 2)}
We can use Proposition \ref{prop:momentexp} by setting $g(y)=f_{a'}(y)1\{|y|>a^{-\frac{1}{6}}\}$ and $T=T_{\e}$, and changing $y_{\frac{p}{2}+1},\dots, y_p\mapsto -y_{\frac{p}{2}+1},\dots,-y_p$. For our convenience, we keep the notation $\{y_j\}_{j=1,\dots,p}$.

We remark that  for $\gamma \in (0,1)$, there exists $C>0$ such that \begin{align*}
1-e^{-x}\leq Cx^\gamma
\end{align*}
for $x\geq 0$.
Combining this with \eqref{eq:labelingpermu}, \eqref{eq:Ri}, and \eqref{eq:varexpand}, we have \begin{align*}
&E\left[\left|\int_{T_\e}\int_{\R^3}e^{i\lan y,\omega_t-\omega_s+u(h_t-h_s)\ran}\left(1-e^{-\frac{(a-a')|y|^2}{2}}\right)f_{a'}(y)1\{|y|>a^{-\frac{1}{6}}\}\dd y\dd s\dd t\right|^p\right]\\
&\leq |a-a'|^{\gamma p}\int_{T^p_\e}\dd \mathbf{s}\dd \mathbf{t}\int_{Y(a^{-\frac{1}{6}})}\dd \mathbf{y}\prod_{j=1}^p |y_j|^{2\gamma}\exp\left(-\frac{1}{2}\var{\sum_{j=1}^p \left\lan y_j,\omega_{t_j}-\omega_{s_j}\right\ran }\right)\\
&= |a-a'|^{\gamma p}\sum_{\pi}\int_{T_{\e}^p\cap \Delta_p(\pi)}\dd {\mathbf{s}}\dd \mathbf{t}\int_{Y(a^{-\frac{1}{6}})}\dd \mathbf{y}\prod_{j=1}^p |y_j|^{2\gamma}e^{-\frac{1}{2}\sum_{i=1}^{2p-1}|R_i|\left|\overline{y}^i\right|^2},
\end{align*}
where we write  $Y(t)=\bigcap_{j=1}^p\{|y_j|\geq t\}$ for $t\geq 0$.

\vspace{3mm}\noindent {\bf (Step 3)} In this step, we will take the integrals in time-variables by changing variables $(\mathbf{s},\mathbf{t})$ into $(z_{\pi(1)},|R_1|,\cdots,|R_{2p-1}|)$. 

Using \eqref{eq:ybary}, we get \begin{align*}
|y_\ell|\leq \left(|\overline{y}^{f(\ell)}|+|\overline{y}^{f(\ell)-1}|\right)^\frac{1}{2}\left(|\overline{y}^{r(\ell)}|+|\overline{y}^{r(\ell)-1}|\right)^\frac{1}{2}.
\end{align*} 
Since each  $\overline{y}^i$ is equal to either  $\overline{y}^{f(\ell)}$ or  $\overline{y}^{r(\ell)}$ for a unique $\ell$, and either  $\overline{y}^{f(q)-1}$ or  $\overline{y}^{r(q)-1}$ for a unique $q$, we have \begin{align*}
\prod_{i=1}^p |y_i|&\leq \prod_{\ell=1}^p \left(|\overline{y}^{f(\ell)}|+|\overline{y}^{f(\ell)-1}|\right)^\frac{1}{2}\left(|\overline{y}^{r(\ell)}|+|\overline{y}^{r(\ell)-1}|\right)^\frac{1}{2}\\
&\leq \prod_{i=1}^{2p-1} (1+|\overline{y}^i|^2)^\frac{1}{2},
\end{align*}
by using \eqref{eq:ybary} and $(a+b)^{\frac{1}{2}}\leq (1+a^2)^\frac{1}{4}(1+b^2)^\frac{1}{4}$ for any  $a,b\geq 0$.

We now change the variable: $J:(S^1,\cdots,S^p)\mapsto (z_{\pi(1)},|R_1|,\cdots,|R_{2p-1}|)$. 
Then, the  Jacobian matrix of $J$ has $\{0,1,-1\}$-entries and the inverse matrix $J^{-1}$ with $\{0,1,-1\}$-entries almost everywhere. 
Thus, we have $|DJ(\mathbf{s},\mathbf{t})|=1$.

We may divide $T_{\e}^p\cap \Delta_p(\pi)$  by the indices $\mathbf{d}=(d(1),\cdots,d(p))$, $T_{\e}^p(\pi,\mathbf{d})$. We remark that $J$ can be written as  the invertible linear transformation of $(S^1,\cdots,S^p)$ on $T_{\e}^p(\pi,\mathbf{d})$. Then, we have \begin{align}
&\int_{T_{\e}^p(\pi,\mathbf{d})}\dd \mathbf{s}\dd \mathbf{t}\int _{Y(a^{-\frac{1}{6}})}\dd \mathbf{y}\exp\left(-\frac{1}{2}\sum_{i=1}^{2p-1}|\overline{y}^i|^2|R_i|\right)\prod_{j=1}^{p}|y_j|^{2\gamma} \notag\\
&\leq C\int_{J(T_{\e}^p(\pi,\mathbf{d}))}\dd z_{\pi(1)}\prod_{i=1}^{2p-1}\dd |R_{i}|\int _{Y(a^{-\frac{1}{6}})}\dd \mathbf{y}\exp\left(-\frac{1}{2}\sum_{i=1}^{2p-1}|\overline{y}^i|^2|R_i|\right)\prod_{j=1}^{2p-1}\left(1+|\overline{y}^j|^2\right)^\gamma \label{eq:bdd1} 
 \end{align} 
The integral in  $ |R_{i}|$ is dominated as follows: If $i=d(k)$ for some $k\in \{1,\dots,p\}$, then there is a constant $C>0$ such that \begin{align*}
\int_{\frac{\e}{2p}}^1 \exp\left({-\frac{|\overline{y}^{d(i)}|^2t}{2}}\right)\dd t\leq \frac{C}{1+|\overline{y}^{d(i)}|^2}\exp\left({-\frac{\e|\overline{y}^{d(i)}|^2}{4p}}\right) \leq C \exp\left({-\frac{\e|\overline{y}^{d(i)}|^2}{4p}}\right).
\end{align*}
If $i\not=d(k)$ for any $k\in \{1,\dots,p\}$, then we have \begin{align*}
\int_0^1 \exp\left({-\frac{|y|^2 t}{2}}\right)\dd t\leq \frac{C}{1+|y|^2},\quad y\in \R^3.
\end{align*}
From these, we can see that the right-hand side of \eqref{eq:bdd1} can be bounded from above by \begin{align}
C\int_{Y(a^{-\frac{1}{6}})}\dd \mathbf{y}\prod_{i=1}^{p}\left(1+|\overline{y}^{f(i)}|^2\right)^{-1+\gamma}\left(1+|\overline{y}^{r(i)}|^2\right)^{-1+\gamma}e^{-\frac{\e|\overline{y}^{d(i)}|^2}{4p^2}},\label{eq:bdd2}
\end{align}
where $p^2$ of the denominator in  the power of $e$ is put to remove the effect of overlaps of $d(i)$.

The product of $2p$-terms seems to appear in \eqref{eq:bdd2} but it is in fact the product of only $(2p-1)$-terms since  $\overline{y}^{r(i)}\equiv 0$ for at least one of $j=1,\dots,p$.

\begin{rem}\label{rem:Rosenmethod}
Rosen's method uses the Fourier transformation and enables us to convert the estimate of integral in time variables into the estimate of the integral in space variables. 
\end{rem}

\vspace{3mm}\noindent {\bf (Step 4)}
{Now, we remark that on $Y(a^{-\frac{1}{6}})$ there are $1\leq i_1<\dots<i_p\leq 2p$ such that $|\overline{y}^{i_j}|\geq \frac{a^{-\frac{1}{6}}}{2}$ for $j=1,\dots,p$. Indeed, for each $i=1,\dots,2p-1$, there exists $k(i)\in \{1,\dots,p\}$ such that $y_{k(i)}=\overline{y}^{i+1}-\overline{y}^i$ or $-\overline{y}^{i+1}+\overline{y}^i$, and hence either $|\overline{y}^{2i-1}|\geq \frac{a^{-\frac{1}{6}}}{2}$ or $|\overline{y}^{2i}|\geq \frac{a^{-\frac{1}{6}}}{2}$ on $Y(a^{-\frac{1}{6}})$.} 
Let
\begin{align*}
A_{(\pi,\mathbf{d})} &:=\left\{ (y_1,\dots,y_p)\in \R^{3p}: \textrm{There exist $1\leq j_1<\dots<j_\frac{p}{2}\leq p$} \phantom{\frac{a^{-\frac{1}{6}}}{2}}\right.\\
&\quad \hspace{4cm} \left. \textrm{such that $\left|\overline{y}^{f(j_k)}\right|\geq \frac{a^{-\frac{1}{6}}}{2}$, $k=1,\dots,\frac{p}{2}$}\right\}.
\end{align*}
By the Cauchy-Schwarz inequality, the right-hand side of \eqref{eq:bdd2} is bounded from above by \begin{align*}
&C\left(\int_{A_{(\pi,\mathbf{d})}} \dd \mathbf{y}\prod_{i=1}^p \left(1+|\overline{y}^{f(i)}|^2\right)^{-2+2\gamma}\right)^\frac{1}{2}\left(\int_{A_{(\pi,\mathbf{d})}} \dd \mathbf{y}\prod_{i=1}^p \left(1+|\overline{y}^{r(i)}|^2\right)^{-2+2\gamma}e^{-\frac{\e|\overline{y}^{d(i)}|^2}{2p^2}}\right)^\frac{1}{2}\\
&+C\left(\int_{A^c_{(\pi,\mathbf{d})}} \dd \mathbf{y}\prod_{i=1}^p \left(1+|\overline{y}^{f(i)}|^2\right)^{-2+2\gamma}\right)^\frac{1}{2} \\
&\quad \hspace{3cm} \times \left(\int_{A^c_{(\pi,\mathbf{d})}\cap Y(a^{-\frac{1}{6}})} \dd \mathbf{y}\prod_{i=1}^p \left(1+|\overline{y}^{r(i)}|^2\right)^{-2+2\gamma}e^{-\frac{\e|\overline{y}^{d(i)}|^2}{2p^2}}\right)^\frac{1}{2}\\
&=:CF_{(\pi,\mathbf{d},1)}^{\frac{1}{2}}R_{(\pi,\mathbf{d},1)}^\frac{1}{2}+CF_{(\pi,\mathbf{d},2)}^{\frac{1}{2}}R_{(\pi,\mathbf{d},2)}^\frac{1}{2}.
\end{align*}
Lemmas \ref{lem:RosLem3} and \ref{lem:RosLem4} allow us to see that $\{\overline{y}^{f(1)},\cdots,{\overline{y}^{f(p)}}\}$ is a (nonsingular) set of coordinates for $\R^{3p}$ and that we can choose a (nonsingular) set of coordinates (denoted by $(\overline{y}^{k_1},\cdots,\overline{y}^{k_p})$) for $\R^{3p}$ from the set
\begin{align*}
\overline{y}^{r(1)},\cdots,\overline{y}^{r(p)},\overline{y}^{d(1)},\cdots,\overline{y}^{d(p)}.
\end{align*}
 Also, it is easy to see that each element in the above Jacobian matrices is a constant.
Therefore, we have for $\gamma\in (0,\frac{1}{4})$
\begin{align*}
F_{(\pi,\mathbf{d},1)}\leq \left(\int_{\R^3}(1+|y|^2)^{-2+2\gamma} dy\right)^{\frac{p}{2}}\left(\int_{|y|\geq \frac{a^{-\frac{1}{6}}}{2}} (1+|y|^2)^{-2+2\gamma} dy \right)^{\frac{p}{2}}\leq Ca^{\frac{(1-4\gamma)p}{12}}.
\end{align*}
Also, it is easy to see that 
 \begin{align*}
R_{(\pi,\mathbf{d},1)}\leq C\e^{-\frac{3}{2}p} , \quad F_{(\pi,\mathbf{d},2)}\leq C.
\end{align*}
Hence, the proof of Lemma \ref{lem:Jacauchy} is completed, once Proposition \ref{prop:R2bdd} below is proved.

\begin{proposition}\label{prop:R2bdd}
We have \begin{align*}
R_{(\pi,\mathbf{d},2)}\leq Ca^{\frac{1-4\gamma}{12}p}\e^{-\frac{3}{2}p}.
\end{align*}
\end{proposition}
\begin{proof}
First, we remark that \begin{align*}
A_{(\pi,\mathbf{d})}^c\cap Y(a^{-\frac{1}{6}}) 
&\subset \bigcup_{1\leq i_1<\dots<i_\frac{p}{2}\leq p}\left\{(y_1,\dots,y_p)\in \R^{3p}:\left| \overline{y}^{r(i_k)}\right|\geq \frac{a^{-\frac{1}{6}}}{2}, k=1, \dots \frac{p}{2} \right\} \\
&=: \bigcup_{\mathbf{i}}B_{\mathbf{i}},
\end{align*}
where we write $\mathbf{i}=\{i_1,\dots,i_\frac{p}{2}\}$ with $1\leq i_1<\dots<i_\frac{p}{2}\leq p$.
Fix a nonsingular coordinate $\overline{y}^{r(j_1)},\dots,\overline{y}^{r(j_k)},\overline{y}^{d(m_1)},\dots,\overline{y}^{d(m_{p-k})}$.

Let $A_1:=\{i_1,\dots,i_{\frac{p}{2}}\}$ and $A_2:=\{j_1,\dots,j_k\}$. Also, let $a_1:=\sharp (A_1\cap A_2)$ and $a_2:=\sharp (A_1\cap A_2^c)$. Then, \begin{align}
&\int_{B_{\mathbf{i}}} \dd \mathbf{y}\prod_{i=1}^p \left(1+|\overline{y}^{r(i)}|^2\right)^{-2+2\gamma}e^{-\frac{\e|\overline{y}^{d(i)}|^2}{4p^2}}\notag\\
&\leq \int_{\R^{3(p-k)}}\dd \overline{y}^{d(m_1)}\dots \dd \overline{y}^{d(m_{p-k})}\prod_{l=1}^{p-k} e^{-\frac{\e|\overline{y}^{d(m_l)}|^2}{4p^2}}\notag\\
&\times \int_{\R^{3k}}\dd \overline{y}^{r(j_1)}\dots \dd \overline{y}^{r(j_k)}\prod_{i\in A_1\cup A_2} \left(1+|\overline{y}^{r(i)}|^2\right)^{-2+2\gamma}.\label{eq:integralppp}
\end{align}

Since \begin{align*}
&\int_{\R^3}\dd y e^{-\frac{\e |y|^2}{4p^2}}\leq Cp^{3}\e^{-\frac{3}{2}}\\
&\int_{\R^3}\dd y (1+|y|^2)^{-2+2\gamma}\leq C\\
&\int_{|y|\geq \frac{a^{-\frac{1}{6}}}{2}}\dd y (1+|y|^2)^{-2+2\gamma}\leq C\left(\frac{a^{-\frac{1}{6}}}{2}\right)^{-4+4\gamma+3}\\
&(1+|y|^2)^{-2+2\gamma}1\{|y|\geq \frac{a^{-\frac{1}{6}}}{2}\}\leq C\left(\frac{a^{-\frac{1}{6}}}{2}\right)^{(-4+4\gamma)},
\end{align*}
\eqref{eq:integralppp} is dominated by \begin{align*}
C\e^{-\frac{3}{2}(p-k)}\left(\frac{a^{-\frac{1}{6}}}{2}\right)^{(-4+4\gamma+3)a_1+(-4+4\gamma)a_2}.
\end{align*}
 Since $a_1+a_2=\sharp A_1=\frac{p}{2}$ and $a_1\leq  \frac{p}{2}$, we have \begin{align*}
 (-4+4\gamma+3)a_1+(-4+4\gamma)a_2=(-4+4\gamma)\frac{p}{2}+3a_1\leq (-1+4\gamma)\frac{p}{2}.
 \end{align*}
 So, \begin{align*}
 R_{(\pi,\mathbf{d},2)}\leq C\e^{-\frac{3}{2}p}a^{\frac{1-4\gamma}{12}p}.
 \end{align*}
\end{proof}

By Proposition \ref{prop:R2bdd} the proof of Lemma \ref{lem:Jacauchy} is finished.

\begin{rem}
In the argument corresponding  with (Step 4) in \cite{ARZ96}, only the typical cases  has been discussed ($s_1<\dots<s_p<t_i$ ($i=1,\dots,p$)). 
\end{rem}

\subsection{Proof of Lemma \ref{lem:Jcauchy}}\label{sec:ProofLemJcauchy}

Lemma \ref{lem:Jcauchy} follows from Propositions \ref{prop:Jea1} and \ref{prop:Jea2} below.

\begin{proposition}\label{prop:Jea1}
Let $\gamma\in (0,\frac{1}{2})$,  and $h\in K$ be given. 
For any given $p\geq 1$, there is a constant $C\in (0,\infty)$ such that
\begin{align}\label{eq:pmom}
&E\left[\left|\widehat{J}_{\e_m,a_\ell}(u_1,u_2,h)-\widehat{J}_{\e_m,a_k}(u_1,u_2,h)\right|^p\right]\\
\nonumber &\leq C|u_1-u_2|^{p}a_\ell^{\frac{p}{7}}+C|u_1-u_2|^{\gamma p}a_\ell^{\frac{(1-2\gamma )p}{28}}\e_m^{-\frac{3}{4}p}
\end{align}
for all $u_1,u_2\in \R$, $k\geq \ell\geq 1$, and $m\geq 1$.
\end{proposition}


\begin{proof}
The proof is similar to the one of Lemma \ref{lem:Jacauchy}.
It is sufficient to consider the case $k>\ell$.
The constant $C\in (0,\infty)$ given in this proof is independent of $n,m$, $u_1,u_2$, which can be different from line to line.

First, we note that \begin{align*}
&\widehat{J}_{\e_m,a_{k}}(u_1,u_2,h)-\widehat{J}_{\e_m,a_\ell}(u_1,u_2,h)\\
&=\int_{T_{\e_m}}\dd s\dd t\int_{\R^3}\dd y\left(e^{i\lan y,\omega_t-\omega_s+u_1(h_t-h_s)\ran}-e^{i\lan y,\omega_t-\omega_s+u_2(h_t-h_s)\ran}\right)\left(1-e^{-\frac{a_\ell-a_{k}}{2}|y|^2}\right)f_{a_{k}}(y).
\end{align*}
By a similar argument to (Step 1) in the proof of Lemma \ref{lem:Jacauchy}, we can see that \begin{align*}
&\left|\widehat{J}_{\e_m,a_{k}}(u_1,u_2,h)-\widehat{J}_{\e_m,a_{\ell}}(u_1,u_2,h)\right|\\
&\leq C|u_1-u_2||a_{k}-a_\ell| a_\ell^{-\frac{6}{7}}\\
&\quad + C\left|\int_{T_{\e_m}}\dd s\dd t\int_{|y|>a_\ell^{-\frac{1}{7}}}\dd y\left(e^{i\lan y,\omega_t-\omega_s+u_1(h_t-h_s)\ran}-e^{i\lan y,\omega_t-\omega_s+u_2(h_t-h_s)\ran}\right) \right. \\
&\quad \hspace{8cm} \left. \times \left(1-e^{-\frac{a_\ell-a_{k}}{2}|y|^2}\right)f_{a_k}(y)\right|\\
&=: C|u_1-u_2||a_{k}-a_\ell| a_\ell^{-\frac{6}{7}}+|E_{k,\ell}| .
\end{align*}
We use Proposition \ref{prop:momentexp} by setting 
\begin{align*}
g(y)=g_{k,\ell}(y):=\left(1-e^{-\frac{a_{\ell} -a_k}{2}|y|^2}\right)f_{a_k}(y) 1\{|y|\geq a_\ell^{-\frac{1}{7}}\}
\end{align*}
and $T=T_{\e_m}$, and changing $y_{\frac{p}{2}+1},\dots, y_p\mapsto -y_{\frac{p}{2}+1},\dots,-y_p$. For our convenience, we will keep the notation $\{y_j\}_{j=1,\dots,p}$.
By using \eqref{eq:labelingpermu}, \eqref{eq:Ri}, and \eqref{eq:varexpand}, we have
\begin{align}
&E\left[|E_{k,\ell}|^p\right]\notag\\
&=\int_{T_{\e_m}^p}\dd {\mathbf{s}}\dd \mathbf{t}\int_{\R^{3p}}\dd \mathbf{y}\left(\prod_{j=1}^p\left(e^{i\left\lan y_j,u_1h_{\mathbf{s},\mathbf{t}}^{(j)}\right\ran}-e^{i\left\lan y_j,u_2h_{\mathbf{s},\mathbf{t}}^{(j)}\right\ran}\right)g_{k.\ell}(y_j)\right) e^{-\frac{1}{2}\sum_{j=1}^{2p-1}|R_i|\left|\overline{y}^i\right|^2}\notag\\
&=\sum_{\pi}\int_{T_{\e_m}^p\cap \Delta_p(\pi)}\dd {\mathbf{s}}\dd \mathbf{t}\int_{\R^{3p}}\dd \mathbf{y}
\left(\prod_{j=1}^p\left(e^{i\left\lan y_j,u_1h_{\mathbf{s},\mathbf{t}}^{(j)}\right\ran}-e^{i\left\lan y_j,u_2h_{\mathbf{s},\mathbf{t}}^{(j)}\right\ran}\right)g_{k,\ell}(y_j)\right)
e^{-\frac{1}{2}\sum_{j=1}^{2p-1}|R_i|\left|\overline{y}^i\right|^2}\label{eq:etanmp} .
\end{align}
We remark that for any given $\gamma\in (0,1]$ there is a constant $C\in(0,\infty)$ such that \begin{align}\label{eq:sincos}
|\sin x|\leq C(|x|^\gamma\wedge 1),\quad |1-\cos x|\leq C(|x|^\gamma\wedge 1),\quad x\in \R.
\end{align}
and it follows that \begin{align}
&E\left[|E_{k,\ell}|^p\right]\notag\\
&\leq C|u_1-u_2|^{\gamma p}\sum_{\pi}\int_{T_{\e_m}^p\cap \Delta_p(\pi)}\dd {\mathbf{s}}\dd \mathbf{t}\int_{\R^{3p}}\dd \mathbf{y}e^{-\frac{1}{2}\sum_{j=1}^{2p-1}|R_i|\left|\overline{y}^i\right|^2}\left(\prod_{j=1}^p|y_j|^\gamma |t_j-s_j|^\gamma \widetilde{g}_{\ell}(y_j)\right),\label{eq:lpeta}
\end{align}
where $\widetilde{g}_\ell(y)=1\{|y|\geq a_\ell^{-\frac{1}{7}}\}$.

Then, the statement follows by modifying the argument of (Step 3) and (Step 4) in the proof of Lemma \ref{lem:Jacauchy} by $\gamma\mapsto \frac{\gamma}{2}$, $a^{-\frac{1}{6}}\mapsto a_{\ell}^{-\frac{1}{7}}$. 
\end{proof}

\begin{proposition}\label{prop:Jea2}
Let $\gamma\in (0,\frac{1}{2})$,  and $h\in K$ be given. 
For any given $p\geq 2$, there is a constant $C\in (0,\infty)$ such that
\begin{align} 
&E\left[\left|\widehat{J}_{\e_{n+1},a_n}(u_1,u_2,h)-\widehat{J}_{\e_{n},a_n}(u_1,u_2,h)\right|^{p}\right]\notag\\
&\leq C|u_1-u_2|^{p}a_{n}^{\frac{p}{7}}+C|u_1-u_2|^{\gamma p}a_{n}^{\frac{(1-2\gamma )p}{28}}\e_{n}^{-\frac{3}{4}p}+C|u_1-u_2|^{\gamma p} \e_{n}^{\frac{\gamma p}{4}},\notag
\end{align}
for $u_1,u_2\in [-M,M]$, $n\geq 1$.

\end{proposition}

\begin{proof}
We may assume $p\in 2\mathbb{N}$.
\begin{align*}
&E\left[|\widehat{J}_{\e_{n+1},a_n}(u_1,u_2,h)-\widehat{J}_{\e_{n},a_n}(u_1,u_2,h)|^{p}\right]\\
&\leq 
3^{p-1}E\left[|\widehat{J}_{\e_{n+1},a_n}(u_1,u_2,h)-\widehat{J}_{\e_n,a_\ell}(u_1,u_2,h)|^{p}\right]\\
&\hspace{1em}+3^{p-1}E\left[|\widehat{J}_{\e_{n},a_n}(u_1,u_2,h)-\widehat{J}_{\e_{n},a_\ell}(u_1,u_2,h)|^{p}\right]\\
&\hspace{1em}+3^{p-1}E\left[|\widehat{J}_{\e_{n+1},a_\ell}(u_1,u_2,h)-\widehat{J}_{\e_{n},a_\ell}(u_1,u_2,h)|^{p}\right]\\
&\leq C|u_1-u_2|^{p}a_{n}^{\frac{p}{7}}+C|u_1-u_2|^{p \gamma}a_{n}^{\frac{1-2\gamma}{28}p}\e_{n}^{-\frac{3p}{4}}+3^{p-1}E\left[|\widehat{J}_{\e_{n+1},a_\ell}(u_1,u_2,h)-\widehat{J}_{\e_{n},a_\ell}(u_1,u_2,h)|^{p}\right]
\end{align*}
for $\ell\geq n\geq 1$, and $p\in 2\N$.
In the same manner as for \eqref{eq:etanmp}, we have \begin{align}
&E\left[|\widehat{J}_{\e_{n+1},a_\ell}(u_1,u_2,h)-\widehat{J}_{\e_{n},a_\ell}(u_1,u_2,h)|^{p}\right]\notag\\
&=\sum_{\pi}\int_{T_{\e_{n+1},\e_n}^{p}\cap \Delta_{p}(\pi)}\dd {\mathbf{s}}\dd \mathbf{t}\int_{\R^{3p}}\dd \mathbf{y}
\left(\prod_{j=1}^{p} \left(e^{i\left\lan y_j,u_1h_{\mathbf{s},\mathbf{t}}^{(j)}\right\ran}-e^{i\left\lan y_j,u_2h_{\mathbf{s},\mathbf{t}}^{(j)}\right\ran}\right)f_{a_\ell}(y_j)\right)
e^{-\frac{1}{2}\sum_{j=1}^{2p-1}|R_i|\left|\overline{y}^i\right|^2},
\label{eq:lpeta2}
\end{align}
where we set \begin{align*}
&T_{\e_{n},\e_{n+1}}:=\{(s,t)\in [0,1]^2: \e_{n+1}\leq t-s\leq \e_{n}\}.
\end{align*}
\cite[Lemma 2]{Ros83} allows us to apply the dominated convergence theorem to \eqref{eq:lpeta2} in $\ell$. Thus, we find that
\begin{align*}
&E\left[|\widehat{J}_{\e_{n+1},a_n}(u_1,u_2,h)-\widehat{J}_{\e_{n},a_n}(u_1,u_2,h)|^{p}\right]\\
&\leq C|u_1-u_2|^{p}a_{n}^{\frac{p}{7}}+C|u_1-u_2|^{p\gamma}a_{n}^{\frac{1-2\gamma}{28}p}\e_{n}^{-\frac{3p}{4}}\\
&\quad +3^{p-1}\sum_{\pi} \int_{T_{\e_{n},\e_{n+1}}^{p}\cap \Delta_{p}(\pi)}\dd {\mathbf{s}}\dd \mathbf{t} \int_{\R^{3p}}\dd \mathbf{y}\left(\prod_{j=1}^{p}\left(e^{i\left\lan y_j,u_1h_{\mathbf{s},\mathbf{t}}^{(j)}\right\ran}-e^{i\left\lan y_j,u_2h_{\mathbf{s},\mathbf{t}}^{(j)}\right\ran}\right)\right)
e^{-\frac{1}{2}\sum_{j=1}^{2p-1}|R_j|\left|\overline{y}^j\right|^2}.
\end{align*}
From this inequality we see that it is sufficient to estimate the integrals in 
\begin{align}\label{eq:l4inte}
&\int_{T_{\e_{n},\e_{n+1}}^{p}\cap \Delta_{p}(\pi)}\dd {\mathbf{s}}\dd \mathbf{t} \int_{\R^{3p}}\dd \mathbf{y}\left(\prod_{j=1}^p \left(e^{i\left\lan y_j,u_1h_{\mathbf{s},\mathbf{t}}^{(j)}\right\ran}-e^{i\left\lan y_j,u_2h_{\mathbf{s},\mathbf{t}}^{(j)}\right\ran}\right)\right)
e^{-\frac{1}{2}\sum_{j=1}^{2p-1}|R_j|\left|\overline{y}^j\right|^2}
\end{align}
due to the type of $\pi$.

\vspace{3mm} \noindent \textbf{In the case of \ref{item:T2}:}
We look at the integral in $y_j$ for $j\in \{1,\dots,p\}$ such that $(s_j,t_j)\cap \bigcup_{k\not=j}[s_k,t_k]=\emptyset$.
We only discuss in details the case $j=p$, because the other cases are similar.
Let us observe that we have
\begin{align}
&\int_{\R^{3p}}\dd \mathbf{y}\left(\prod_{j=1}^{p}\left(e^{i\left\lan y_j,u_1h_{\mathbf{s},\mathbf{t}}^{(j)}\right\ran}-e^{i\left\lan y_j,u_2h_{\mathbf{s},\mathbf{t}}^{(j)}\right\ran}\right)\right)
e^{-\frac{1}{2}\sum_{j=1}^{2p-1}|R_j|\left|\overline{y}^j\right|^2}\notag\\
&=\int_{\R^3}\dd {y}_{p}\left(e^{i\left\lan y_{p},u_1h_{\mathbf{s},\mathbf{t}}^{(p)}\right\ran}-e^{i\left\lan y_{p},u_2h_{\mathbf{s},\mathbf{t}}^{(p)}\right\ran}\right)e^{-\frac{1}{2}(t_{p}-s_{p})|y_{p}|^2}\notag\\
&\quad \times \int_{\R^{3p-3}}\dd y_1\dots \dd y_{p-1}\left(\prod_{j=1}^{p-1}\left(e^{i\left\lan y_j,u_1h_{\mathbf{s},\mathbf{t}}^{(j)}\right\ran}-e^{i\left\lan y_j,u_2h_{\mathbf{s},\mathbf{t}}^{(j)}\right\ran}\right)\right) \label{eq:Type2}\\
&\quad \times \exp\left({-\frac{1}{2}\sum_{\begin{smallmatrix}j=1,\dots,{2p-1}\\ j\not=f({p}),r({p})\end{smallmatrix}}|R_j|\left|\overline{y}^j\right|^2}\right). \notag
\end{align} 
It is easy to see that \begin{align*}
&\left|\int_{\R^3}\dd {y}_{p}\left(e^{i\left\lan y_{p},u_1h_{\mathbf{s},\mathbf{t}}^{({p})}\right\ran}-e^{i\left\lan y_{p},u_2h_{\mathbf{s},\mathbf{t}}^{({p})}\right\ran}\right)e^{-\frac{1}{2}(t_{p}-s_{p})|y_{p}|^2}\right|\\
&=\left|p_{t_{p}-s_{p}}\left(u_1\left(h_{t_{p}}-h_{s_{p}}\right)\right)-p_{t_{p}-s_{p}}\left(u_2\left(h_{t_{p}}-h_{s_{p}}\right)\right)\right|\\
&\leq \frac{C}{|t_{p}-s_{p}|^\frac{5}{2}} |u_2^2-u_1^2| |h_{t_{p}}-h_{s_{p}}|^2
\end{align*}
and 
\begin{align}
\int_{(s_p,t_p)\in T_{\e_{n},\e_{n+1}}}\frac{C}{|t_{p}-s_{p}|^\frac{5}{2}} |u_2^2-u_1^2| |h_{t_{p}}-h_{s_{p}}|^2\dd s_p\dd t_p
\leq CM|u_2-u_1|\e_n^\frac{1}{2}\label{eq:Type2inequ}
\end{align}
for $u_1,u_2\in [-M,M]$. Thus, we can reduce it to the case that $p-1$.  By induction, we can reduce the problem to \ref{item:T1} or \ref{item:T3} with smaller $p$.

\vspace{3mm} \noindent \textbf{In the case of \ref{item:T3}:}
The union $\bigcup_{j=1}^{p}[s_j,t_j]$ is decomposed into two disjoint sets as $\bigcup_{j\in J}[s_{j}, t_{j}] $ and $\bigcup_{j\in \{1,\dots,p\}\backslash J}[s_{j}, t_{j}]$ where $J$ is a subset of $ \{ 1,\dots,p\}$ with $2\leq |J|\leq p-2$ such that $\bigcup_{j\in J}[s_{j}, t_{j}]$ is connected, e.g.~$J=1,\dots,q$.
Then, similarly to \eqref{eq:Type2},
\begin{align*}
&\int_{\R^{3p}}\dd \mathbf{y}\left(\prod_{j=1}^{p}\left(e^{i\left\lan y_j,u_1h_{\mathbf{s},\mathbf{t}}^{(j)}\right\ran}-e^{i\left\lan y_j,u_2h_{\mathbf{s},\mathbf{t}}^{(j)}\right\ran}\right)\right)
e^{-\frac{1}{2}\sum_{j=1}^{2p-1}|R_j|\left|\overline{y}^j\right|^2} \\
&= \int \prod_{j=1}^{ q}\dd y_{j}  \left(\prod_{j=1}^q\left(e^{i\left\lan y_{j},u_1 h_{\mathbf{s},\mathbf{t}}^{(j)}\right\ran}-e^{i\left\lan y_{j},u_2h_{\mathbf{s},\mathbf{t}}^{(j)}\right\ran}\right)\right)
e^{-\frac{1}{2}\sum_{k=1}^{2q-1}|R_k|\left|\overline{y}^k\right|^2}\\
&\quad \times \int \prod_{j'=q+1}^p\dd y_{j'} \left(\prod_{j'=q+1}^p\left(e^{i\left\lan y_{j'},u_1 h_{\mathbf{s},\mathbf{t}}^{(j')}\right\ran}-e^{i\left\lan y_{j'},u_2h_{\mathbf{s},\mathbf{t}}^{(j')}\right\ran}\right)\right)
e^{-\frac{1}{2}\sum_{k=2q+1}^{2p-1}|R_k|\left|\overline{y}^k\right|^2},
\end{align*}
which are estimated by those for the cases  $p-q$ and $q$. Inductively, the problem are reduced to \ref{item:T1}.

\vspace{3mm} \noindent \textbf{In the case of \ref{item:T1}:} Here, we consider the estimate of \eqref{eq:l4inte} for $ p\geq 2$. Also, we may assume that $\overline{y}^{r(p)}\equiv 0$.
We apply the same argument as in \eqref{eq:bdd1} so that
\begin{align}
&\left|\int_{T_{\e_{n},\e_{n+1}}^p\cap \Delta_p(\pi)}\dd {\mathbf{s}}\dd \mathbf{t}\int_{\R^{3p}}\dd \mathbf{y}\left(\prod_{j=1}^p\left(e^{i\left\lan y_j,u_1h_{\mathbf{s},\mathbf{t}}^{(j)}\right\ran}-e^{i\left\lan y_j,u_2h_{\mathbf{s},\mathbf{t}}^{(j)}\right\ran}\right)\right)
e^{-\frac{1}{2}\sum_{j=1}^{2p-1}|R_j|\left|\overline{y}^j\right|^2}\right|\notag\\
&\leq \left|\int_{T_{\e_n,\e_{n+1}}^p\cap \Delta_p(\pi)}\dd {\mathbf{s}}\dd \mathbf{t}\int_{\R^{3p}}\dd \mathbf{y}\left(\prod_{j=1}^p\left(e^{i\left\lan y_j,u_1h_{\mathbf{s},\mathbf{t}}^{(j)}\right\ran}-e^{i\left\lan y_j,u_2h_{\mathbf{s},\mathbf{t}}^{(j)}\right\ran}\right)\right)
e^{-\frac{1}{2}\sum_{j=1}^{2p-1}|R_j|\left|\overline{y}^j\right|^2}\right|\notag\\
&\leq C |u_1-u_2|^{\gamma p}\e_n^{p\gamma }\int_{J(T_{\e_n,\e_{n+1}}^p(\pi,\mathbf{d}))}\dd z_{\pi(1)}\prod_{i=1}^{2p-1}\dd |R_{i}| \notag\\
&\quad \hspace{5cm} \times \int _{\R^{3p}}\dd \mathbf{y}\exp\left(-\frac{1}{2}\sum_{i=1}^{2p-1}|\overline{y}^i|^2|R_i|\right)\prod_{j=1}^{2p-1}\left(1+|\overline{y}_j|^2\right)^\frac{\gamma}{2},\label{eq:type1bdd0}
\end{align}
where we have used $|h_{\mathbf{s},\mathbf{t}}^{(j)}|\leq C\e_n$ for $\e_{n+1}\leq t_j-s_j\leq \e_n$ in the last line.

Now, we modify the estimates of the integral in $d|R_i|$ as follows: 
Since $\Delta_p(\pi)$ is connected, we find that $|R_i|\leq \e_n$ on $T_{\e_n,\e_{n+1}}^p$ for $1\leq p\leq 2p-1$. Hence, if $i=d(k)$ for some $k\in \{1,\dots,p\}$, then there is a constant $C>0$ such that \begin{align*}
\int_{\frac{\e_{n+1}}{2p}}^{\e_n} \exp\left(-\frac{|\overline{y}^{d(i)}|^2t}{2}\right)\dd t\leq C\left(\frac{1}{1+|\overline{y}^{d(i)}|^2}\wedge \e_n\right)\exp\left(-\frac{\e_{n+1}|\overline{y}^{d(i)}|^2}{4p}\right).
\end{align*}
If $i\not=d(k)$ for any $k\in \{1,\dots,p\}$, then we have \begin{align*}
\int_0^{\e_n} \exp\left(-\frac{|y|^2t}{2}\right)\dd t\leq C\left(\frac{1}{1+|y|^2}\wedge \e_n\right).
\end{align*}
Then, we  prove that
\begin{align}
&\text{the term on the left in the first inequality in  } \eqref{eq:type1bdd0}\notag\\
&\leq C |u_1-u_2|^{\gamma p} \e_n^{p\gamma } \int_{\R^{3p}}\dd \mathbf{y}\prod_{i=1}^{p}\left(\frac{1}{1+|\overline{y}^{f(i)}|^2}\wedge \e_n\right)\left(1+|\overline{y}^{f(i)}|^2\right)^{\frac{\gamma}{2}}\notag\\
&\hspace{12em} \prod_{j=1}^{p-1}\left(\frac{1}{1+|\overline{y}^{r(j)}|^2}\wedge \e_n\right)\left(1+|\overline{y}^{r(j)}|^2\right)^{\frac{\gamma}{2}}\prod_{k=1}^p\exp\left(-\frac{\e_{n+1}|\overline{y}^{d(k)}|^2}{4p^2}\right)\notag\\
&\leq C |u_1-u_2|^{\gamma p} \e_n^{p\gamma }\notag\\
&\hspace{3em}\times \left( \int_{\R^{3p}}\dd \mathbf{y}\prod_{i=1}^{p}\left(\frac{1}{1+|\overline{y}^{f(i)}|^2}\wedge \e_n\right)^2\left(1+|\overline{y}^{f(i)}|^2\right)^{{\gamma}}\right)^\frac{1}{2}\notag\\
&\hspace{3em}\times \left(\int_{\R^{3p}}\dd \mathbf{y}\prod_{j=1}^{p-1}\left(\frac{1}{1+|\overline{y}^{r(j)}|^2}\wedge \e_n\right)^2\left(1+|\overline{y}^{r(j)}|^2\right)^{\gamma}\prod_{k=1}^p\exp\left(-\frac{\e_{n+1}|\overline{y}^{d(k)}|^2}{2p^2}\right)\right)^\frac{1}{2}.\label{eq:typ1bdd1}
\end{align}
Changing variables $\mathbf{y}\to (\overline{y}^{f(1)},\dots,\overline{y}^{f(p)})$ (where the Jacobian is  a constant), we can see that for $\gamma<\frac{1}{2}$, 
\begin{align*}
&\int_{\R^{3p}}\dd \mathbf{y}\prod_{i=1}^{p}\left(\frac{1}{1+|\overline{y}^{f(i)}|^2}\wedge \e_n\right)^2\left(1+|\overline{y}^{f(i)}|^2\right)^{{\gamma}}\\
&\leq C\left(\int_{\R^3} \left(\frac{1}{1+|y|^2}\wedge \e_n\right)^2 (1+|y|^2)^\gamma \dd y\right)^p\\
&\leq C \left(\int_0^\infty \left(\frac{1}{t}\wedge e_n\right)^2 t^{\frac{1}{2}+\gamma}\dd t\right)^p\\
&=C\left(\int_0^{\frac{1}{\e_n}} \e^2 t^{\frac{1}{2}+\gamma}\dd t+\int_{\frac{1}{\e_n}}^\infty t^{-\frac{3}{2}+\gamma}\dd t\right)^p=C\e_n^{(\frac{1}{2}-\gamma)p}.
\end{align*}
Similarly, we find that \begin{align*}
&\int_{\R^{3p}}\dd \mathbf{y}\prod_{i=1}^{p-1}\left(\frac{1}{1+|\overline{y}^{r(i)}|^2}\wedge \e_n\right)^2\left(1+|\overline{y}^{r(i)}|^2\right)^{\gamma}\prod_{j=1}^p\exp\left(-\frac{\e_{n+1}|\overline{y}^{d(j)}|^2}{2p^2}\right)\\
&\leq C\left(\int_{\R^3} \left(\frac{1}{1+|y|^2}\wedge \e_n\right)^2 (1+|y|^2)^\gamma \dd y\right)^R \\
&\hspace{5em}\times \left(\sup_{x>0} \left(\left(\frac{1}{x}\wedge \e_n\right)^2x^\gamma\right) \right)^{(p-1-R)}\left(\int_{\R^{3}}\dd y
\exp\left(-\frac{\e_{n+1}|y|^2}{2p^2}\right)\right)^{p-R},
\end{align*}
where $R$ is a number of $\overline{y}^{r(*)}$ in a basis of $\left\{\overline{y}^{r(1)},\dots,\overline{y}^{r(p)},\overline{y}^{d(1)},\dots,\overline{y}^{d(p)}\right\}$. Using estimates \begin{align*}
&\left(\frac{1}{x}\wedge \e\right)x^{\frac{\gamma}{2}}\leq \e^{1-\frac{\gamma}{2}},\quad \text{for all $x>0$ and $\e>0$}\\
&\int_{\R^3}\exp\left(-\frac{\e_{n+1}}{2p^2}|y|^2\right)\dd y\leq C\e_{n}^{-\frac{3}{2}},
\end{align*}
we have \begin{align}
&\text{the term on the left in the first inequality in  }\eqref{eq:typ1bdd1}\notag\\
&\leq C|u_1-u_2|^{\gamma p} \e_n^{p\gamma }\cdot \e_n^{(\frac{1}{2}-\gamma)\frac{p}{2}}\cdot \e_n^{(\frac{1}{2}-\gamma)\frac{R}{2}} \cdot \e_n^{(1-\frac{\gamma}{2})(p-1-R)}\e_n^{-\frac{3}{2}\frac{p-R}{2}}\leq C|u_1-u_2|^{\gamma p} \e_n^{\frac{p}{2}-\left(1-\frac{\gamma}{2}\right)}.\label{eq:Type1ineq}
\end{align}
Finally, for each $\pi$, we divide $\bigcup_{j\in \{1,\dots,p\}}(s_j,t_j)$ into the disjoint open intervals  $\bigcup_{k=1}^P(a_k,b_k)$. Then, for each $k=1,\dots,P$, there exists $J_k\subset \{1,\dots, p\}$ such that $(a_k,b_k)=\bigcup_{j\in J_k}(s_j,t_j)$. Let $\Pi_1=\{k:|J_k|\geq 2\}$ and $\Pi_2=\{k:|J_k|=1\}$.
In particular, we can regard $\bigcup_{j\in J_k}(s_j,t_j)$ as \ref{item:T1} if $k\in\Pi_1$ or \ref{item:T2} if $k\in \Pi_2$, and hence the above estimates \eqref{eq:Type2inequ} and \eqref{eq:Type1ineq}  yield that \begin{align*}
&\text{the term on the left in the first inequality in } \eqref{eq:l4inte}\\
&\leq C|u_1-u_2|^{|\Pi_2|+\gamma (p-|\Pi_2|)}\e_n^{\frac{|\Pi_2|}{2}+\sum_{k\in \Pi_1}\left(\frac{|J_k|}{2}-\left(1-\frac{\gamma}{2}\right)\right)}\notag\\
 &\leq C|u_1-u_2|^{\gamma p}\e_n^{\frac{\gamma p}{4}}
\end{align*}
for $u_1,u_2\in [-M,M]$.

\end{proof}

\subsection{Remarks on the treatment of Dirac functions }\label{subsec:delta}

In this section, we will look at the self-intersection local time for pinned Brownian motion. In particular, its continuity will allow us to use the notation
\begin{align*}
E\left[e^{-\overline{J}^{\e,\lambda}_{0,1}}\Phi\prod_{i=1}^n \delta_{x_i}(\omega_{t_i})\right]
\end{align*}
(see Remark \ref{rem:deltaprod}) where $\overline{J}^{\e,\lambda}_{0,1}$ is given after \eqref{eq:kappa2}.

Similarly to \cite{Bol93}, we define a self-intersection local time of conditioned Brownian motion.
Fix $\e>0$. For each $0<u<v<1$,  $x,y\in \R^3$, $a>0$, we define \begin{align*}
H_{u,v,x,y}^{\e,a}:=\int_\e^1\dd t\int_0^{t-\e}\dd sp_{a}\left(\Psi_{u,v,x,y}(\omega)(t)-\Psi_{u,v,x,y}(\omega)(s)\right),
\end{align*}
where for $\omega\in C([0,1],\R^3)$\begin{align*}
\Psi_{u,v,x,y}(\omega)(\tau)=\begin{cases}
\dis x\frac{\tau}{u}+\omega_\tau-\frac{\tau}{u}\omega_u,\quad &0\leq \tau\leq u\\
\dis x\left(1-\frac{\tau-u}{v-u}\right)+y\frac{\tau-u}{v-u}+\left(\omega_\tau-\omega_u-\frac{\tau-u}{v-u}(\omega_v-\omega_u)\right),\quad &u\leq \tau\leq v\\
\dis y+\omega_\tau-\omega_v,\quad &v\leq \tau.
\end{cases}
\end{align*}

By modifying the proof of Proposition \ref{prop:Jea1} or \cite[Proposition 2.1]{Bol93}, we can prove the following.
\begin{lemma}\label{lem:Jspcont}
Let $\e>0$,  $0\leq u<v\leq 1$, and $p\geq 2$ be given. Then, there exists a version of $H_{u,v,x,y}^{\e,a}$ in $(X,\mathcal{B},\nu_0)$ which is jointly continuous in $x,y\in \R^3$, and $a> 0$.Moreover, $\dis\lim_{a\searrow 0}H_{u,v,x,y}^{\e,a}=:H_{u,v,x,y}^{\e}$ exists $\nu_0$-a.s.~and in $L^p(\nu_0)$ and also it is jointly continuous in all variables $x,y\in \R^3$ and the parameter $\e>0$.
\end{lemma}

\begin{proof}
We omit the proof since it follows by the arguments of \cite[Proposition 2.1]{Bol93}.
In fact we have the estimate that there exists $\gamma\in (0,\frac{1}{4})$ such that for any $p\geq 2$, \begin{align*}
E\left[\left|H_{u,v,x,y}^{\e,a}-H_{u,v,x',y'}^{\e,a'}\right|^p\right]\leq C\left(|x-x'|^{\gamma p}+|y-y'|^{\gamma p}+|a-a'|^\frac{\gamma p}{2}\right).
\end{align*}  
\end{proof}

\begin{rem}
The map $\Psi_{u,v,x,y}$ maps a Brownian motion to conditional Brownian motion which bridges between $(0,0)$, $(u,x)$, $(v,y)$. Lemma \ref{lem:Jspcont} can be generalized for any $n$ intermediate points.  For $0\leq t_1<\dots<t_n\leq 1$, $x_1,\dots,x_n\in \R^3$, we denote by $H^{\e,a}_{(t_i,x_i)}$ a modified version of $H$. We omit the upper index $a$ for $a=0$. 
\end{rem}

\begin{cor}\label{cor:patodelta}
Let $\Phi$ be any bounded measurable function. Then, we have \begin{align*}
&\lim_{a\to 0}E\left[e^{-\overline{J}^{\e,\lambda}_{0,1}}\Phi \prod_{i=1}^np_a(\omega_{t_i},x_i)\right]\\
&= E\left[\left.e^{-\overline{J}^{\e,\lambda}_{0,1}}\Phi \right|\omega_{t_i}=x_i,i=1,\dots,n\right]p_{t_1}(0,x_1)\prod_{j=2}^{n}p_{t_j-t_{j-1}}(x_{j-1},x_j)
\end{align*}
for $0\leq t_1<\dots<t_n\leq 1$, $x_1,\dots,x_n\in\R^3$.
\end{cor}

\begin{proof}
From construction, we  have that \begin{align*}
&\int_{\R^{3n}}\dd \mathbf{y}E\left[\left.e^{-\overline{J}^{\e,\lambda}_{0,1}}\Phi \prod_{i=1}^np_a(\omega_{t_i},x_i)\right|\omega_{t_i}=y_i,i=1,\dots,n\right]p_{t_1}(0,y_1)\prod_{j=2}^{n}p_{t_j-t_{j-1}}(y_{j-1},y_j)\\
&=\int_{\R^{3n}}\dd \mathbf{y}E\left[e^{-\lambda {H}^{\e}_{(t_i,y_i)}+\lambda \kappa_1(\e)-\lambda^2\kappa_2(\e)}\Phi \right]\prod_{i=1}^np_a(y_i,x_i)p_{t_1}(0,y_1)\prod_{j=2}^{n}p_{t_j-t_{j-1}}(y_{j-1},y_j).
\end{align*}
Letting $a\searrow 0$, we see that the statement holds.
\end{proof}

Corollary \ref{cor:patodelta} allows us to use the formal expression.

\begin{rem}\label{rem:deltaprod}
Let us use the following notation: 
\begin{align*}
&E\left[e^{-\overline{J}^{\e,\lambda}_{0,1}}\Phi\prod_{i=1}^n \delta_{x_i}(\omega_{t_i})\right]\\
&:=E\left[\left.e^{-\overline{J}^{\e,\lambda}_{0,1}}\Phi \right|\omega_{t_i}=x_i,i=1,\dots,n\right]p_{t_1}(0,x_1)\prod_{j=2}^{n}p_{t_j-t_{j-1}}(x_{j-1},x_j).
\end{align*}
Also, we set
\begin{align*}
&E\left[e^{-\overline{J}^{\e,\lambda}_{0,1}}\Phi\prod_{i=1}^n \delta_0(\omega_{s_i}-\omega_{t_i})\right] :=\int_{\R^{3n}}E\left[e^{-\overline{J}^{\e,\lambda}_{0,1}}\Phi\prod_{i=1}^n \left(\delta_{x_i}(\omega_{s_i})\delta_{x_i}(\omega_{t_i})\right)\right]\dd x_1\dots \dd x_n
\end{align*}
for $s_1,\dots,s_n,t_1,\dots,t_n\in [0,1]$. 
\end{rem}

Similarly to Corollary \ref{cor:patodelta}, we can show that \begin{align*}
&E\left[e^{-\overline{J}^{\e,\lambda}_{0,1}}\prod_{i=1}^n \delta_{x_i}(\omega_{r_i})\prod_{j=1}^mJ_{s_j,t_j}^{\e} \Phi\right]=\int_{T_\e^m}\dd \mathbf{s}\dd \mathbf{t} E\left[e^{-\overline{J}^{\e,\lambda}_{0,1}}\prod_{i=1}^n \delta_{x_i}(\omega_{r_i})\prod_{j=1}^m\delta_0(\omega_{s_j}-\omega_{t_j}) \Phi\right]
\end{align*}
for $r_1,\dots,r_n\in [0,1]$ and $x_1,\dots,x_n\in \R^3$. We omit the detail of the proof.

\subsection{Differentiability}\label{subsec:diff}

In this subsection, we will look at the differentiability of \begin{align}
\rho(\e):=E\left[\exp\left(-\overline{J}_{0,1}^{\e,\lambda}\right)\Phi\right]\quad \label{dfn:rho2}
\end{align}
in $\e\in (0,1)$, where $\Phi$ is any bounded random variable.

\begin{lemma}\label{lem:diffrho}
$\rho$ is differentiable at $\e\in (0,1)$ for any bounded random variable $\Phi$ and satifies
\begin{align*}
\frac{\dd}{\dd \e}\rho(\e)&=\lambda \int_0^{1-\e}\dd sE\left[e^{-\overline{J}^{\e,\lambda}_{0,1}}\Phi\delta _0 (\omega_s-\omega_{s+\e})\right]+\left(\lambda\frac{\dd}{\dd \e}\kappa_1(\e)-\lambda^2 \frac{\dd}{\dd \e}\kappa_2(\e)\right)\rho(\e) .
\end{align*}
\end{lemma}

\begin{rem}
In \cite{Bol93}, $\rho(\e)$ with some $\Phi$ is used to construct the polymer measure $\nu_\lambda$. 
The author focused on a lower bound for
\begin{align*}
\rho(\e_1)-\rho(\e_2).
\end{align*}
for $0<\e_1<\e_2$. He avoided the discussion of differentiability of $\rho$ and used the lower bound derived from the derivative of the modified function in which ${J}_{0,1}^\e$ is replaced by $J_{0,1}^{\e,a}$.
\end{rem}

\begin{proof}
We may assume that $\Phi$ is a positive random variable.
Fix $\e\in (0,1)$. Let $\widetilde{\e}\in(0,1-\e)$ be small enough.  Then, we have \begin{align*}
\rho(\e+\widetilde{\e})-\rho(\e)=-E\left[\exp\left(-\overline{J}^{\e+\widetilde{\e},\lambda}_{0,1}\right)\left(\exp\left(\overline{J}^{\e+\widetilde{\e},\lambda}_{0,1}-\overline{J}^{\e,\lambda}_{0,1}\right)-1\right)\Phi\right].
\end{align*}
We can cope with the deterministic part in $\overline{J}^{\cdot,\lambda}_{0,1}$ easily. Thus, it is enough to focus on \begin{align*}
E\left[\exp\left(-\overline{J}^{\e+\widetilde{\e},\lambda}_{0,1}\right)\left(\exp\left(\lambda \left({J}^{\e+\widetilde{\e}}_{0,1}-{J}^{\e}_{0,1}\right)\right)-1\right)\Phi\right].
\end{align*}
Since $\lambda \left({J}^{\e+\widetilde{\e}}_{0,1}-{J}^{\e}_{0,1}\right)$ is negative, we have \begin{align*}
&E\left[\exp\left(-\overline{J}^{\e+\widetilde{\e},\lambda}_{0,1}\right)\left(\exp\left(\lambda \left({J}^{\e+\widetilde{\e}}_{0,1}-{J}^{\e}_{0,1}\right)\right)-1\right)\Phi\right]\\
&\geq \lambda E\left[\exp\left(-\overline{J}^{\e+\widetilde{\e},\lambda}_{0,1}\right)\left({J}^{\e+\widetilde{\e}}_{0,1}-{J}^{\e}_{0,1}\right)\Phi\right]\\
&=-\lambda \int_{T_\e\backslash T_{\e+\widetilde{\e}}}\dd u\dd v E\left[\exp\left(-\overline{J}^{\e+\widetilde{\e},\lambda}_{0,1}\right)\delta_{0}(\omega_u-\omega_v)\Phi\right].
\end{align*}
It is easy to see that \begin{align*}
&\lim_{\widetilde{\e}\to 0}\frac{1}{\widetilde{\e}}\int_{T_\e\backslash T_{\e+\widetilde{\e}}}\dd u\dd v E\left[\exp\left(-\overline{J}^{\e+\widetilde{\e},\lambda}_{0,1}\right)\delta_{0}(\omega_u-\omega_v)\Phi\right]\\
&= \int_0^{1-\e}\dd sE\left[e^{-\overline{J}^{\e,\lambda}_{0,1}}\Phi\delta _0 (\omega_s-\omega_{s+\e})\right].
\end{align*}

On the other hand, we can use the following bound obtained by Rosen \cite[p.336, just below (4.4)]{Ros83}: For any $k\geq 2$ and $p>4 $, there exists a constant $C>0$ such that 
\begin{align}
E\left[\left|{J}^{\e+\widetilde{\e}}_{0,1}-{J}^{\e}_{0,1}\right|^k\right]\leq C|T_\e\backslash T_{\e+\widetilde{\e}}|^{\frac{k}{p}}\leq C\widetilde{\e}^{\frac{k}{p}},\label{eq:Jdiffmoment}
\end{align}
where $|T_\e\backslash T_{\e+\widetilde{\e}}|$ is the Lebesgue measure of  $T_\e\backslash T_{\e+\widetilde{\e}}$ in $\R^2$. 
Therefore, for any $\delta>0$, \begin{align*}
\varlimsup_{\widetilde{\e}\to 0}\frac{1}{\widetilde{\e}}E\left[\exp\left(-\overline{J}^{\e+\widetilde{\e},\lambda}_{0,1}\right)\left(\exp\left({J}^{\e+\widetilde{\e}}_{0,1}-{J}^{\e}_{0,1}\right)-1\right)\Phi:{J}^{\e}_{0,1}-{J}^{\e+\widetilde{\e}}_{0,1}>\delta\right]=0.
\end{align*}
Let $r>0$ be a small constant and take $\delta>0$ such that $1-e^{-x}\geq (1-r)x$ for $0\leq x\leq \delta$. Then, \begin{align*}
&E\left[\exp\left(-\overline{J}^{\e+\widetilde{\e},\lambda}_{0,1}\right)\left(\exp\left(\lambda \left({J}^{\e+\widetilde{\e}}_{0,1}-{J}^{\e}_{0,1}\right)\right)-1\right)\Phi\right]\\
&\leq -(1-r)\lambda E\left[\exp\left(-\overline{J}^{\e+\widetilde{\e},\lambda}_{0,1}\right)\left({J}^{\e}_{0,1}-{J}^{\e+\widetilde{\e}}_{0,1}\right)\Phi:{J}^{\e}_{0,1}-{J}^{\e+\widetilde{\e}}_{0,1}\leq \delta\right]+o(\widetilde{\e})\\
&\leq  -(1-r)\lambda E\left[\exp\left(-\overline{J}^{\e+\widetilde{\e},\lambda}_{0,1}\right)\left({J}^{\e}_{0,1}-{J}^{\e+\widetilde{\e}}_{0,1}\right)\Phi\right]+o(\widetilde{\e}),
\end{align*}
where we have used \eqref{eq:Jdiffmoment} in the last line.
Thus, we have \begin{align*}
&\varlimsup_{\widetilde{\e}\to 0}\frac{1}{\widetilde{\e}}E\left[\exp\left(-\overline{J}^{\e+\widetilde{\e},\lambda}_{0,1}\right)\left(\exp\left(\lambda \left({J}^{\e+\widetilde{\e}}_{0,1}-{J}^{\e}_{0,1}\right)\right)-1\right)\Phi\right]\\
&\leq (1-r)\lambda \int_0^{1-\e}\dd sE\left[e^{-\overline{J}^{\e,\lambda}_{0,1}}\Phi\delta _0 (\omega_s-\omega_{s+\e})\right].
\end{align*}
Since $0<r<1$ is arbitrary, the right differentiability follows. The left differentiability follows by the same kind of arguments.
\end{proof}

\section{Some estimates with respect to $\nu_{\e,\lambda}$ }\label{sec:3}

In this section, we see that  $|\rho(\e)-\rho(\e')|$ decays polynomially in $\e,\e'\searrow 0$ for $\Phi$ in a suitable class of functions, where $\rho(\e)$ is defined in \eqref{dfn:rho2}. Thus, we find that $\lim_{\e\searrow 0}\rho(\e)$ exists.

\subsection{Recursive estimates}
 
For $x\in\R^3$, we set $p_t(x):=(2\pi t)^{-\frac{3}{2}}\exp\left(-\frac{|x|^2}{2t}\right)$, and
\begin{align*}
g_t^{\e,\lambda}(x):=E\left[\exp\left(-\overline{J}_{0,t}^{\e,\lambda}\right)\delta_{x}(\omega_t )\right] .
\end{align*}
In \cite[Proposition 3.1]{Bol93} and \cite[Proposition 3.4]{AZ98} it was shown that for any given $\lambda\in [0,\infty)$ there is a constant $C\in (0,\infty)$ such that
\begin{align}
\left|g_{t}^{\e,\lambda}(x)-p_t(x)\right|\leq Ct^{\frac{1}{2}}p_{2t}(x)\label{eq:partpest}
\end{align}
for all $t\in [0,1]$ and $x\in \R^3$.
For $n\geq 1$, we define an $n$-points ``transition density" by
\begin{align*}
g_{\mathbf{t}}^{T,\e,\lambda}(\mathbf{x}):=E\left[\exp\left(-\overline{J}_{0,T}^{\e,\lambda}\right)\delta_{x_1}(\omega_{t_1}) \cdots \delta_{x_n}(\omega_{t_n}) \right],
\end{align*}
for $\mathbf{t}=(t_1,\cdots,t_n)$ with $0<t_1<\cdots<t_n\leq T\leq 1$ and $\mathbf{x}=(x_1,\dots,x_n)\in (\R^{3})^{n}$.
The following lemma is provided in \cite[Lemma 4.1]{AZ98}.

\begin{lemma}\label{lem:finitetra}
For any $n\geq 1$ and $\lambda\in [0,\infty)$, there exists a constant $C\in (0,\infty)$ such that \begin{align*}
g_{\mathbf{t}}^{T,\e,\lambda}(\mathbf{x})\leq Cp_{2t_1}(x_1)p_{2(t_2-t_1)}(x_2-x_1)\cdots p_{2(t_n-t_{n-1})}(x_n-x_{n-1}).
\end{align*}
for $\mathbf{x}=(x_1,\dots,x_n)\in (\R^3)^n$, $\lambda\in [0,\infty)$ and $\e\in (0,1)$.
\end{lemma}

Let $\psi:[0,1]\times \R^3\to \R$ be bounded and infinitely often differentiable with bounded derivatives of all orders, and for $0\leq s<t\leq 1$, we define $\Psi_{s,t}$ by \begin{align}
\Psi_{s,t}:=\exp\left(\int_s^t \psi(u,\omega_u)\dd u\right).\label{eq:psidef}
\end{align}
As was shown in \cite[\S 4]{Bol93}, \eqref{eq:partpest} implies that for any such $\psi$ and $\lambda\in [0,\infty)$ there is a constant $C\in (0,\infty)$ such that \begin{align}
\left|E\left[\exp\left(-\overline{J}_{0,1}^{\e_1,\lambda}\right)\Psi_{0,1}\right]-E\left[\exp\left(-\overline{J}_{0,1}^{\e_2,\lambda}\right)\Psi_{0,1}\right]\right|\leq C(\e_1^{\delta}+\e_2^{\delta})
\label{eq:psiest}
\end{align} 
for all $\e_1,\e_2\in (0,1)$, where $\delta\in (0,1)$. This implies that the limit
\begin{align}
N(\lambda)=\lim_{\e\to 0}E\left[\exp\left(-\overline{J}_{0,1}^{\e,\lambda}\right)\right] \in (0,\infty ) \label{eq:Nlam}
\end{align} 
exists for all $\lambda\in[0,\infty)$. Strict positivity follows from the fact that the constructed measure corresponds to the one constructed by Westwater  \cite[Theorem 1.2]{AZ98}.
From \eqref{eq:psiest} and \eqref{eq:Nlam}, the convergence of the approximated polymer measure  $\nu_{\varepsilon,\lambda}$ follows. Thus, Bolthausen constructed the three-dimensional polymer measure in \cite{Bol93}. 
In this section, we will give a similar estimate for more general functions $\Phi$.

Let $f:\mathbb{R}^3\to \R$ be a bounded measurable function. Let $\Delta_{u,v;s,t}:=\{(\sigma,\tau):u\leq \sigma\leq v,s\leq \tau\leq t\}$, and for any fixed $a\in (0,1)$ set
\begin{align}
\Phi_{u,v;s,t}:=\int_{T_{a}\cap \Delta_{u,v;s,t}}f(\omega_\sigma-\omega_\tau)\dd \sigma \dd \tau \label{eq:PhiF}	
\end{align}
where $T_a := \{ (s,t)\in [0,1]^2: t-s \geq a\}$ as above. 
For short, we set $\Phi:=\Phi_{0,1;0,1}$ and $\Phi_{s,t}:=\Phi_{s,t;s,t}$. For this $\Phi$, we set \begin{align}
\rho(\e)=E\left[\exp\left(-\overline{J}_{0,1}^{\e,\lambda}\right)\Phi\right]\label{eq:phidef}
\end{align}
in \eqref{dfn:rho2}.

The main aim of this section is to prove the following lemma, that will be used in the proof of Theorem \ref{thm:rho} below.

\begin{lemma}\label{lem:phiest}
There are constants $C\in (0,\infty)$ and $\widetilde{\delta} \in (0,1)$ such that
\begin{align}
|\rho(\e_1)-\rho({\e_2})|\leq C\sup_{x\in\mathbb{R}^3}|f(x)|(\e_1^{\widetilde{\delta}}+\e_2^{\widetilde{\delta}})\label{eq:phiest}
\end{align}
for $\e_1,\e_2\in (0,1)$.
\end{lemma}

The proof of Lemma \ref{lem:phiest} will be provided in Section \ref{sec:3.2} below.
The estimate \eqref{eq:phiest}, with $|\Phi|$ replaced by $\Psi_{0,1}$ in \eqref{eq:psidef}, was already proved by Bolthausen in \cite[\S 4]{Bol93}. By modifying his argument one can see that the discussion in \cite[\S 4]{Bol93} is also suitable for the present case. Since $|\Phi_{0,{t+s}}|$ is not equal to $|\Phi_{0,s}||\Phi_{s,s+t}|$, some arguments given in the proof of Lemma \ref{lem:phiest} below are different from those given in \cite[\S 4]{Bol93}.

\begin{rem}\label{rem:lemphiest}\begin{itemize}

\item Lemma \ref{lem:phiest} is true for $|\Phi|$ where $\Phi=\Phi_1-\Phi_2$ for $\Phi_1,\Phi_2$ of the form of the $\Phi$ entering in \eqref{eq:PhiF}.
A typical example of $\Phi$ to which Lemma \ref{lem:phiest} applies is:\begin{align*}
\Phi:=\widetilde{J}_{\e_n,a_n}(u,h)-\widetilde{J}_{\e_{n-1},a_{n-1}}(u,h)\quad h\in K
\end{align*}
where $\widetilde{J}$ was defined just after Corollary \ref{cor:JaJ}.

\item We are able to generalize \eqref{eq:phiest} as follows: 
Let $h$ be a given function in $K_0$ (with $K_0$ in \eqref{eq:Kodfn}  and $K$ in \eqref{eq:Kdfn}) and $U_{s,t}(h) :=h'_t\omega_t-h'_s\omega_s-\int_s^t \omega_\tau h''_\tau\dd \tau$ with $U_1(h) :=U_{0,1}(h)$.
Let $F$ be an element of $C_b^1(\R^{n+2})$, the set of continuously differentiable functions $F$ with $\dis \sup_{x\in \R^{n+2}}|F(x)|+\sum_{|\alpha|=1}\sup_{x\in \R^{n+2}}|\partial _{\alpha}F(x)|<\infty$. 
Then, we have (with $\overline{J}^{\e,\lambda}_{0,1}$ as given after \eqref{eq:kappa2})
\begin{align*}
&\left| E\left[\exp\left(-\overline{J}_{0,1}^{\e_1,\lambda}\right)F(\Phi,U_1,\omega_{t_1},\cdots,\omega_{t_n})\right]-E\left[\exp\left(-\overline{J}_{0,1}^{\e_2,\lambda}\right)F(\Phi,U_1,\omega_{t_1},\cdots,\omega_{t_n})\right]\right|\\
&\leq C\left(\sup_{x\in \R^{n+2}} |F(x)|+\max|f|\sup_{x\in \R^{n+2}}|\partial _{x_1}F(x)|+\sum_{i=2}^{n+2}\sup_{x\in \R^{n+2}}|\partial _{x_i}F(x)|\right)(\e_1^{\delta_0}+\e_2^{\delta_0}),
 \end{align*}
for any $0\leq t_1<\cdots<t_n\leq 1$.

\item We also have another generalization:
For $i=1,\cdots,m$, let $f_i:\R^3\to \R$ be any bounded measurable function. For $a_i\in (0,1)$ and $i=1,\cdots,m$, set
\begin{align*}
\Psi^{(i)}_{u,v;s,t}:=\int_{T_{a_i}\cap \Delta_{u,v;s,t}}f_i(\omega_\tau-\omega_{\sigma})\dd \sigma\dd \tau .
\end{align*} 
Then, we have that for any given $\ell\geq 1$, there are constants $C\in (0,\infty)$ and $\widetilde{\delta} \in (0,1)$ such that (with $\overline{J}^{\e,\lambda}_{0,1}$ as given after \eqref{eq:kappa2})
\begin{align*}
&\left|E\left[\exp\left(-\overline{J}_{0,1}^{\e_1,\lambda}\right) \prod_{i=1}^m\Psi^{(i)}_{0,1;0,1}\right]-E\left[\exp\left(-\overline{J}_{0,1}^{\e_2,\lambda}\right)\prod_{i=1}^m\Psi^{(i)}_{0,1;0,1}\right]\right|\\
&\leq C\prod_{i=1}^m \max_{x\in R^3}|f_i(x)|\left(\e_1^{\widetilde{\delta} }+\e_2^{\widetilde{\delta} }\right) 
\end{align*}
for $\e_1,\e_2\in (0,1)$.
For the proof, we remark that the same argument as the proof of Lemma \ref{lem:phiest} can be applied in view of the fact that
\begin{align*}
\dis \left|\prod_{i=1}^n (x_i+v_i)-\prod_{i=1}^n x_i\right|\leq \sum_{i=1}^n |v_i|\prod_{j\not= i} |x_j+v_j|\vee |x_j|
\end{align*}
for any $x_i, v_i \in {\mathbb R}$ ($i=1,2,\dots ,n$).

\end{itemize}\end{rem} 

\begin{rem}
Lemma \ref{lem:phiest} is the same as the one in \cite{ARZ96} and the idea of the proof is based on the one in \cite{ARZ96} (where the detailed estimation of some integrals was omitted, but we confirm that it can be provided).
\end{rem}

\subsection{Proof of Lemma \ref{lem:phiest}}\label{sec:3.2}

\begin{proof}[Proof of Lemma \ref{lem:phiest}.]
For short, we will denote $\sup_{x\in \R^3}|f(x)|$ by $\max|f|$. In addition, we remark that $\rho(\e)\leq C\max|f|$ for $0<\e<\frac{1}{2}$.
From Lemma \ref{lem:diffrho},  we have the following relation:
\begin{align}
\frac{\dd}{\dd \e}\rho(\e)&=\lambda \int_0^{1-\e}\dd sE\left[e^{-\overline{J}^{\e,\lambda}_{0,1}}\Phi\delta _0 (\omega_s-\omega_{s+\e})\right]+\left(\lambda\frac{\dd}{\dd \e}\kappa_1(\e)-\lambda^2 \frac{\dd}{\dd \e}\kappa_2(\e)\right)\rho(\e) .\label{eq:derofrho}
\end{align}
We remark that in \cite{Bol93} the differentiability of $\rho(\e)$ is not provided.

We decompose $\overline{J}$ as follows (see Figure \ref{fig:DecomJ01}):
\begin{align}
\overline{J}_{0,1}^{\e,\lambda}&=\lambda J_{0,1}^\e-\lambda \kappa_1(\e)+\lambda^2\kappa_2(\e)\notag\\
&=\lambda J_{0,s}^\e-\lambda s \kappa_1(\e)+\lambda^2 s\kappa_2(\e)\notag\\
&\quad +\lambda J_{s+\e,1}^\e-\lambda (1-s-\e) \kappa_1(\e)+\lambda^2 (1-s-\e)\kappa_2(\e)\notag\\
&\quad +\lambda J_{0,s;s+\e,1}^\e+\lambda J_{0,s;s,s+\e}^\e+\lambda J_{s,s+\e;s+\e,1}^\e-\e (\lambda \kappa_1(\e)-\lambda^2\kappa_2(\e)) \notag\\
&=\overline{J}_{0,s}^{\e,\lambda}+\overline{J}_{s+\e,1}^{\e,\lambda}+\lambda J_{0,s;s+\e,1}^{\e}+\lambda J_{0,s;s,s+\e}^\e+\lambda J_{s,s+\e;s+\e,1}^\e-\e (\lambda \kappa_1(\e)-\lambda^2\kappa_2(\e))\label{eq:jtilde}
 \end{align} 
and  set \begin{align*}
\widetilde{J}_{0,1}^{s,\e,\lambda}:=\overline{J}_{0,s}^{\e,\lambda}+\overline{J}_{s+\e,1}^{\e,\lambda}+\lambda J_{0,s;s+\e,1}^{\e}.
\end{align*}

\begin{rem}
The readers have to be careful not to confuse it with the similar notation $\widetilde{J}_{\e,a}(u,h)$ given just after Corollary \ref{cor:JaJ}.
\end{rem}

\begin{figure}[ht]
\begin{center}
\includegraphics[width=4in,pagebox=cropbox,clip]{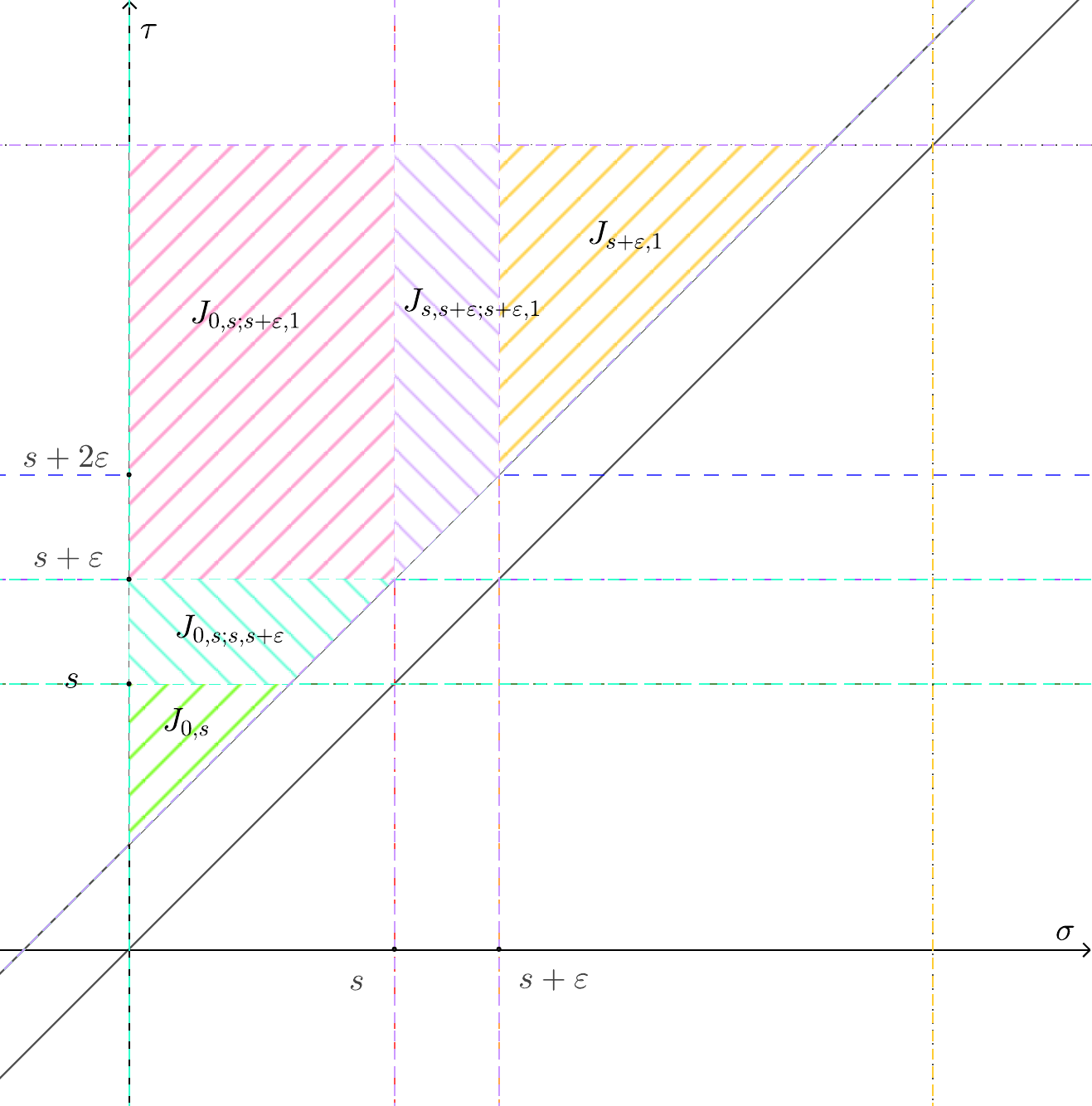}
\caption{A decomposition of $J_{0,1}^\e$.}\label{fig:DecomJ01}
\end{center}
\end{figure}

It is obvious from the definition of $\kappa_1$ in \eqref{eq:kappa1a} and $\kappa _2$ entering \eqref{eq:kappa2} that
\begin{align}
K_{\e,\lambda}:=\e (\lambda \kappa_1(\e)-\lambda^2\kappa_2(\e)) =O(\e^{\frac{1}{2}}).\label{eq:Kelam}
\end{align}
Using the fact that $1-x\leq e^{-x}\leq 1-x+\frac{1}{2}x^2$ for $x>0$, one shows that \begin{align}
E\left[e^{-\overline{J}_{0,1}^{\e,\lambda}}\delta_0 (\omega_s-\omega_{s+\e}) \Phi\right]
&=E\left[e^{-\widetilde{J}_{0,1}^{s,\e,\lambda}}\delta_0 (\omega_s-\omega_{s+\e}) \Phi\right]e^{K_{\e,\lambda}}\notag\\
&\quad -\lambda E\left[e^{-\widetilde{J}_{0,1}^{s,\e,\lambda}}\delta_0 (\omega_s-\omega_{s+\e}) (J_{0,s;s,s+\e}^{\e}+J_{s,s+\e;s+\e,1}^{\e})\Phi\right]e^{K_{\e,\lambda}}\notag\\
&\quad +R_{\lambda,s,\e},\label{eq:Jdecom}
\end{align}
where the remainder term $R_{\lambda,s,\e}$ is dominated by
\begin{align*}
&\frac{\lambda^2}{2}E\left[e^{-\widetilde{J}_{0,1}^{s,\e,\lambda}}\delta_0 (\omega_s-\omega_{s+\e}) (J_{0,s;s,s+\e}^{\e}+J_{s,s+\e;s+\e,1}^{\e})^2|\Phi|\right]e^{K_{\e,\lambda}}.
\end{align*}
Substituting this into \eqref{eq:derofrho}, we have \begin{align}
\frac{\dd }{\dd \e}\rho(\e)&=\lambda\int_0^{1-\e}\dd s E\left[e^{-\widetilde{J}_{0,1}^{s,\e,\lambda}}\delta_0 (\omega_s-\omega_{s+\e}) \Phi\right]e^{K_{\e,\lambda}}\notag\\
&-\lambda^2 \int_0^{1-\e  }\dd s E\left[e^{-\widetilde{J}_{0,1}^{s,\e,\lambda}}\delta_0 (\omega_s-\omega_{s+\e}) (J_{0,s;s,s+\e}^{\e}+J_{s,s+\e;s+\e,1}^{\e})\Phi\right]e^{K_{\e,\lambda}}\notag\\
&\quad +\lambda \int_0^{1-\e}\dd s R_{\lambda,s,\e}+\left(\lambda\frac{\dd}{\dd \e}\kappa_1(\e)-\lambda^2 \frac{\dd}{\dd \e}\kappa_2(\e)\right)\rho(\e). \label{eq:rhoexpand}
\end{align}
We set 
\begin{align*}
&I_{\e,s}^{(1)}:=E\left[e^{-\widetilde{J}_{0,1}^{s,\e,\lambda}}\delta_0(\omega_s-\omega_{s+\e})\Phi\right] ,\\
&I_{\e,s}^{(2)}:=E\left[e^{-\widetilde{J}_{0,1}^{s,\e,\lambda}}\delta_0(\omega_s-\omega_{s+\e})(J_{0,s;s,s+\e}^{\e}+J_{s,s+\e;s+\e,1}^{\e})\Phi\right] ,\\
&I_{\e,s}^{(3)}:=E\left[e^{-\widetilde{J}_{0,1}^{s,\e,\lambda}}\delta_0(\omega_s-\omega_{s+\e})(J_{0,s;s,s+\e}^{\e}+J_{s,s+\e;s+\e,1}^{\e})^2\Phi\right].
\end{align*}
and will prove that there exists $\widetilde{\delta} \in (0,1)$ such that \begin{align}
\int_0^{1-\e}I_{\e,s}^{(1)}\dd s&=-(1-\e)\rho(\e)\frac{\dd}{\dd \e}\kappa_1(\e)+\lambda(1-\e)\e\kappa_1(\e)\frac{\dd}{\dd \e}\kappa_1(\e)\rho(\e) \notag\\ 
&\quad +O(1)\max|f|\e^{\widetilde{\delta}-1}\label{eq:1st}\\
&=-(1-\e)(1-\e\lambda\kappa_1(\e))\rho(\e)\frac{\dd }{\dd\e}\kappa_1(\e)+O(1)\max|f|\e^{\widetilde{\delta}-1} \notag\\
\int_0^{1-\e}I_{\e,s}^{(2)}\dd s&=-\frac{\dd }{\dd \e}\kappa_2(\e)\rho(\e) +O(1)\max|f|\e^{\widetilde{\delta}-1} \label{eq:2nd}\\
\int_0^{1-\e}I_{\e,s}^{(3)}\dd s&\leq  O(1)\max|f|\e^{-\frac{1}{2}} . \label{eq:3rd}
\end{align}
\eqref{eq:1st}-\eqref{eq:3rd} will be proven in Sections \ref{sec:eq3rd}--\ref{sec:eq2nd} below.
Once we obtain \eqref{eq:1st}-\eqref{eq:3rd}, we deduce from \eqref{eq:rhoexpand} that 
\begin{align*}
\frac{\dd }{\dd \e}\rho(\e)=
&-\lambda\left((1-\e)(1-\e\lambda\kappa_1(\e))e^{K_{\e,\lambda}}-1\right)\rho(\e)\frac{\dd }{\dd\e}\kappa_1(\e)\\
&+ \lambda^2\left(e^{K_{\e,\lambda}}-1\right)\frac{\dd }{\dd \e}\kappa_2(\e)\rho(\e)\\
&+O (1)\max|f|\e^{\widetilde{\delta}-1}e^{K_{\e,\lambda}}+		O(1)\max|f|\e^{-\frac{1}{2}}e^{K_{\e,\lambda}}.
\end{align*}
On the other hand, we find from \eqref{eq:kappa1}, \eqref{eq:kappa2}, and \eqref{eq:Kelam} that 
\begin{align*}
(1-\e)(1-\e\lambda\kappa_1(\e))e^{K_{\e,\lambda}}-1&=(1-\e)(1-\e\lambda \kappa _1(\e))\left(1+\e \left(\lambda\kappa _1(\e)-\lambda^2\kappa_2(\e)\right)+O(\e)\right)-1\\
&=O\left(\e\log \frac{1}{\e}\right)\\
e^{K_{\e,\lambda}}-1&=\left(1+O\left(\e^\frac{1}{2}\right)\right)-1=O\left(\e^\frac{1}{2}\right),
\end{align*}
and hence we have
\begin{align*}
\frac{\dd }{\dd \e}\rho(\e)&= O \left(\e\log \frac{1}{\e}\right) \rho(\e)\frac{\dd}{\dd \e}\kappa_1(\e)+\lambda^2 O (1) \e^\frac{1}{2} \rho(\e)\frac{\dd }{\dd \e}\kappa_2(\e)+O(1)\max|f|\e^{-\max\{\frac{1}{2},1-\widetilde{\delta}\}}\\
&=O(1)\max|f|\e^{\widehat{\delta}-1},
\end{align*} 
for some $\widehat{\delta}\in \left(0,\min\left\{\frac{1}{2},\widetilde{\delta}\right\}\right)$, where we have used \eqref{eq:kappa2order}.
From this, we complete the proof of Lemma \ref{lem:finitetra}.
\end{proof}

\subsubsection{Proof of \eqref{eq:3rd}}\label{sec:eq3rd}

\begin{proof}
The Cauchy-Schwarz inequality implies that
\begin{align*}
\left|I_{\e,s}^{(3)}\right| &\leq E\left[e^{-\widetilde{J}_{0,1}^{s,\e,\lambda}}\delta_0(\omega_s-\omega_{s+\e})(J_{0,s;s,s+\e}^{\e}+J_{s,s+\e;s+\e,1}^{\e})^2|\Phi|\right]\\
&\leq 2E\left[e^{-\widetilde{J}_{0,1}^{s,\e,\lambda}}\delta_0 (\omega_s-\omega_{s+\e}) (|J_{0,s;s,s+\e}^{\e}|^2+|J_{s,s+\e;s+\e,1}^{\e}|^2)|\Phi|\right].
\end{align*}
By definition, we know that \begin{align*}
|J_{0,s;s,s+\e}^{\e}|^2=\int_{([0,s]\times [s,s+\e])^2}\prod_{i=1}^2 1\{\tau_i-\sigma_i\geq \e\}\delta _0 (\omega_{\sigma_i}-\omega_{\tau_i})\dd \tau_i\dd \sigma_i
\end{align*}
and by using  Lemma \ref{lem:finitetra}, we obtain that 
\begin{align}
&E\left[e^{-\widetilde{J}_{0,1}^{s,\e,\lambda}}\delta _0(\omega_s-\omega_{s+\e})\delta _0 (\omega_{\sigma_1}-\omega_{\tau_1})\delta _0 (\omega_{\sigma_2}-\omega_{\tau_2})\right]\notag\\
&\leq 
\begin{cases}
\displaystyle C\int_{\R^3}\dd x\int_{\R^3}\dd y\int_{\R^3}\dd z p_{2\sigma_1}(x)p_{2(\sigma_2-\sigma_1)}(x,y)p_{2(s-\sigma_{2})}(y,z) \\
\hspace{7em} \times p_{2(\tau_1-s)}(z,x)p_{2(\tau_2-\tau_1)}(x,y)p_{2(s+\e-\tau_2)}(y,z)\\
\hspace{10em}\text{if }0<\sigma_1<\sigma_2<s<\tau_1<\tau_2<s+\e ; \\[2mm]
\displaystyle C\int_{\R^3}\dd x\int_{\R^3}\dd y\int_{\R^3}\dd z p_{2\sigma_1}(x)p_{2(\sigma_2-\sigma_1)}(x,y)p_{2(s-\sigma_{2})}(y,z)\\
\hspace{7em} \times p_{2(\tau_2-s)}(z,y)p_{2(\tau_1-\tau_2)}(y,x)p_{2(s+\e-\tau_1)}(x,z)\\
\hspace{10em}\text{if }0<\sigma_1<\sigma_2<s<\tau_2<\tau_1<s+\e .
\end{cases}\label{eq:elbound}
\end{align}
By the relations\begin{align*}
&p_s(x)p_t(x)=\frac{1}{(2\pi (s+t))^{3/2}}p_{\frac{st}{s+t}}(x),\\
&\int_{\R^3}p_s(x,y)p_t(y,z)\dd y=p_{s+t}(x,z),
\end{align*}
the right-hand side in \eqref{eq:elbound} is given by \begin{align*}
&C({(\sigma_2-\sigma_1)(\tau_2-\tau_1)}({s+\e-\tau_2+s-\sigma_2})+{({\sigma_2-\sigma_1+\tau_2-\tau_1})(s+\e-\tau_2)(s-\sigma_2)}\\
&+({\sigma_2-\sigma_1+\tau_2-\tau_1})({s+\e-\tau_2+s-\sigma_2})(\tau_1-s))^{-\frac{3}{2}}
\intertext{for $0<\sigma_1<\sigma_2<s<\tau_1<\tau_2<s+\e$ and }
&C{(\tau_2-\sigma_2)(\sigma_2-\sigma_1)(\tau_1-\tau_2)}+{(\sigma_2-\sigma_1+\tau_1-\tau_2)(s-\sigma_2)(\tau_2-s)}\\
&+(\tau_2-\sigma_2)(\sigma_2-\sigma_1+\tau_1-\tau_2)(s+\e-\tau_1))^{-\frac{3}{2}}
\end{align*}
for $0<\sigma_1<\sigma_2<s<\tau_2<\tau_1<s+\e$.
\begin{figure}[ht]
\begin{center}
\includegraphics[width=3in,pagebox=cropbox,clip]{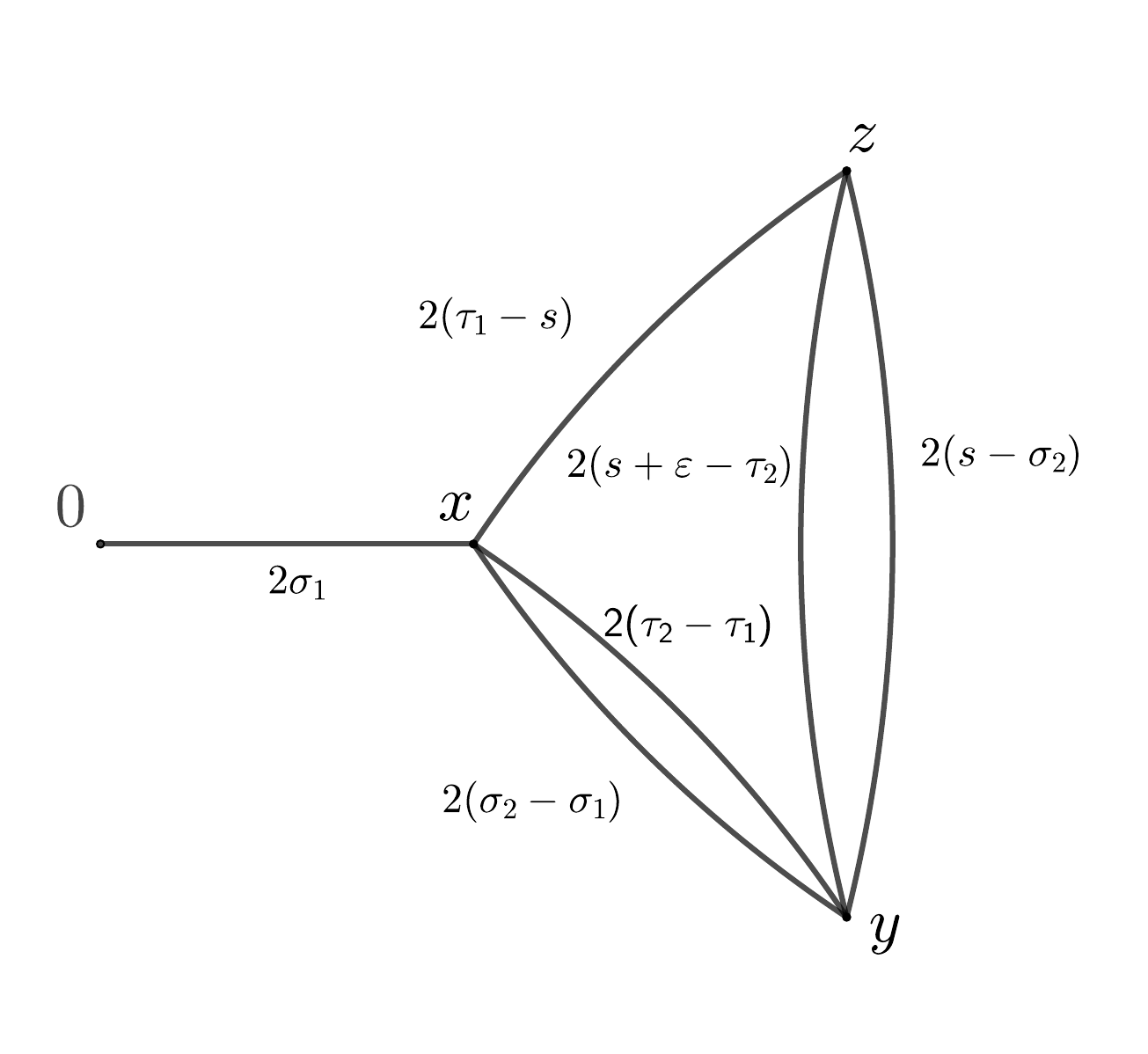}
\caption{The transitions for the case that $0<\sigma_1<\sigma_2<s<\tau_1<\tau_2<s+\e$. When we regard it as an electrical circuit, the computation technique of the resistance helps us to compute the integrals with respect to  the space variables $x,y,z$.}
\end{center}
\end{figure}

Hence, we have
\begin{align*}
&E\left[e^{-\widetilde{J}_{0,1}^{s,\e,\lambda}}\delta_0 (\omega_s-\omega_{s+\e}) |J_{0,s;s,s+\e}^{\e}|^2 |\Phi|\right] \\
&\leq C \max |f| \int _{0<\sigma _1< \sigma _2 <s <\tau _1 <\tau _2 <s+\e} \dd \sigma _1 \dd \tau _1 \dd \sigma _2 \dd \tau _2\\
&\quad \hspace{4cm}[ {(\sigma_2-\sigma_1)(\tau_2-\tau_1)}({s+\e-\tau_2+s-\sigma_2})\\
&\quad \hspace{4cm}+{({\sigma_2-\sigma_1+\tau_2-\tau_1})(s+\e-\tau_2)(s-\sigma_2)}\\
&\quad \hspace{4cm}+({\sigma_2-\sigma_1+\tau_2-\tau_1})({s+\e-\tau_2+s-\sigma_2})(\tau_1-s)]^{-\frac{3}{2}}\\
&\quad + C \max |f| \int _{0<\sigma_1<\sigma_2<s<\tau_2<\tau_1<s+\e}\dd \sigma _1 \dd \tau _1 \dd \sigma _2 \dd \tau _2\\
&\quad \hspace{4cm} [ (\tau_2-\sigma_2)(\sigma_2-\sigma_1)(\tau_1-\tau_2)\\
&\quad \hspace{4cm}+(\sigma_2-\sigma_1+\tau_1-\tau_2)(s-\sigma_2)(\tau_2-s) \\
&\quad \hspace{4cm}+(\tau_2-\sigma_2)(\sigma_2-\sigma_1+\tau_1-\tau_2)(s+\e-\tau_1)]^{-\frac{3}{2}}.
\end{align*}
Changing variables $\tau_i-s=\e t_i$, $s-\sigma_i=\e s_i$ ($i=1,2$), we have
\begin{align}
&E\left[e^{-\widetilde{J}_{0,1}^{s,\e,\lambda}}\delta _0 (\omega_s-\omega_{s+\e})|J_{0,s;s,s+\e}|^2|\Phi|\right]\notag\\
&\leq C\max |f| {\e^{-\frac{1}{2}}} \left( \iint _{0<s_2<s_1 <\infty, 0<t_1 <t_2 < 1}\dd s_1 \dd s_2 \dd t_1 \dd t_2 \right. \notag\\
&\quad \hspace{4cm} [ (s_1-s_2)(t_2-t_1)(1-t_2+s_2) +(s_1-s_2+t_2-t_1)(1-t_2)s_2 \notag\\
&\quad \hspace{8cm}+(s_1-s_2+t_2-t_1)(1-t_2+s_2)t_1]^{-\frac{3}{2}}\notag\\
&\quad + \iint _{0<s_2<s_1<\infty, 0<t_2<t_1<1} \dd s_1 \dd s_2 \dd t_1 \dd t_2 \notag\\
&\quad \hspace{4cm} [ (t_2+s_2)(s_1-s_2)(t_1-t_2)+(s_1-s_2+t_1-t_2) s_2 t_2 \notag\\
&\quad \hspace{7cm} \left. \phantom{\int}+(t_2+s_2)(s_1-s_2+t_1-t_2)(1-t_1)]^{-\frac{3}{2}} \right) \notag\\
& = C{\e^{-\frac{1}{2}}}.\label{eq:I3estimate}
\end{align}
For details see Appendix \ref{app:I3esti}. 

Similarly, we have that for some constant $C>0$,
\begin{align*}
E\left[e^{-\widetilde{J}_{0,1}^{s,\e,\lambda}}\delta _0 (\omega_s-\omega_{s+\e})  |J_{s,s+\e;s+\e,1}^{\e}|^2 |\Phi|\right]\leq C{\e^{-\frac{1}{2}}}.
\end{align*}
In this way we then obtain \eqref{eq:3rd}.
\end{proof}

\subsubsection{Proof of \eqref{eq:1st}}\label{sec:eq1st}

Before starting the proof of \eqref{eq:1st}, we give two propositions.

\begin{proposition}\label{lem:Phi1st}
We take $(s,t)$ and $u,v$ such as \begin{alignat*}{2}
&0\leq s\leq 1-\e,\quad &&1-\e\leq t\leq 1\\
&0\leq u\leq s\leq v\leq t,\quad&& v-u\leq \e.
\end{alignat*}
Then we have \begin{align}
&E\left[\delta _0 (\omega_{v}-\omega_{u})e^{-\overline{J}_{0,u}^{\e,\lambda}-\overline{J}_{v,t}^{\e,\lambda}-\lambda J_{0,u;v,t}^{\e}}\Phi_{0,t} \right]\notag\\
&=p_{v-u}(0)E\left[e^{-\overline{J}_{0,t-(v-u)}^{{\e,\lambda}}-\lambda Y_{u,v}}\Phi_{0,t-(v-u)}\right]+O(1)\max |f|(v-u)^{-\frac{1}{2}},
\end{align}
where
\begin{align*}
&Y_{u,v}:=Y_{u,v}^{t,\e}=\iint_{\left\{\begin{smallmatrix}(u-\e)\vee 0<\sigma<u\\
u<\tau<(u+\e)\wedge (t-(v-u))\\
u-(v-\e)<\tau-\sigma<\e
\end{smallmatrix}\right\}} \dd \sigma\dd\tau \delta _0 (\omega_\tau - \omega_\sigma) .
\end{align*}
\end{proposition}

\begin{figure}[ht]
    \begin{tabular}{cc}
      \begin{minipage}[t]{0.45\hsize}
        \centering
        \includegraphics[keepaspectratio, scale=0.22]{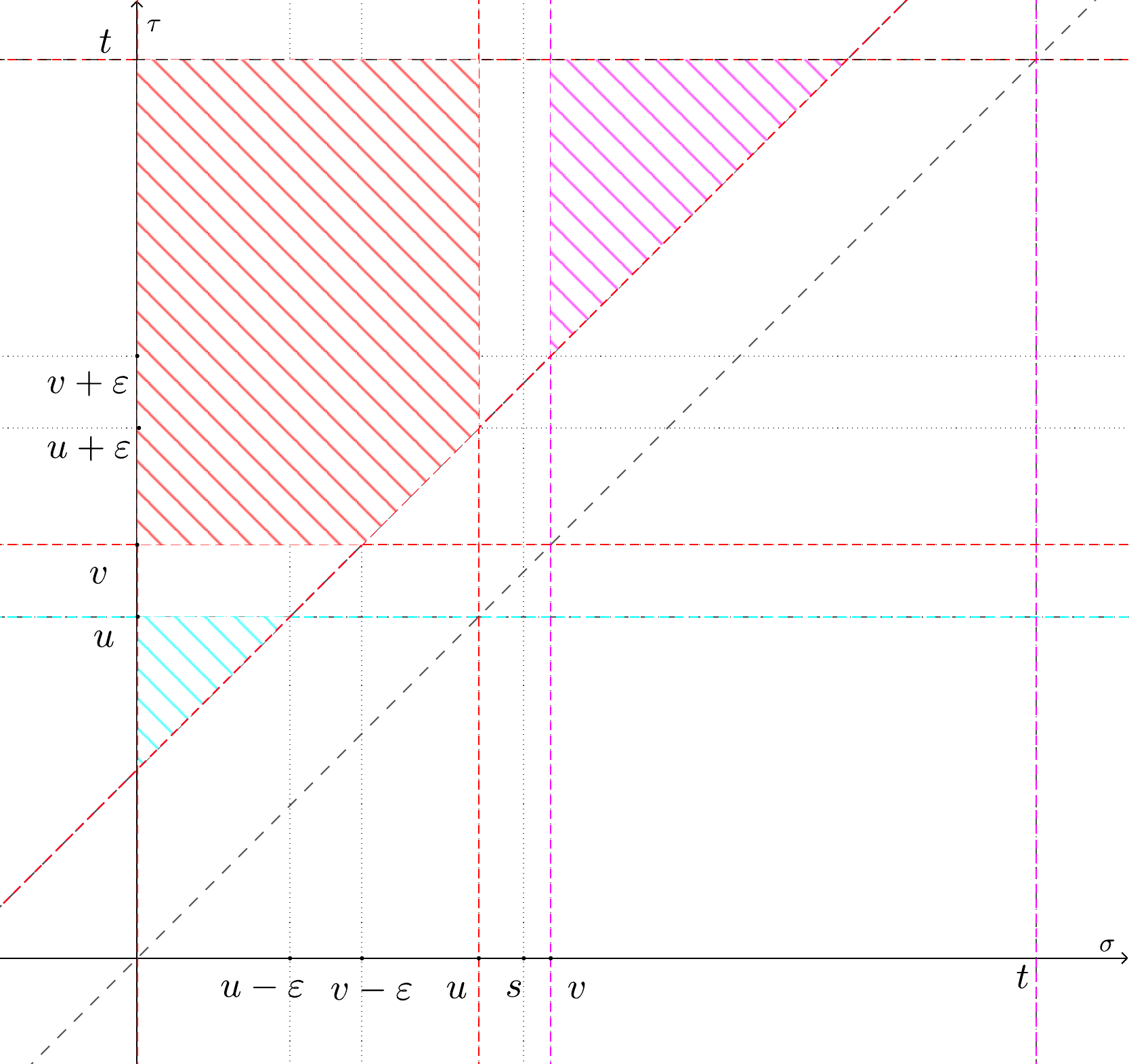}
        \caption{The domains in the integrals  $\overline{J}_{0,u}^{\e,\lambda}$, $\overline{J}_{v,t}^{\e,\lambda}$, and $J_{0,u;v,t}^{\e}$.}
        \label{fig1}
      \end{minipage} &
      \begin{minipage}[t]{0.45\hsize}
        \centering
        \includegraphics[keepaspectratio, scale=0.22]{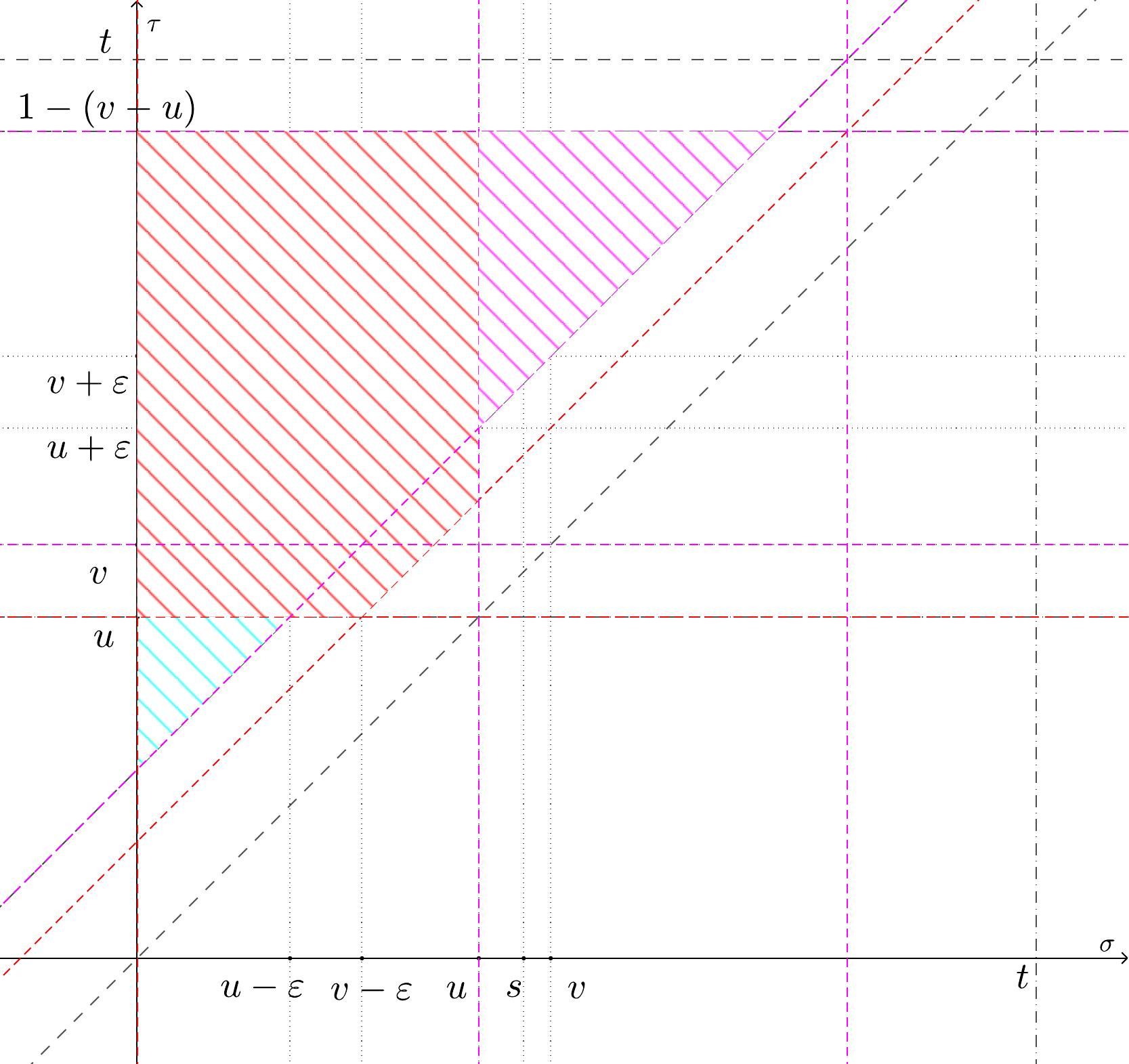}
        \caption{The domains in the integral is shifted by the Markov property of Brownian motion.}
        \label{fig2}
      \end{minipage}
    \end{tabular}
\end{figure}

\begin{proof}
By definition we know that \begin{align*}
|\Phi_{0,u;u,v}|+|\Phi_{u,v;v,t}|+\left|\Phi_{u,v;u,v}\right|\leq C\max|f|(v-u).
\end{align*}
This and Lemma \ref{lem:finitetra} imply that
\begin{align*}
&E\left[\delta _0 (\omega_{v}-\omega_{u})e^{-\overline{J}_{0,u}^{\e,\lambda}-\overline{J}_{v,t}^{\e,\lambda}-\lambda J_{0,u;v,t}^{\e}}\left(|\Phi_{0,u;u,v}|+|\Phi_{u,v;v,t}|+\left|\Phi_{u,v;u,v}\right|\right)\right]\\
&\leq  C\max|f|(v-u) E\left[\delta _0 (\omega_{v}-\omega_{u})e^{-\overline{J}_{0,u}^{\e,\lambda}-\overline{J}_{v,t}^{\e,\lambda}} \right] \\
&\leq  C\max|f|(v-u) \int _{\R ^3} \int _{\R ^3} \dd x \dd y g_u(x) p_{v-u}(0) g_{t-v}(y-x) \\
&\leq  C\max |f|(v-u)^{-\frac{1}{2}}.
\end{align*}
Similarly we have
\begin{align*}
&E\left[\delta _0 (\omega_{v}-\omega_{u})e^{-\overline{J}_{0,u}^{\e,\lambda}-\overline{J}_{v,t}^{\e,\lambda}-\lambda J_{0,u;v,t}^{\e}}\left|\int_{T_a\cap \left\{\begin{smallmatrix}(u-a)\vee 0<\sigma<u\\ v<\tau<(v+a)\wedge t\\
\tau-\sigma\leq v-u+a\end{smallmatrix}\right\}}f(\omega_\tau-\omega_\sigma)\dd \tau\dd \sigma\right|\right]\\
&\leq C\max |f| (v-u)^{-\frac{1}{2}}. 
\end{align*}
We denote $\Theta_{a,u,v}:=\int_{T_a\cap \left\{\begin{smallmatrix}(u-a)\vee 0<\sigma<u\\ v<\tau<(v+a)\wedge t\\
\tau-\sigma\leq v-u+a\end{smallmatrix}\right\}}f(\omega_\tau-\omega_\sigma)\dd \tau\dd \sigma$. (See Figure \ref{fig:DecPhi} below.)
\begin{figure}[ht]
\begin{center}
\includegraphics[width=4in,pagebox=cropbox,clip]{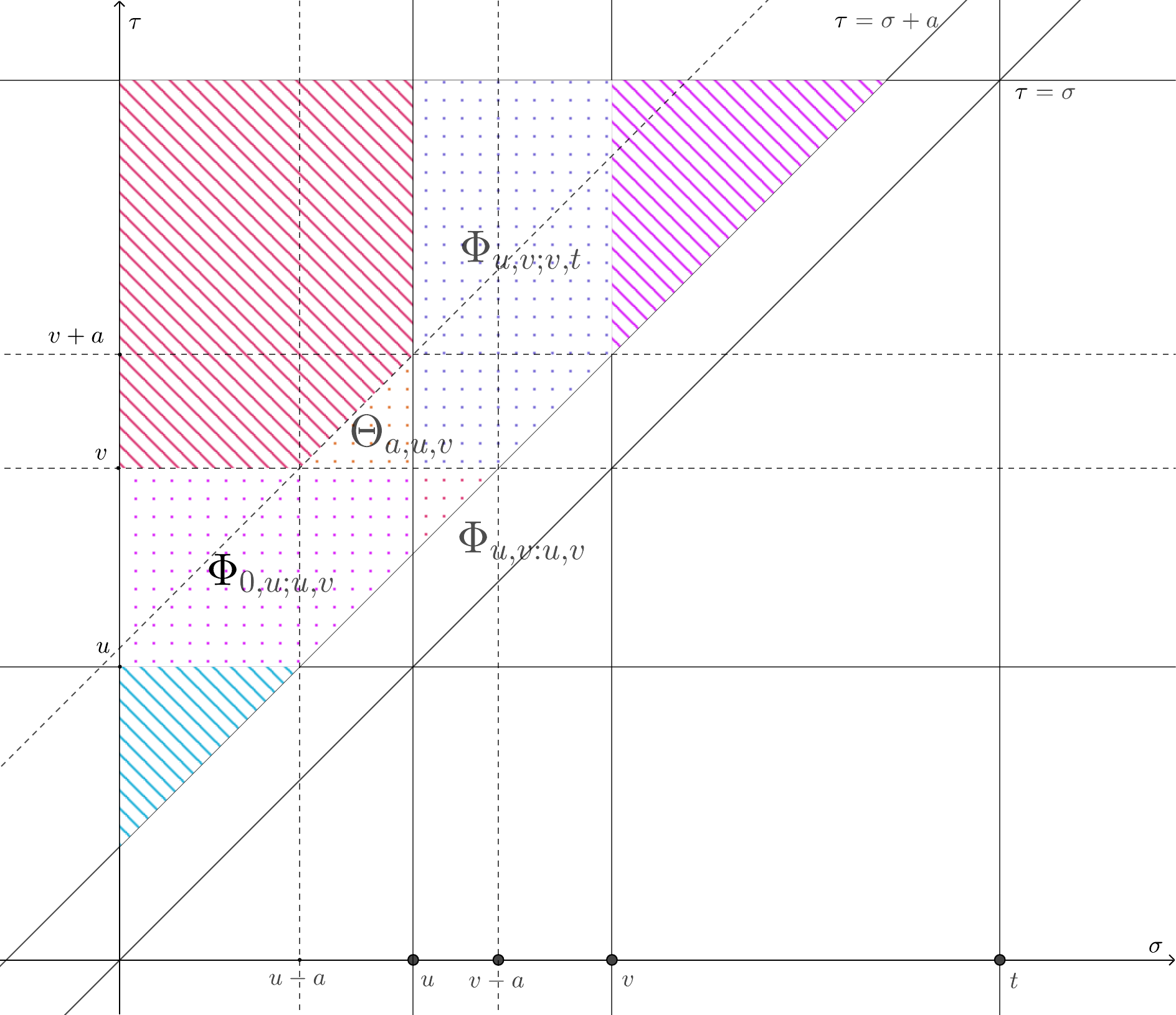}
\caption{Decomposition of the integral of $\Phi$.}\label{fig:DecPhi}
\end{center}
\end{figure}
This shows that
\begin{align*}
&E\left[\delta _0 (\omega_{v}-\omega_{u})e^{-\overline{J}_{0,u}^{\e,\lambda}-\overline{J}_{v,t}^{\e,\lambda}-\lambda J_{0,u;v,t}^{\e}}\Phi_{0,t} \right]\\
&= O(1)\max |f|(v-u)^{-\frac{1}{2}}\\
&\quad +E\left[e^{-\overline{J}_{0,u}^{\e,\lambda}-\overline{J}_{v,t}^{\e,\lambda}-\lambda J_{0,u;v,t}^{\e}}\delta _0 (\omega_v-\omega_{u})\left(\Phi_{0,t}-\Phi_{0,u;u,v}-\Phi_{u,v;v,t}-\Phi_{u,v;u,v}-\Theta_{a,u,v}\right)\right]\\
&= O(1)\max |f|(v-u)^{-\frac{1}{2}}\\
&\quad +E\left[e^{-\overline{J}_{0,u}^{\e,\lambda}-\overline{J}_{v,t}^{\e,\lambda}-\lambda J_{0,u;v,t}^{\e}}\delta _0 (\omega_v-\omega_{u})\left(\Phi_{0,u}+\Phi_{v,t}+\Phi_{0,u;v,t}-\Theta_{a,u,v}\right)\right].
\end{align*}
By the Markov property of Brownian motion and the shift invariance of $J$ we have that conditioned on $\omega_v=\omega_u$, the following equality in distribution holds:
\begin{align*}
&(\overline{J}_{0,u}^{\e,\lambda}+\overline{J}_{v,t}^{\e,\lambda}+\lambda J_{0,u;v,t}^{\e}, \Phi_{0,u},\Phi_{v,t},\Phi_{0,u;v,t}-\Theta_{a,u,v})\\
&\stackrel{d}{=}(\overline{J}_{0,t-(v-u)}^{\e,\lambda}+\lambda Y_{u,v}, \Phi_{0,u},\Phi_{u,t-(v-u)},\Phi_{0,u;u,t-(v-u)}).
\end{align*}
From this we have
\begin{align*}
&E\left[e^{-\overline{J}_{0,u}^{\e,\lambda}-\overline{J}_{v,t}^{\e,\lambda}-\lambda J_{0,u;v,t}^{\e}}\delta _0 (\omega_v-\omega_{u})\left(\Phi_{0,u}+\Phi_{v,t}+\Phi_{0,u;v,t}-\Theta_{a,u,v}\right)\right]\\
&=p_{v-u}(0)E\left[e^{-\overline{J}_{0,t-(v-u)}^{\e,\lambda}-\lambda Y_{u,v}} ( \Phi_{0,u}+\Phi_{u,t-(v-u)}+\Phi_{0,u;u,t-(v-u)} ) \right]\\
&=p_{v-u}(0)E\left[e^{-\overline{J}_{0,t-(v-u)}^{\e,\lambda}-\lambda Y_{u,v}} \Phi_{0,t-(v-u)} \right].
\end{align*}
From this equality and the inequalities above we obtain the desired estimate of this proposition. 
\end{proof}

\begin{lemma}[{cf. \cite[(4.6)]{Bol93}}]\label{lem:Phir}
For any $\lambda>0$, there exist $C>0$ and $\widetilde{\delta} >0$ such that for sufficiently small $\e >0$ and $0<r\leq 2\e$,
\begin{align*}
\left|E\left[e^{-\bar{J}_{0,1-r}^{\e ,\lambda}}\Phi\right]-E\left[e^{-\bar{J}_{0,1}^{\e ,\lambda}}\Phi\right]\right| \leq C\max|f|\e^{\widetilde{\delta} -\frac{1}{2}}.
\end{align*}
\end{lemma}

\begin{proof}
This estimate is proved similarly to \cite[(4.6)]{Bol93} by using \eqref{eq:partpest} instead of \cite[Proposition 3.1]{Bol93}. 
We omit the detailed proof.
\end{proof}

\begin{proof}[Proof of \eqref{eq:1st}]
Applying Proposition \ref{lem:Phi1st} for $u=s$, $v=s+\e$, and $t=1$, we have 
\begin{align*}
I_{\e,s}^{(1)}
&=O(1)\max|f|\e^{-\frac{1}{2}}+p_\e(0)E\left[e^{-\overline{J}_{0,1-\e}^{\e,\lambda}-\lambda Y_{s,s+\e}} \Phi_{0,1-\e}
\right]\\
&=O(1)\max|f|\e^{-\frac{1}{2}} +p_\e(0)\left(E\left[e^{-\overline{J}_{0,1-\e}^{\e,\lambda}}\Phi_{0,1-\e}\right]-\lambda E\left[Y_{s,s+\e}e^{-\overline{J}_{0,1-\e}^{\e,\lambda}}\Phi_{0,1-\e}\right] \right. \\
&\hspace{8cm} \left. +O(1)\lambda^2 E\left[Y_{s,s+\e}^2e^{-\overline{J}_{0,1-\e}^{\e,\lambda}}\Phi_{0,1-\e}\right]\right)\\
&=:O(1)\max|f|\e^{-\frac{1}{2}}+p_\e(0)\left(A_{\e}^{(1)}-\lambda A_{\e,s}^{(2)}+O(1)\lambda^2 A_{\e,s}^{(3)}\right).
\end{align*}
Here, we note that $p_\e(0)=-\frac{\dd}{\dd \e}\kappa_1(\e)$.
Since  \begin{align*}
A_{\e}^{(1)}=E\left[e^{-\overline{J}_{0,1-\e}^{\e,\lambda}}\Phi_{0,1}\right]+O(1)\max|f| \e,
\end{align*}
Lemma \ref{lem:Phir} gives that
\begin{align*}
A_{\e}^{(1)}=\rho(\e)+O(1)\max |f|\e^{\frac{1}{2}}.
\end{align*}
By Lemma \ref{lem:finitetra}, we have
\begin{align}
A_{\e,s}^{(3)}=E\left[Y_{s,s+\e}^2e^{-\overline{J}_{0,1-\e}^{\e,\lambda}} \Phi_{0,1-\e}\right]\leq C\max|f|\e \label{eq:A3}
\end{align}
(see Appendix \ref{app:A3}) and hence, 
\begin{align*}
p_\e(0)\lambda^2 A_{\e,s}^{(3)}\leq C\max|f|\e^{-\frac{1}{2}}
\end{align*}
for $0\leq s\leq 1-\e$.

In the rest, we will prove that \begin{align*}
A_{\e,s}^{(2)}=\begin{cases}
O(1)\max |f|\e^\frac{1}{2}, &1-2\e\leq s\leq 1-\e\\
\e\left(\kappa_1(\e)+\frac{2}{(2\pi)^{\frac{3}{2}}}\right)\rho(\e)+O(1)\max|f|\e,\quad &\e\leq s\leq 1-2\e\\
\frac{2}{(\sqrt{2\pi})^3}\left(2\sqrt{s}-\frac{s}{\sqrt{\e}}\right)\rho(\e)+O(1)\max|f|\e,\quad &0\leq s\leq \e
\end{cases}
\end{align*}
which implies that
\begin{align*}
\int_0^{1-\e}A_{\e,s}^{(2)}\dd s=(1-\e)\e\kappa_1(\e)\rho(\e)+O(1)\max|f|\e,
\end{align*}
which then completes the proof of \eqref{eq:1st}.

Since \begin{align*}
E\left[Y_{s,s+\e}e^{-\overline{J}_{0,1-\e}^{\e,\lambda}}|\Phi_{0,1-\e}-\Phi_{0,1}|\right]\leq C\max |f|\e^{\frac{3}{2}},
\end{align*}
we have \begin{align*}
&A_{\e,s}^{(2)}=E\left[Y_{s,s+\e}e^{-\overline{J}_{0,1-\e}^{\e,\lambda}}\Phi_{0,1-\e}\right]\\
&=E\left[Y_{s,s+\e}e^{-\overline{J}_{0,1-\e}^{\e,\lambda}}\Phi_{0,1}\right]+O(1)\max|f|\e^{\frac{3}{2}}\\
&=E\left[\int\int_{0\leq u\leq s\leq v\leq 1-\e,v-u\leq \e}\dd u\dd v\delta_0(\omega_v-\omega_u)e^{-\overline{J}_{0,1-\e}^{\e,\lambda}}\Phi_{0,1}\right]+O(1)\max|f|\e^{\frac{3}{2}}\\
&=:A_{\e,s}^{(2,1)}+O(1)\max|f|\e^{\frac{3}{2}}.
\end{align*}

As in \eqref{eq:Jdecom}, we have
\begin{align*}
&A_{\e,s}^{(2,1)}\\
&=E\left[\int\int_{0\leq u\leq s\leq v\leq 1-\e,v-u\leq \e}\dd u\dd v\delta_0(\omega_v-\omega_u)e^{-\overline{J}_{0,1-\e}^{\e,\lambda}}\Phi_{0,1}\right]\\
&=E\left[\int\int_{0\leq u\leq s\leq v\leq 1-\e,v-u\leq \e}\dd u\dd v\delta_0(\omega_{v}-\omega_{u})e^{-\overline{J}_{0,u}^{\e,\lambda}-\overline{J}_{v,1-\e}^{\e,\lambda}-\lambda J_{0,u;v,1-\e}^{\e}}\Phi_{0,1} e^{K_{\e,v-u,\lambda}} \right]\\
&\quad -O(1)\lambda E\left[\int\int_{ 0\leq u\leq s\leq v\leq 1-\e, v-u\leq \e}\dd u\dd v\delta_0(\omega_{v}-\omega_{u})e^{-\overline{J}_{0,u}^{\e,\lambda}-\overline{J}_{v,1-\e}^{\e,\lambda}-\lambda J_{0,u;v,1-\e}^{\e}}\Phi_{0,1}\right. \\
&\hspace{25em} \times  (J_{0,u;u,v}^\e+J_{u,v;v,1-\e}^\e)e^{K_{\e,v-u,\lambda}} \bigg]\\
&=:A_{\e,s}^{(2,1,1)}-O(1)\lambda A_{\e,s}^{(2,1,2)},
\end{align*}
where \begin{align*}
K_{\e,\theta,\lambda}:=\theta (\lambda \kappa_1(\e)-\lambda^2\kappa_2(\e)).
\end{align*}

Using the same argument as in the proof of \eqref{eq:3rd}, one proves that
\begin{align}\label{eq:A212}
&A_{\e,s}^{(2,1,2)}
= O(1)\max |f|\e
\end{align}
for $0\leq s\leq 1-\e$ (see Appendix \ref{A212}).

Applying Proposition \ref{lem:Phi1st}, we get that
\begin{align*}
&A_{\e,s}^{(2,1,1)}\\
&=E\left[\iint_{0\leq u\leq s\leq v\leq 1-\e,v-u\leq \e}\delta_0(\omega_{v}-\omega_{u})e^{-\overline{J}_{0,u}^{\e,\lambda}-\overline{J}_{v,1-\e}^{\e,\lambda}-\lambda J_{0,u;v,1-\e}^{\e}}\Phi_{0,1}  e^{K_{\e,v-u,\lambda}}\right]\\
&=\iint_{0\leq u\leq s\leq v\leq 1-\e,v-u\leq \e}p_{v-u}(0)E\left[e^{-\overline{J}_{0,1-\e-(v-u)}^{{\e,\lambda}}-\lambda Y_{u,v}}\Phi_{0,1-\e-(v-u)}\right]e^{K_{\e,v-u,\lambda}}\dd u\dd v\\
&\hspace{5em}+O(1)\max|f|\e^{\frac{3}{2}}\\
&=\iint_{0\leq u\leq s\leq v\leq 1-\e,v-u\leq \e}p_{v-u}(0)E\left[e^{-\overline{J}_{0,1-\e-(v-u)}^{{\e,\lambda}}}\Phi_{0,1-\e-(v-u)}\right]e^{K_{\e,v-u,\lambda}}\dd u\dd v\\
&-O(1)\lambda\iint_{0\leq u\leq s\leq v\leq 1-\e,v-u\leq \e}p_{v-u}(0)E\left[e^{-\overline{J}_{0,1-\e-(v-u)}^{{\e,\lambda}}}\Phi_{0,1-\e-(v-u)}Y_{u,v}\right]e^{K_{\e,v-u,\lambda}}\dd u\dd v\\
&+O(1)\max|f|\e^{\frac{3}{2}}\\
&=:A_{\e,s}^{(2,1,1,1)}-O(1)\lambda A_{\e,s}^{(2,1,1,2)}
+O(1)\max|f|\e^{\frac{3}{2}}.
\end{align*}
By Lemma \ref{lem:finitetra}, we know that \begin{align*}
A_{\e,s}^{(2,1,1,2)}&\leq C\iint_{0\leq u\leq s\leq v\leq 1-\e,v-u\leq \e}p_{v-u}(0)E\left[e^{-\overline{J}_{0,1-\e-(v-u)}^{{\e,\lambda}}}\Phi_{0,1-\e-(v-u)}Y_{u,v}\right]\dd u\dd v\\
&=O(1)\max|f|\e
\end{align*}
for $0\leq s\leq 1-\e$.
Since \begin{align*}
\iint_{0\leq u\leq s\leq v\leq 1-\e,v-u\leq \e}p_{v-u}(0)\dd u\dd v=\begin{cases}
4\sqrt{1-\e-s}-2\frac{1-\e-s}{\sqrt{\e}},\quad &1-2\e\leq s\leq 1-\e\\
\e\kappa_1(\e)+\frac{2}{(2\pi)^{\frac{3}{2}}}\e,\quad &\e\leq s\leq 1-2\e\\
\frac{2}{(\sqrt{2\pi})^3}\left(2\sqrt{s}-\frac{s}{\sqrt{\e}}\right),\quad &0\leq s\leq \e
\end{cases}
\end{align*}
and \begin{align*}
E\left[e^{-\overline{J}_{0,1-\e-(v-u)}^{{\e,\lambda}}}\left|\Phi_{0,1-\e-(v-u)}-\Phi_{0,1}\right|\right]=O(1)\max |f|\e,\quad \text{for }0<v-u<\e,
\end{align*}
Lemma \ref{lem:Phir} yields that \begin{align*}
A_{\e,s}^{(2,1,1,1)}=\begin{cases}
O(1)\max|f|\e^\frac{1}{2}&1-2\e\leq s\leq 1-\e\\
\e\left(\kappa_1(\e)+\frac{2}{(2\pi)^{\frac{3}{2}}}\right)\rho(\e)+O(1)\max|f|\e,\quad &\e\leq s\leq 1-2\e\\
\frac{2}{(\sqrt{2\pi})^3}\left(2\sqrt{s}-\frac{s}{\sqrt{\e}}\right)\rho(\e)+O(1)\max|f|\e\quad &0\leq s\leq \e
\end{cases}
\end{align*}
and \begin{align*}
A_{\e,s}^{(2,1,1)}=\begin{cases}
O(1)\max|f|\e^\frac{1}{2},&1-2\e\leq s\leq 1-\e\\
\e\left(\kappa_1(\e)+\frac{2}{(2\pi)^{\frac{3}{2}}}\right)\rho(\e)+O(1)\max|f|\e,\quad &\e\leq s\leq 1-2\e\\
\frac{2}{(\sqrt{2\pi})^3}\left(2\sqrt{s}-\frac{s}{\sqrt{\e}}\right)\rho(\e)+O(1)\max|f|\e\quad &0\leq s\leq \e.
\end{cases}
\end{align*}
Putting things together, we obtain that \begin{align*}
A^{(2)}_{\e,s}=\begin{cases}
O(1)\max|f|\e^\frac{1}{2},&1-2\e\leq s\leq 1-\e\\
\e\left(\kappa_1(\e)+\frac{2}{(2\pi)^{\frac{3}{2}}}\right)\rho(\e)+O(1)\max|f|\e,\quad &\e\leq s\leq 1-2\e\\
\frac{2}{(\sqrt{2\pi})^3}\left(2\sqrt{s}-\frac{s}{\sqrt{\e}}\right)\rho(\e)+O(1)\max|f|\e \quad &0\leq s\leq \e.
\end{cases}
\end{align*}
\end{proof}

\subsubsection{Proof of \eqref{eq:2nd}}\label{sec:eq2nd}

Let \begin{align*}
&I_{\e,s}^{(2,1)}:=E\left[e^{-\widetilde{J}_{0,1}^{s,\e,\lambda}}\delta_0 (\omega_s-\omega_{s+\e}) J_{0,s;s,s+\e}^{\e} \Phi\right] ,\\
&I_{\e,s}^{(2,2)}:=E\left[e^{-\widetilde{J}_{0,1}^{s,\e,\lambda}}\delta_0 (\omega_s-\omega_{s+\e}) J_{0,s+\e;s+\e,1}^{\e} \Phi\right] .
\end{align*}
As in \cite[\S 4]{Bol93}, we choose $1>\beta>\gamma>\frac{1}{2}>\alpha>0$. 
Letting 
\begin{align*}
T^1(\e^\alpha)&:=\{(s,u,v):s<\e^\alpha,0<u<s<v<s+\e,v-u\geq \e\} \\
&\qquad \cup \{(s,u,v):s>1-\e^\alpha,0<u<s<v<s+\e,v-u\geq \e\} , \\
T^2(\e^\alpha,\e^\beta) &:=\{(s,u,v):\e^\alpha<s<1-\e^\alpha,0<u<s-\e^\beta<s<v<s+\e,v-u\geq \e\} , 
\end{align*}
we find that
\begin{align*}
&\int_{T^1(\e^\alpha)}\dd s\dd u\dd v E\left[e^{-\widetilde{J}_{0,1}^{s,\e,\lambda}}\delta_0(\omega_s-\omega_{s+\e}) \delta_0 (\omega_u-\omega_v) |\Phi|\right]\\
&\leq C\max|f|\int_{T^1(\e^\alpha)}\dd s\dd u\dd v \int \dd x\dd yp_{2u}(x)p_{2(s-u)}(x,y)p_{2(v-s)}(y,x)p_{2(s+\e-v)}(x,y)\\
&\leq C\max|f| \e^{\alpha-1}
\end{align*}
and
\begin{align*}
&\int_{T^2(\e^\alpha,\e^\beta)}\dd s\dd u \dd v E\left[ e^{-\widetilde{J}_{0,1}^{s,\e,\lambda}} \delta_0 (\omega_s-\omega_{s+\e}) \delta_0 (\omega_v-\omega_u) |\Phi | \right] \\
&\leq C\max|f|\int_{T^2(\e^\alpha,\e^\beta)}\dd s\dd u \dd v \int \dd x\dd yp_{2u}(x)p_{2(s-u)}(x,y)p_{2(v-s)}(y,x)p_{2(s+\e-v)}(x,y)\\
&\leq C\max|f|\e^{-\frac{\beta+1}{2}}.
\end{align*}  
Therefore, we have 
\begin{align*}
\int_0^{1-\e}\dd sI_{\e,s}^{(2,1)} &=\int_{0}^{1-\e}\dd s\int_{u<s<v<s+\e,v-u>\e}\dd u\dd vE\left[ e^{-\widetilde{J}_{0,1}^{s,\e,\lambda}} \delta_0 (\omega_s-\omega_{s+\e}) \delta_0 (\omega_u-\omega_v) \Phi\right]\\
&=\int_{\e^\alpha}^{1-\e^\alpha}\dd s\int_{s-\e^\beta}^s\dd u\int_{s\vee (u+\e)}^{s+\e}\dd v E\left[ e^{-\widetilde{J}_{0,1}^{s,\e,\lambda}} \delta_0 (\omega_s-\omega_{s+\e}) \delta_0 (\omega_u-\omega_v) \Phi\right] \\
&\quad +\max|f|O\left(\e^{(\alpha-1)\wedge \left(-\frac{\beta+1}{2}\right)}\right).
\end{align*}
We write
\begin{align*}
T^3 (\e^\alpha,\e^\beta) &:=\{(s,u,v):\e^\alpha<s<1-\e^\alpha,s-\e^\beta<u<s<v<s+\e,v-u\geq \e\} \\
&=\{(s,u,v):\e^\alpha<s<1-\e^\alpha,s-\e^\beta<u<s\vee (u+\e)<v<s+\e\}.
\end{align*}
For $\e^\gamma\leq s\leq 1-\e^\gamma$, we set
\begin{align*}
\check{J}_{0,1}^{s,\e,\lambda} := \overline{J}_{0,s-\e^\gamma}^{\e,\lambda} +\overline{J}_{s+\e^\gamma,1}^{\e,\lambda} +\lambda J_{0,s-\e^\gamma;s+\e^\gamma,1}^{\e}.
\end{align*}
Let us recall that $1>\beta>\gamma>\frac{1}{2}>\alpha$.
We note then that the deterministic part in 
$\widetilde{J}_{0,1}^{s,\e,\lambda} - \check{J}_{0,1}^{s,\e,\lambda} $ derived from $\kappa_1$ and $\kappa_2$ is of order $\e^{\gamma-\frac{1}{2}} $ and the random part comes from the intersection local time is positive and dominated by \begin{align}
\lambda \left(J_{0,1}^\e -J_{0,s-\e^\gamma}^{\e}- J_{s+\e^\gamma,1}^\e-J^\e_{0,s-\e^\gamma;s+\e^\gamma,1}\right)&=\lambda \left(J^\e_{0,s-\e^\gamma;s-\e^\gamma,s+\e^\gamma}+J^\e_{s-\e^\gamma,s+\e^\gamma;s+\e^\gamma,1}+J_{s-\e^\gamma,s+\e^\gamma}^\e\right)\notag\\
&= \lambda \left(J^\e_{0,s+\e^\gamma;s-\e^\gamma,s+\e^\gamma}+J^\e_{s-\e^\gamma,s+\e^\gamma;s+\e^\gamma,1}\right)\label{eq:JJhatdiff}
\end{align}
Thus, we have \begin{align*}
&\int_{T^3_{s,u,v}(\e^\alpha,\e^\beta)}\dd s\dd u\dd v E\left[e^{-\widetilde{J}_{0,1}^{s,\e,\lambda}} \delta_0(\omega_s-\omega_{s+\e})\delta_{0}(\omega_u-\omega_v)\Phi\right]
\\
&=\int_{T^3_{s,u,v}(\e^\alpha,\e^\beta)}\dd s\dd u\dd v E\left[e^{-\check{J}_{0,1}^{s,\e,\lambda}} \delta_0(\omega_s-\omega_{s+\e})\delta_0(\omega_u-\omega_v)\Phi\right]\\
&\quad+O\left(\e^{\gamma-\frac{1}{2}}\right)\int_{T^3_{s,u,v}(\e^\alpha,\e^\beta)}\dd s\dd u\dd v E\left[e^{-\check{J}_{0,1}^{s,\e,\lambda}} \delta_0(\omega_s-\omega_{s+\e})\delta_0(\omega_u-\omega_v)|\Phi|\right]\\
&\quad+O(1) \int_{T^3_{s,u,v}(\e^\alpha,\e^\beta)}\dd s\dd u\dd v E\left[e^{-\check{J}_{0,1}^{s,\e,\lambda}} J_{0,s+\e^\gamma;s-\e^\gamma,s+\e^\gamma}^\e\delta_0({\omega_s-\omega_{s+\e}})\delta_0({\omega_u-\omega_v})|\Phi|\right]\\
&\quad +O(1) \int_{T^3_{s,u,v}(\e^\alpha,\e^\beta)}\dd s\dd u\dd v E\left[e^{-\check{J}_{0,1}^{s,\e,\lambda}} J_{s-\e^\gamma,s+\e^\gamma;s+\e^\gamma,1}^\e\delta_0({\omega_s-\omega_{s+\e}})\delta_0({\omega_u-\omega_v})|\Phi|\right].
\end{align*}
Moreover, we can see from Appendix \ref{app:Rem} that there exists $\delta\in (0,1)$ such that
\begin{align}
\int_{T^3_{s,u,v}(\e^\alpha,\e^\beta)}\dd s\dd u\dd v E\left[e^{-\check{J}_{0,1}^{s,\e,\lambda}} J^\e_{0,s+\e^\gamma;s-\e^\gamma,s+\e^\gamma}\delta_0({\omega_s-\omega_{s+\e}})\delta_0({\omega_u-\omega_v})|\Phi|\right]=O(\e^{\delta-1})\label{eq:Remainder1}\\
\int_{T^3_{s,u,v}(\e^\alpha,\e^\beta)}\dd s\dd u\dd v E\left[e^{-\check{J}_{0,1}^{s,\e,\lambda}} J^\e_{s-\e^\gamma,s+\e^\gamma;s-\e^\gamma,1}\delta_0({\omega_s-\omega_{s+\e}})\delta_0({\omega_u-\omega_v})|\Phi|\right]=O(\e^{\delta-1}).\label{eq:Remainder2}
\end{align}
Thus, we have 
\begin{align}
\int_0^{1-\e}\dd sI_{\e,s}^{(2,1)} &=\left(1+O\left(\e^{\gamma-\frac{1}{2}}\right)\right) \int_{T^3(\e^\alpha,\e^\beta)}\dd s\dd u\dd v E\left[ e^{-\check{J}_{0,1}^{s,\e,\lambda}} \delta_0 (\omega_s-\omega_{s+\e}) \delta_0 (\omega_u-\omega_v) \Phi\right] \label{eq:I3-1}\\
&\quad +\max|f|O(\e^{\widetilde{\delta}-1}). \notag
\end{align}
for some $\widetilde{\delta} >0$.

Also, we can see from \eqref{eq:kappa2order2} that 
\begin{align}
&\int_{T^3(\e^\alpha,\e^\beta)}\dd s\dd u\dd v E\left[ e^{-\check{J}_{0,1}^{s,\e,\lambda}} \delta_0 (\omega_s-\omega_{s+\e}) \delta_0 (\omega_u-\omega_v) \Phi\right] \notag\\
&=\int_{T^3(\e^\alpha,\e^\beta)}\dd s\dd u\dd v E\left[ e^{-\check{J}_{0,1}^{s,\e,\lambda}} \delta_0 (\omega_s-\omega_{s+\e}) \delta_0 (\omega_u-\omega_v) \left(\Phi_{0,s-\e^\gamma}+\Phi_{s+\e^\gamma,1}+\Phi_{0,s-\e^\gamma;s+\e^\gamma,1}\right)\right] \notag\\
&\quad +\max|f|O(\e^{\gamma-1})\notag\\
&=\int_{T^3(\e^\alpha,\e^\beta)}\dd s\dd u\dd v\iint \dd x\dd y E\left[ e^{-\check{J}_{0,1}^{s,\e,\lambda}} \delta_0 (\omega_s-\omega_{s+\e}) \delta_0 (\omega_u-\omega_v) \right.\notag\\
&\hspace{7em}\left. \phantom{e^{-\check{J}_{0,1}^{s,\e,\lambda}}} \times \delta_0(x-\omega_{s-\e^\gamma})\delta_0({y}-\omega_{s+\e^\gamma}) \left(\Phi_{0,s-\e^\gamma}+\Phi_{s+\e^\gamma,1}+\Phi_{0,s-\e^\gamma;s+\e^\gamma,1}\right)\right] \notag\\
&\quad +\max|f|O(\e^{\gamma-1})\notag\\
&=\int_{T^3(\e^\alpha,\e^\beta)}\dd s\dd u\dd v \iint \dd x\dd yE\left[\left.  e^{-\check{J}_{0,1}^{s,\e,\lambda}} \delta_0(x-\omega_{s-\e^\gamma}) \right. \right. \notag\\
&\quad \hspace{7em}\left. \left. \phantom{e^{-\check{J}_{0,1}^{s,\e,\lambda}}} \times \left(\Phi_{0,s-\e^\gamma}+\Phi_{s+\e^\gamma,1}+\Phi_{0,s-\e^\gamma;s+\e^\gamma,1}\right) \right|\omega_{s+\e^\gamma}=y\right]q_{u,s,v}(x,y)\notag\\
&\quad +\max|f|O(\e^{\gamma-1}),\label{eq:I3-2}
\end{align}
where
\begin{align*}
q_{u,s,v}(x,y) &:=\int \dd z\int \dd w p_{u-(s-\e^\gamma)}(x,z)p_{s-u}(z,w)p_{v-s}(w,z)p_{s+\e-v}(z,w)p_{(s+\e^\gamma)-(s+\e)}(w,y).
\end{align*}
An explicit calculation implies
\begin{align*}
q_{u,s,v}(x,y)=\tau p_{2\e^\gamma-\e-(s-u)+\sigma}(x,y),
\end{align*}
where 
\begin{align*}
\tau &:=\frac{1}{(2\pi)^3}((s+\e-v)(v-s)+(s+\e-v)(s-u)+(v-s)(s-u))^{-\frac{3}{2}},\\
\sigma&:=\frac{(s+\e-v)(v-s)(s-u)}{(s+\e-v)(v-s)+(s+\e-e)(s-u)+(v-s)(s-u)}\leq s-u.
\end{align*} 
In particular, we have $ 2\e^\gamma-\e^\beta-\e\leq 2\e^\gamma-\e-(s-u)+\sigma\leq 2\e^\gamma-\e$.
Let
\begin{align*}
r_{u,s,,v}(x,y):=q_{u,s,v}(x,y)-\tau p_{2\e^\gamma}(x,y).
\end{align*}
Then, we derive in a similar way to \eqref{eq:Remainder1}-\eqref{eq:Remainder2} that
\begin{align}
&\int_{T^3(\e^\alpha,\e^\beta)}\dd s\dd u\dd v \iint \dd x\dd yE\left[\left.  e^{-\check{J}_{0,1}^{s,\e,\lambda}} \delta_0(x-\omega_{s-\e^\gamma}) \right. \right. \notag\\
&\quad \hspace{5em}\left. \left. \phantom{e^{-\check{J}_{0,1}^{s,\e,\lambda}}} \times \left(\Phi_{0,s-\e^\gamma}+\Phi_{s+\e^\gamma,1}+\Phi_{0,s-\e^\gamma;s+\e^\gamma,1}\right) \right|\omega_{s+\e^\gamma}=y\right]\tau p_{2\e^\gamma}(x,y)\notag\\
&=\int_{T^3(\e^\alpha,\e^\beta)}\dd s\dd u\dd v \tau E\left[ e^{-\check{J}_{0,1}^{s,\e,\lambda}}
 \left(\Phi_{0,s-\e^\gamma}+\Phi_{s+\e^\gamma,1}+\Phi_{0,s-\e^\gamma;s+\e^\gamma,1}\right) \right] \notag\\
&=\int_{T^3(\e^\alpha,\e^\beta)}\dd s\dd u\dd v \tau E\left[e^{-\overline{J}_{0,1}^{\e,\lambda}}\Phi_{0,1}\right]+O\left(\e^{\gamma -\frac{3}{2}}\right)\max|f|\label{eq:l3-2}\\
&=-\frac{1}{2}\frac{\dd \kappa_2(\e)}{\dd \e}\rho(\e)+O\left(\e^{\gamma-\frac{3}{2}}\right)\max|f|,\label{eq:I3-3}
\end{align}
where we have used \eqref{eq:kappa2order2} in the last line. 
For the proof of \eqref{eq:l3-2}, we note that from \eqref{eq:JJhatdiff} we have
\begin{align*}
&E\left[ e^{-\check{J}_{0,1}^{s,\e,\lambda}}    \left(\Phi_{0,s-\e^\gamma}+\Phi_{s+\e^\gamma,1}+\Phi_{0,s-\e^\gamma;s+\e^\gamma,1}\right) \right] \\
&=E\left[ e^{-\overline{J}_{0,1}^{\e,\lambda}}    \left(\Phi_{0,s-\e^\gamma}+\Phi_{s+\e^\gamma,1}+\Phi_{0,s-\e^\gamma;s+\e^\gamma,1}\right) \right]\\
&+O\left(E\left[ e^{-\check{J}_{0,1}^{s,\e,\lambda}} \left|\overline{J}_{0,1}^{\e,\lambda}-\check{J}_{0,1}^{s,\e,\lambda}\right|  \left(\Phi_{0,s-\e^\gamma}+\Phi_{s+\e^\gamma,1}+\Phi_{0,s-\e^\gamma;s+\e^\gamma,1}\right) \right]\right)
\end{align*}
where we have used the fact that
$\overline{J}_{0,1}^{\e,\lambda}-\check{J}_{0,1}^{s,\e,\lambda}\leq 2\lambda \e^\gamma \kappa_1(\e)-\lambda^2\e^\gamma \kappa_2(\e)\leq 1$ for $\e>0$ small enough
and $\left|e^x-e^y\right|\leq e^{x+1}|x-y|$ for $x+1>y$. Then, it is easy to see that the last term is $O\left(\e^{\gamma-\frac{1}{2}}\right)\max|f|$ and the first term is $E\left[ e^{-\overline{J}_{0,1}^{\e,\lambda}}    \Phi \right]+O(\e^\gamma)\max|f|$.

Also, we show that
\begin{align}
&\int_{T^3(\e^\alpha,\e^\beta)}\dd s\dd u\dd v \iint \dd x\dd yE\left[\left.  e^{-\check{J}_{0,1}^{s,\e,\lambda}} \delta_0(x-\omega_{s-\e^\gamma}) \right. \right. \notag\\
&\quad \hspace{5em}\left. \left. \phantom{e^{-\check{J}_{0,1}^{s,\e,\lambda}}} \times \left(\Phi_{0,s-\e^\gamma}+\Phi_{s+\e^\gamma,1}+\Phi_{0,s-\e^\gamma;s+\e^\gamma,1}\right) \right|\omega_{s+\e^\gamma}=y\right] r_{u,s,v}(x,y)\notag\\
&=O\left(\e^{\beta-\gamma-1}\right)\max|f|.\label{eq:I3-4}
\end{align}
Indeed, it is easy to see that
\begin{align*}
\left|q_{u,s,v}(x,y)-\tau p_{2\e^\gamma}(x,y)\right|&\leq  C\tau \frac{s+\e-u-\sigma}{2\e^\gamma-\e-(s-u)+\sigma}p_{2\e^\gamma-\e-(s-u)+\sigma}(x,y)\\
&\quad +C\tau \frac{(s+\e-u-\sigma)|x-y|^2}{(2\e^\gamma-\e-(s-u)+\sigma)(2\e^\gamma)}p_{2\e^\gamma}(x,y) \\
&\leq C\e^{\beta-\gamma}\tau \left(1+\frac{|x-y|^2}{\e^\gamma}\right)p_{2\e^\gamma}(x,y)\\
&=:\widetilde{R}_{u,s,v}(x,y) 
\end{align*}
and $\check{J}_{0,1}^{s,\e,\lambda}\geq \overline{J}_{0,s-\e^\gamma}^{\e,\lambda}+\overline{J}_{s+\e^\gamma,1}^{\e,\lambda}$.
Therefore, in view of Lemma \ref{lem:finitetra} we have
\begin{align*}
&\left| \int_{T^3(\e^\alpha,\e^\beta)}\dd s\dd u\dd v \iint \dd x\dd yE\left[\left.  e^{-\check{J}_{0,1}^{s,\e,\lambda}} \delta_0(x-\omega_{s-\e^\gamma}) \right. \right. \right. \\
&\quad \hspace{4em} \left. \phantom{\int} \left. \left. \phantom{e^{-\check{J}_{0,1}^{s,\e,\lambda}}} \times \left(\Phi_{0,s-\e^\gamma}+\Phi_{s+\e^\gamma,1}+\Phi_{0,s-\e^\gamma;s+\e^\gamma,1}\right) \right|\omega_{s+\e^\gamma}=y\right] r_{u,s,v}(x,y) \right| \\
&\leq \max|f|\int_{T^3 (\e^\alpha,\e^\beta)} \dd s\dd u\dd v \\
&\quad \times \iint \dd x\dd y E\left[ \left.e^{-\overline{J}_{0,s-\e^\gamma}^{\e,\lambda} -\overline{J}_{s+\e^\gamma,1}^{\e,\lambda}} \delta_0(x-\omega_{s-\e^\gamma}) \right|\omega_{s+\e^\gamma}=y\right]\widetilde{R}_{u,s,v}(x,y)\\
&\leq C\max|f|\e^{\beta-\gamma}\int_{T^3 (\e^\alpha,\e^\beta)} \dd s\dd u\dd v  \tau  \\
&\leq O(\e^{\beta-\gamma-1})\max |f|.
\end{align*}
Thus, we obtain from \eqref{eq:I3-1}-\eqref{eq:I3-4} that 
\begin{align*}
\int_{0}^{1-\e}\dd s I_{\e,s}^{(2,1)}=-\frac{1}{2}\frac{\dd}{\dd \e}\kappa_2(\e)\rho(\e)+\max|f|O(\e^{\widetilde{\delta}-1})
\end{align*}
for some $\widetilde{\delta} \in (0,1)$.

Similarly, we can prove that
\begin{align*}
\int_{0}^{1-\e}\dd s I_{\e,s}^{(2,2)}=-\frac{1}{2}\frac{\dd}{\dd \e}\kappa_2(\e)\rho(\e)+\max|f|O(\e^{\widetilde{\delta}-1})
\end{align*}
for some $\widetilde{\delta} \in (0,1)$.

\section{Some estimates for the self-intersection local times with respect to $\nu_\lambda$}\label{sec:4}

In Section \ref{sec:2}, we actually proved that $\left\{\widehat{J}_{2^{-\ve_0n},2^{-n}}(u_1,u_2,h)\right\}_{n\geq 1}$ is a Cauchy sequence in $L^p(\nu_0)$ for any $p\geq 1$. The aim of this section is  to prove that this sequence is also  a Cauchy sequence in $L^{p}(\nu_\lambda)$ for $\lambda\in (0,\infty)$. Lemma \ref{lem:phiest} will  help us to achieve it. The main result in this section is the following theorem, which will play an important role in the proof that $\nu_\lambda$ is $h$-quasi-invariant for $h\in K_0$.

\begin{theorem}\label{thm:rho}
Let $\gamma\in \left(\frac{1}{3},\frac{1}{2}\right)$ be the constant in Lemma \ref{lem:Jcauchy}. Then,  the followings hold:

\begin{enumerate}[label=(\roman*)]
\item\label{item:4.1-1} For any given $u_1,u_2\in\R$, $\lambda>0$, $p\geq 1$ and $h\in K$ (with $K$ as in Section \ref{sec:Main}), there is  a random variable $\rho_{\lambda}(u_1,u_2,h)\in L^p(\nu_\lambda)$ such that \begin{align*}
\lim_{n\to\infty}E_{\nu_\lambda}\left[\left|\widehat{J}_{\e_n,a_n}(u_1,u_2,h)-\rho_{\lambda}(u_1,u_2,h)\right|^p\right]=0,
\end{align*}
where $\widehat{J}$ is defined in \eqref{eq:Jhatuuh}.

\item\label{item:4.1-2} For any given $u_1,u_2\in [-M,M]$ ($M>0$),  $\lambda>0$, $h\in K$, and  $p\geq 1$, there is a constant $C\in(0,\infty)$ such that \begin{align*}
E_{\nu_\lambda}\left[\left|\rho_{\lambda}(u_1,u_2,h)\right|^p\right] \leq C|u_1-u_2|^{\gamma p}. 
\end{align*}
\end{enumerate}
\end{theorem}

\begin{rem}\label{rem:thmrho}
When $u_1=u$ and $u_2=0$, we have that
\begin{align*}
\widetilde{J}_{\e_n,a_n}(u,h)\to \rho(u,0,h)=:\rho(u,h)\quad \mathrm{in ~}L^p(\nu_\lambda)
\end{align*}
as $n\to \infty$, where $\widetilde{J}$ is defined in \eqref{eq:Jtideuh}.
\end{rem}

\begin{proof}
For $u_1,u_2\in \R$ and $h\in K$, we set
\begin{align}
\Phi_n = \Phi_n(u_1,u_2,h) := \widehat{J}_{\e_n,a_n}(u_1,u_2,h)-\widehat{J}_{\e_{n-1},a_{n-1}}(u_1,u_2,h),\quad n\in \N.\label{eq:PhiDef}
\end{align}
It follows that
\begin{align*}
\Phi_n&=\int_{T_{\e_n}\cap \Delta_{0,1,0,1}}\dd \sigma\dd \tau\int e^{i\langle y,\omega_\sigma-\omega_\tau\rangle}\left(e^{i\langle y,u_1(h_\sigma-h_\tau)\rangle}-e^{i\langle y,u_2(h_\sigma-h_\tau)\rangle}\right)f_{a_n}(y)\dd y
\\
&\hspace{3em}- \int_{T_{\e_{n-1}}\cap \Delta_{0,1,0,1}}\dd \sigma\dd \tau\int e^{i\langle y,\omega_\sigma-\omega_\tau\rangle}\left(e^{i\langle y,u_1(h_\sigma-h_\tau)\rangle}-e^{i\langle y,u_2(h_\sigma-h_\tau)\rangle}\right)f_{a_{n-1}}(y)\dd y.
\end{align*}
Hence, $\Phi_n$ has the form $\Phi^1-\Phi^2$, where $\Phi^1$ and $\Phi^2$ are functions defined by \eqref{eq:PhiF} with replacement of $f$ by $\widetilde{f}_1$ and $\widetilde{f}_2$, respectively (see Remark \ref{rem:lemphiest}). Since $|e^{ix}-e^{iy}|\leq |x-y|$ for any $x,y\in \R$, there exists a constant $C>0$ which is independent of $\ve_0$ such that  $\max |\widetilde{f}_1|+\max|\widetilde{f}_2|\leq C|u_1-u_2|2^{2n}$. 
Therefore,  Lemma \ref{lem:phiest} and Remark \ref{rem:lemphiest} imply that
\begin{align*}
&\left|E\left[\exp\left(-\overline{J}_{0,1}^{\e_1,\lambda}\right)|\Phi_n|^p\right]-E\left[\exp\left(-\overline{J}_{0,1}^{\e_2,\lambda}\right)|\Phi_n|^p\right]\right|\leq C|u_1-u_2|^{p}2^{2pn}(\e_1^{\widetilde{\delta}}+\e_2^{\widetilde{\delta}}),\\
&\hspace{28em} \e_1,\e_2\in (0,1).
\end{align*}
Letting $\e_1\to 0$, in view of \eqref{eq:nulambda} we have that for $\e \in (0,1)$
\begin{align}
\left|N(\lambda)E_{\nu_\lambda}\left[|\Phi_n|^p\right]-E\left[\exp\left(-\overline{J}_{0,1}^{\e,\lambda}\right)|\Phi_n|^p\right]\right|\leq C|u_1-u_2|^p 2^{2pn} \e^{\widetilde{\delta}}. \label{eq:cauchylambda}
\end{align}
We choose
\begin{align}
\widetilde{\e}_n=2^{-3 p \widetilde{\delta}^{-1}n}.\label{eq:tildeen}
\end{align}
Then, it holds that
\begin{align}
\left|N(\lambda)E_{\nu_\lambda}\left[|\Phi_n|^p\right]-E\left[\exp\left(-\overline{J}_{0,1}^{\widetilde{\e}_n,\lambda}\right)|\Phi_n|^p\right]\right|\leq C|u_1-u_2|^p2^{-pn}.\label{eq:nulambdaJphi}
\end{align}
By Lemma \ref{lem:Jcauchy}, we know that there is a constant $\widetilde{\delta}' \in (0,1)$ such that
\begin{align*}
E\left[|\Phi_n|^{pq}\right]\leq C_p|u_1-u_2|^{pq\gamma}2^{-pq \widetilde{\delta}' n}
\end{align*}
for $q\geq 1$. 
From the definition of $\overline{J}_{u,v}^{\e,\lambda}$, we see that for $q\geq 1$ and $0\leq u< v\leq 1$,
\begin{align}
\exp\left(-\frac{q}{q-1}\overline{J}_{u,v}^{{\e},\lambda}\right)&=\exp\left(-\overline{J}_{u,v}^{{\e},\frac{q}{q-1}\lambda}+\left(\frac{q^2}{(q-1)^2}-\frac{q}{q-1}\right)\lambda^2(v-u)\kappa_2(\e)\right)\notag\\
&=\exp\left(-\overline{J}_{u,v}^{{\e},\frac{q}{q-1}\lambda}\right)\exp\left(\frac{q\lambda^2\kappa_2({\e})}{(q-1)^2}(v-u)\right) . \label{eq:eJmoment}
\end{align}

This together with \eqref{eq:partpest} implies that
\begin{align}
E\left[\exp\left(-\overline{J}_{0,1}^{\widetilde{\e}_n,\lambda}\right)|\Phi_n|^p\right]&\leq E\left[\exp\left(-\frac{q}{q-1}\overline{J}_{0,1}^{\widetilde{\e}_n,\lambda}\right)\right]^\frac{q-1}{q}E\left[|\Phi_n|^{pq}\right]^\frac{1}{q}\notag\\
&\leq C|u_1-u_2|^{p\gamma}2^{-p\widetilde{\delta}' n}\exp\left( \frac{\lambda^2\kappa_2(\widetilde{\e}_n)}{q-1}\right)E\left[\exp\left(-\overline{J}_{0,1}^{\widetilde{\e}_n,\frac{q\lambda}{q-1}}\right)\right] ^\frac{q-1}{q} \notag\\
&\leq C|u_1-u_2|^{p\gamma}2^{-p\widetilde{\delta}' n}\exp\left( \frac{\lambda^2\kappa_2(\widetilde{\e}_n)}{q-1}\right)\notag\\
&\leq C|u_1-u_2|^{p\gamma}2^{-\frac{p\widetilde{\delta}' n}{2}},\label{eq:jlambdaphi}
\end{align}
if $q>1$ is large  enough such that $\dis \frac{\lambda^2\kappa_2(\widetilde{\e}_n)}{q-1}< \frac{p\widetilde{\delta}' n}{2}\log 2$.
By \eqref{eq:nulambdaJphi} and \eqref{eq:jlambdaphi}, we then have that
\begin{align*}
N(\lambda)E_{\nu_\lambda}[|\Phi_n|^p]\leq C|u_1-u_2|^{p\gamma }\left(2^{-pn}+2^{-\frac{p\widetilde{\delta}' n}{2}}\right),
\end{align*}
which proves that $(\widehat{J}_{\e_n,a_n}(u_1,u_2,h))_{n\geq 1}$ is a Cauchy sequence in $L^p(\nu_\lambda)$. Therefore, there exists a $\rho_\lambda(u_1,u_2,h)$ such that
\begin{align*}
\lim_{n\to\infty}E_{\nu_{\lambda}}\left[\left|\widehat{J}_{\e_n,a_n}(u_1,u_2,h)-\rho_\lambda(u_1,u_2,h)\right|^p\right]=0
\end{align*}
and assertion \ref{item:4.1-1} is proved.

From \ref{item:4.1-1} and \eqref{eq:Nlam} we see that
\begin{align*}
&N(\lambda)E_{\nu_\lambda}\left[|\rho_\lambda(u_1,u_2,h)|^p\right] ^{1/p} \\
&\leq N(\lambda)E_{\nu_\lambda}\left[|\widetilde{J}_{\e_1,a_1}(u_1,h)-\widetilde{J}_{\e_1,a_1}(u_2,h)|^p\right] ^{1/p} +\sum_{n\geq 2}N(\lambda)E_{\nu_\lambda}\left[|\Phi_n|^p\right] ^{1/p} \\
&\leq C|u_1-u_2|^{\gamma}.
\end{align*}
Thus, we obtain \ref{item:4.1-2}.
\end{proof}

By a similar argument as in the proof of Theorem \ref{thm:rho}, we obtain the following corollary.

\begin{cor}\label{cor:jphilp}
Let $\lambda\geq 0$. Suppose  that a random variable $\Phi_n$ satisfies that for any $q>1$, $E[|\Phi_n|^q]\leq C2^{-bqn}$ for some $b>0$ and $C>0$. We set $\widetilde{\e}_n=2^{-cn}$ for some $c>0$.
If $q>1$ large enough so that $\frac{ c}{(q-1)(2\pi)^2}\lambda^2<\frac{b}{2}$, then
\begin{align*}
E\left[\exp\left(-\overline{J}_{0,1}^{\widetilde{\e}_n,\lambda}\right)\left|\Phi_n\right|\right] \leq C2^{-\widetilde{\delta}n}
\end{align*}
for some $C>0$ and $\widetilde{\delta}>0$.

\end{cor}

\section{Proof of Theorem \ref{thm:thm3}}\label{sec:5}

We provide the proof of Theorem \ref{thm:thm3}. As mentioned in Section \ref{sub:strategy}, we may expect that \begin{align*}
\frac{\dd \nu_\lambda \circ \tau_{uh}^{-1}}{\dd \nu_\lambda}=\exp\left(-\lambda (J^{0}_{0,1}(\omega-uh)-J^{0}_{0,1}(\omega))+u\int_0^1h'_s\dd \omega_s-\frac{u^2}{2}\int_0^1 (h_s')^2\dd s\right).
\end{align*} 
Moreover, the result in Section \ref{sec:4} allows us to expect $J^{0}_{0,1}(\omega-uh)-J^{0}_{0,1}(\omega)$ is given by $\rho_{\lambda}(-u,h)$. We verify that this expectations is true. In this section, we fix $\lambda\geq 0$ and we omit the subscript of $\rho_\lambda$.

Let \begin{align*}
\mathcal{J}_{n,\lambda}&:=\lambda J_{0,1}^{\e_n,a_n}-\lambda\kappa_1(\e_n)+\lambda^2\kappa_2(\e_n),\\
\mu_{n,\lambda}&:=E\left[\exp\left( -\mathcal{J}_{n,\lambda}' \right) \right]^{-1}\exp\left(-\mathcal{J}_{n,\lambda}'\right)\nu_0,
\end{align*}
where $\mathcal{J}_{n,\lambda}':=\mathcal{J}_{n,\lambda} 1\left\{\left|J_{0,1}^{\e_n,a_n}-J_{0,1}^{\e_n}\right|\leq 2^{-\widetilde{\delta}_1n}\right\}$ for $\widetilde{\delta}_1\in(0,\frac{1}{200})$ which is chosen later. 

In this section, $\ve_0\in (0,1)$ will be retaken and made smaller again and again, but only finitely many times. 
Let us first prove the following lemma.

\begin{lemma}\label{lem:munlambda}
For $\lambda\in [0,\infty)$ and for $\widetilde{\delta}_1\in \left(0,\frac{1}{200}\right)$ small enough,
\begin{align*}
\mu_{n,\lambda}\Rightarrow\nu_\lambda\quad \text{in the weak sense as $n\to\infty$}.
\end{align*}
\end{lemma}

\begin{proof}

First we note from Corollary \ref{cor:JaJ} that there exists $C>0$ such that
\begin{align}
E\left[\left(1_{\{\left|J_{0,1}^{\e_n,a_n}-J_{0,1}^{\e_n}\right|\geq 2^{-\widetilde{\delta}_1n}\}}\right)^q\right]\leq C\left(2^{\left(\widetilde{\delta}_1-\frac{1}{6}\right)qn}+2^{\left(\widetilde{\delta}_1-\frac{(1+20\gamma)}{24}+\frac{3\ve_0}{4}\right)qn}\right)\label{eq:Eventtail}
\end{align}
for $n\geq 1$ and $q\geq 1$

Let $\Xi$  be a nonnegative measurable function on the space $(X,\mathcal{B})$. Then, for any given $\lambda\in (0,\infty)$,
\begin{align}
&\left|E\left[\Xi\exp\left(-\mathcal{J}_{n,\lambda}'\right)\right]-E\left[\Xi\exp\left(-\overline{J}_{0,1}^{\e_n,\lambda}\right)\right]\right|\notag\\
&\leq 2\lambda 2^{-\widetilde{\delta}_1n}E\left[\Xi\exp\left(-\overline{J}_{0,1}^{\e_n,\lambda}\right)\right]+E\left[\Xi\left(1+\exp\left(-\overline{J}_{0,1}^{\e_n,\lambda}\right)\right):\left|J_{0,1}^{\e_n,a_n}-J_{0,1}^{\e_n}\right|\geq 2^{-\widetilde{\delta}_1n}\right]\label{eq:JJerr}
\end{align}
as long as $\lambda 2^{-\widetilde{\delta}_1n}\leq 1$ is sufficiently small.
Taking $\Xi=\Psi_{0,1}$ which is defined as in \eqref{eq:psidef}, we can see from Corollary \ref{cor:JaJ}, Lemma \ref{lem:finitetra},  Corollary \ref{cor:jphilp}, \eqref{eq:Eventtail}, and \eqref{eq:JJerr} that 
\begin{align}
&\left|E\left[\Psi_{0,1}\exp\left(-\mathcal{J}_{n,\lambda}'\right)\right]-E\left[\Psi_{0,1}\exp\left(-\overline{J}_{0,1}^{\e_n,\lambda}\right)\right]\right|
\leq C2^{-\widetilde{\delta}_2n}\label{eq:JJest} 
\end{align}
for some $\widetilde{\delta}_2\in (0,1)$ if  $\widetilde{\delta}_1$ is sufficiently small.
By definition of $\nu_\lambda$, we have \begin{align*}
\lim_{n \to0 }E\left[\exp\left(-\overline{J}_{0,1}^{\e_n,\lambda}\right)\right]^{-1}E\left[\Psi_{0,1}\exp\left(-\overline{J}_{0,1}^{\e_n,\lambda}\right)\right]=E_{\nu_\lambda}\left[\Psi_{0,1}\right].
\end{align*}
Thus, it follows from \eqref{eq:JJest} that 
\begin{align*}
\lim_{n \to0 }E\left[\exp\left(-\mathcal{J}_{n,\lambda}'\right)\right]^{-1}E\left[\Psi_{0,1}\exp\left(-\mathcal{J}_{n,\lambda}'\right)\right]=E_{\nu_\lambda}\left[\Psi_{0,1}\right].
\end{align*}
\end{proof}

Let 
\begin{align}
A_n:=\left\{\omega:\left|J_{0,1}^{\e_n,a_n}(\omega)-J_{0,1}^{\e_n}(\omega)\right|<2^{-\widetilde{\delta}_1n}\right\}\label{eq:eventAndfn}
\end{align}
and let $\widetilde{\delta}_1\in (0,\frac{1}{200})$ be the small constant  which is given in Lemma \ref{lem:munlambda} and will be retaken and made smaller in the proof of Lemma \ref{lem:RNder} below.
Set
\begin{align}
D_{n,\lambda}(\omega):=
\exp\left(-\mathcal{J}'_{n,\lambda}(\omega-uh)+\mathcal{J}'_{n,\lambda}(\omega)\right)\exp\left(u\int_0^1h_t'\dd \omega_t-\frac{1}{2}u^2\int_0^1 (h_t')^2\dd t\right)\label{eq:Dnlambda}
\end{align}
with $\int_0^1h_t'\dd \omega_t$ understood as an It\^o stochastic integral. It is clear that $D_{n,\lambda}=\frac{\dd (\mu_{n,\lambda}\circ \tau_{uh}^{-1})}{\dd \mu_{n,\lambda}}$. Since we do not know at present whether $(\omega_t)_{t\in [0,1]}$, as coordinate process under $\nu_\lambda$, is a semimartingale or not, the quantity $\int_0^1 h'_t\dd \omega_t$ is not guaranteed to be well-defined under the polymer measure $\nu_\lambda$.
Moreover, it is known that $\nu_\lambda$ is singular with respect to $\nu_0$ (see \cite{Wes82}).
However, if $h\in K_0$, by It\^o's formula we have
\begin{align}
\int_0^1 h'_t\dd \omega_t = h'_1\omega_1-\int_0^1 \omega_t\dd h'_t.\label{eq:stoint}
\end{align} 
The right-hand side of the above equality is well-defined under the measure $\nu_\lambda$.
From now on we assume that $h\in K_0$, and we always regard the left-hand side of \eqref{eq:stoint} under the measure $\nu_\lambda$ as the term defined by the right-hand side of \eqref{eq:stoint}.
Let 
\begin{align*}
a_{uh}:=e^{-\lambda \rho_\lambda(-u,h)+V(u,h)},
\end{align*}
where 
\begin{align*}
V(u,h) = V(\omega,u,h) :=uh'_1\omega_1-u\int_0^1 \omega_t h''_t\dd t-\frac{1}{2}u^2 \int_0^1 |h_t'|^2\dd t,
\end{align*}
and $\rho_\lambda(u,h)$ is as defined in Theorem \ref{thm:rho}.

\begin{lemma}\label{lem:RNder}
For any given $u\in\R$, $h\in K_0$ and for any $\lambda\geq 0$, it  holds that 
\begin{align*}
\frac{\dd{(\nu_\lambda\circ \tau_{{uh}}^{-1})}}{\dd \nu_\lambda}=a_{uh},\quad \nu_\lambda\text{-a.e.}
\end{align*}
\end{lemma}

This lemma follows from the following lemma.

\begin{lemma}\label{lem:step1}
Let $\mathcal{F}_0$ be the set of bounded continuous functions $f:\mathbb{R}^{3p}\to \R$ for some $p\geq 1$. Then, for any given $u\in\R$, $h\in K_0$ and for any $\lambda\in [0,\infty)$, the following holds
\begin{align}
\lim_{n\to\infty} \int f(\omega_{t_1},\dots,\omega_{t_p})\frac{\dd (\mu_{n,\lambda}\circ \tau_{uh}^{-1})}{\dd \mu_{n,\lambda}}\dd \mu_{n,\lambda}=\int f(\omega_{t_1},\dots,\omega_{t_p})a_{uh}\dd \nu_\lambda\label{eq:RNmuconvauk}
\end{align}
for any $f\in \mathcal{F}_0 $ and any $0\leq t_1< \dots<t_p\leq 1$ with $p\geq 1$.
\end{lemma}

\begin{proof}[Proof of Lemma \ref{lem:RNder}]
By lemma \ref{lem:munlambda}, we also know that for any given $u\in \R$ and $h\in K_0$ \begin{align*}
\mu_{n,\lambda}\circ \tau_{uh}^{-1}\Rightarrow \nu_\lambda\circ \tau_{uh}^{-1} \,\,\textrm{ in the weak sense},
\end{align*}
as $n\to\infty$. 
Hence, we obtain that for any given $f\in\mathcal{F}_0$ and $0\leq t_1<\dots<t_p\leq 1$ ($p\geq 1$)
\begin{align}
\int f(\omega_{t_1},\dots,\omega_{t_p})\frac{\dd (\mu_{n,\lambda}\circ \tau_{uh}^{-1})}{\dd \mu_{n,\lambda}}\dd \mu_{n,\lambda} \to  \int f(\omega_{t_1},\dots,\omega_{t_p})\dd (\nu_\lambda\circ \tau_{uh}^{-1}),\label{eq:weakconv}
\end{align} 
as $n\to\infty$.
Combining \eqref{eq:RNmuconvauk} and \eqref{eq:weakconv}, we have that
\begin{align*}
\int f(\omega_{t_1},\dots,\omega_{t_p})\dd (\nu_\lambda\circ \tau_{uh}^{-1}) &=\int f(\omega_{t_1},\dots,\omega_{t_p})a_{uh}\dd \nu_\lambda, \\
& \textrm{for all\, }f\in \mathcal{F}_0,  0\leq t_1<\dots<t_p\leq 1
\end{align*}
This implies that both $\dd \left(\nu_\lambda\circ \tau_{uh}^{-1}\right)$ and $a_{uh}\dd \nu_\lambda$ are probability measures on  the Wiener space and that their finite-dimensional distributions coincide.
Therefore, we obtain the desired result.
\end{proof}

We postpone the proof of Lemma \ref{lem:step1} to Section \ref{subsec:prooflemmas}, and now prove Theorem \ref{thm:thm3}.

\begin{proof}[Proof of Theorem \ref{thm:thm3}] We have to show that for each $h\in K_0$, there is $\{a_{uh}',u\in \R\}$ such that \begin{align}
&\nu_\lambda(a_{uh}'=a_{uh})=1,\quad u\in \R,\label{eq:contver}\\
&\nu_\lambda(a_{uh}'\text{ is continuous with respect to $u\in \R$})=1.\label{eq:contver2}
\end{align}
We fix $\lambda\geq 0$.
By Lemma \ref{lem:RNder}, we know that $a_{uh}=\frac{\dd \tau_{uh}(\nu_\lambda)}{\dd \nu_{\lambda}}$, $\nu_\lambda$-almost everywhere.
By Theorem \ref{thm:rho} and the Kolmogorov continuity theorem, we know that there is $\rho'_\lambda(u,h)$ such that
\begin{align*}
&\nu_\lambda(\rho'_\lambda(u,h)=\rho_\lambda(u,h))=1,\quad u\in \R\\
&\nu_\lambda(\rho'_\lambda(u,h) \text{ is continuous with respect to $u\in \R$})=1.
\end{align*}
For $h\in K_0$, we set \begin{align*}
a_{uh}'=e^{-\lambda\rho'_\lambda(-u,h)+V(u,h)}.
\end{align*}
Then, $\{a_{uh}';u\in \R\}$ satisfies \eqref{eq:contver} and \eqref{eq:contver2}.

\end{proof}

\subsection{Proof of Lemma \ref{lem:step1}}\label{subsec:prooflemmas}

We prove Lemma \ref{lem:step1}. Hereafter, we fix $f\in \mathcal{F}_0$ and $0\leq t_1<\dots<t_p\leq 1$.
Let
\begin{align*}
m(n):=\frac{\widetilde{\delta}\ve_0}{2}n,
\end{align*}
where $\widetilde{\delta}>0$ is a constant  given in Lemma \ref{lem:phiest} and $\ve_0$ is given in Lemma \ref{lem:Jcauchy}. By Lemma \ref{lem:munlambda}, we expect that  $\dis \left(\widetilde{J}_{\e_n,a_n}(u,h),V(u,h)\right)$ under $\mu_{n,\lambda}$ converges to $(\rho_\lambda(u,h),V(u,h))$ under $\nu_\lambda$. In Lemma \ref{lem:step2}, we prove the tightness of them, which implies the convergence of $V(u,h)$. After that, we  prove the convergence of $\dis \widetilde{J}_{\e_n,a_n}(u,h)$, and then we complete the proof of Lemma \ref{lem:step1}.

For the proof of Lemma \ref{lem:step1}, we prepare following lemmas.

\begin{lemma}\label{lem:step2-1}
For any given $u\in\R$, $h\in K_0$ and for any $\lambda\geq 0$
 \begin{align}
\sup_{n\geq 1} E_{\mu_{n,\lambda}}\left[\left|\widetilde{J}_{\e_n,a_n}(u,h)\right|\right]<\infty.\label{eq:munlambdaetabdd}
\end{align}
\end{lemma}

\begin{rem}
In \cite{ARZ96}, the statement of Lemma \ref{lem:step2-1} is mentioned but no proof is given. 
They referred to the equation number (6.2), which did not exist in \cite{ARZ96}.
\end{rem}

\begin{proof}
We find from Lemma \ref{lem:Jcauchy} and   Corollary \ref{cor:jphilp} that 
\begin{align}
\sup_{n\geq 1}E\left[\exp\left(-\overline{J}_{0,1}^{\e_n,\lambda}\right)\left|\widetilde{J}_{\e_n,a_n}(u,h)-\widetilde{J}_{\e_{m(n)},a_{m(n)}}(u,h)\right|\right]\leq C.\label{eq:supejnj}
\end{align}
Moreover, we will see later that
\begin{align}
\left\{E\left[\exp\left(-\overline{J}_{0,1}^{\e_n,\lambda}\right)\left|\widetilde{J}_{\e_{m(n)},a_{m(n)}}(u,h)\right|\right]\right\}\text{ is a Cauchy sequence.}\label{eq:JJecauchy}
\end{align}
Combining \eqref{eq:supejnj} and \eqref{eq:JJecauchy} yields
\begin{align}
E\left[\exp\left(-\overline{J}_{0,1}^{\e_n,\lambda}\right)\left|\widetilde{J}_{\e_n,a_n}(u,h)\right|\right]\leq C \quad \label{eq:JJenbdd}.
\end{align}
Taking  $\Xi= |\widetilde{J}_{\e_n,a_n}(u,h)|$ in \eqref{eq:JJerr} and recalling that $\widetilde{\delta}_1 \in (0,\frac{1}{200})$, we see  that 
\begin{align*}
&\left|E\left[\exp\left(-\mathcal{J}_{n,\lambda}'\right)\left|\widetilde{J}_{\e_n,a_n}(u,h)\right|\right]-E\left[\exp\left(-\overline{J}_{0,1}^{\e_n,\lambda}\right)\left|\widetilde{J}_{\e_n,a_n}(u,h)\right|\right]\right|\\
&\leq \lambda 2^{-\widetilde{\delta}_1n}e^{\lambda 2^{-\widetilde{\delta}_1n}}E\left[\left|\widetilde{J}_{\e_n,a_n}(u,h)\right|\exp\left(-\overline{J}_{0,1}^{\e_n,\lambda}\right)\right]\\
&\quad +E\left[\left|\widetilde{J}_{\e_n,a_n}(u,h)\right|\left(1+\exp\left(-\overline{J}_{0,1}^{\e_n,\lambda}\right)\right):\left|J_{0,1}^{\e_n,a_n}-J_{0,1}^{\e_n}\right|\geq 2^{-\widetilde{\delta}_1n}\right]\\
&\leq C2^{-\widetilde{\delta}_3n}
\end{align*}
for some $\widetilde{\delta}_3>0$ by Corollaries \ref{cor:JaJ}, \ref{cor:lemJcauchy}, and \ref{cor:jphilp},  \eqref{eq:eJmoment}, and \eqref{eq:Eventtail}. 
Hence, we obtain  \eqref{eq:munlambdaetabdd} for any  $\lambda>0$ by \eqref{eq:JJenbdd}.

We complete the proof by showing that  \eqref{eq:JJecauchy} holds. We have
 \begin{align*}
&\left|E\left[\exp\left(-\overline{J}_{0,1}^{\e_n,\lambda}\right)\left|\widetilde{J}_{\e_{m(n)},a_{m(n)}}(u,h)\right|\right]-E\left[\exp\left(-\overline{J}_{0,1}^{\e_{n-1},\lambda}\right)\left|\widetilde{J}_{\e_{m(n-1)},a_{m(n-1)}}(u,h)\right|\right]\right|\\
&\leq \left|E\left[\exp\left(-\overline{J}_{0,1}^{\e_n,\lambda}\right)\left|\widetilde{J}_{\e_{m(n)},a_{m(n)}}(u,h)\right|\right]-E\left[\exp\left(-\overline{J}_{0,1}^{\e_{n-1},\lambda}\right)\left|\widetilde{J}_{\e_{m(n)},a_{m(n)}}(u,h)\right|\right]\right|\\
&\quad +E\left[\exp\left(-\overline{J}_{0,1}^{\e_{n-1},\lambda}\right)\left|\widetilde{J}_{\e_{m(n)},a_{m(n)}}(u,h)-\widetilde{J}_{\e_{m(n-1)},a_{m(n-1)}}(u,h)\right|\right]\\
&=:P_n^1+P_n^2.
\end{align*}
Then, we see from Corollaries \ref{cor:JaJ}, \ref{cor:lemJcauchy} and \ref{cor:jphilp} that there exist $C>0$ and $\widetilde{\delta}_{4}\in (0,1)$ such that
\begin{align*}
P_n^2\leq 2^{-\widetilde{\delta}_{4}n} 
\end{align*}
for $n\geq 1$ for any $\lambda>0$.
Noting the above definition of $m(n)$ and applying Lemma \ref{lem:phiest} (see Remark \ref{rem:lemphiest}) to $P_n^1$, we see that there exist $C>0$ and $\widetilde{\delta}_{5}\in (0,1)$ such that \begin{align*}
P_n^1\leq C2^{-\widetilde{\delta}_{5}n}
\end{align*}
for $n\geq 1$. 
Thus, we obtain \eqref{eq:JJecauchy}.
\end{proof}

For $M>1$ large enough, we set
\begin{align}
B_n(M):=&\left\{\left|\lambda \widetilde{J}_{\e_n,a_n}(-u,h)\right|\geq \log M\right\}
\cup \left\{|V(u,h)|\geq \log M\right\}.\label{eq:Bndfn}
\end{align}

\begin{lemma}\label{lem:step2}
For any given $u\in\R$, $h\in K_0$ and for any $\lambda \in (0,\infty)$, there exists a constants $C>0$ such that
\begin{align*}
\mu_{n,\lambda}\circ \tau_{uh}^{-1}(B_n(M))=\mu_{n,\lambda}\left(\{\omega+uh\in B_n(M)\}\right) \leq \frac{C}{\log M}
\end{align*}
for any $n\geq 1$.
\end{lemma}

\begin{proof}

From Lemma \ref{lem:finitetra} we easily obtain that
\begin{align}
&\sup_{n\geq 1} E_{\mu_{n,\lambda}}\left[|\omega_1|+\left(\int_0^1 |\omega_t|^2\dd t\right)^\frac{1}{2}\right]\leq C<\infty.\label{eq:omega1expbdd}
\end{align}
Since it holds that
\begin{align*}
V(\omega+uh,u,h)=uh_1'\omega_1-u\int_0^1\omega_t h_t''\dd t + \frac{u^2}{2} \int_0^1 |h_t'|^2\dd t+ u^2 h_0 h'_0,
\end{align*}
from \eqref{eq:omega1expbdd} we have
\begin{align}
&\sup_{n\geq 1}\mu_{n,\lambda}\left( \left| V(\cdot +uh, u,h)\right|\geq \log M\right) \notag\\
&\leq \sup_{n\geq 1}\mu_{n,\lambda}\left(\left|uh'_1\omega_1-u\int_0^1 \omega_th''_t\dd t\right|+\frac{u^2}{2}\left| \int_0^1 |h_t'|^2\dd t + h_0 h'_0\right| \geq \log M\right) \notag\\
&\leq \frac{C}{\log M}.\label{eq:munk}
\end{align}
From Lemma \ref{lem:step2-1} we also have 
\begin{align}
\mu_{n,\lambda}\circ \tau_{uh}^{-1}\left(\left|\widetilde{J}_{\e_n,a_n}(-u,h)\right|\geq \log M\right) =\mu_{n,\lambda}\left(\left|\widetilde{J}_{\e_n,a_n}(u,h)\right|\geq \log M\right)\leq \frac{C}{\log M}\label{eq:tildeJbdd}
\end{align}
for $n\geq 1$.
The assertion follows from \eqref{eq:munk} and \eqref{eq:tildeJbdd}.
\end{proof}

\begin{lemma}\label{lem:step3}
For any $R>0$, any given $u\in[-R,R]$, $h\in K_0$ and for any $\lambda >0$, there exist constants $C>0 $ and $\widetilde{\delta}_{6}>0$ such that\begin{align}
&\mu_{n,\lambda}\circ \tau_{uh}^{-1}(A_n^c)=\mu_{n,\lambda}(\{\omega+{u}h\in A_n^c\})\leq {C2^{-{\widetilde{\delta}_{6}n}}}\label{eq:Jac}
\end{align}
for $n\geq 1$, where $A_n$ is define by \eqref{eq:eventAndfn}.
\end{lemma}

\begin{proof}

Taking $\Xi=\left|J_{0,1}^{\e_n,a_n}(\omega+uh)-J_{0,1}^{\e_n}(\omega+uh)\right|$ in \eqref{eq:JJerr},  we have
\begin{align}
&E\left[\left|J_{0,1}^{\e_n,a_n}(\cdot+uh)-J_{0,1}^{\e_n}(\cdot+uh)\right|\exp\left(-\mathcal{J}_{n,\lambda}'\right)\right] \notag\\
&\leq E\left[\left|J_{0,1}^{\e_n,a_n}(\cdot+uh)-J_{0,1}^{\e_n}(\cdot+uh)\right|\exp\left (-\overline{J}_{0,1}^{\e_n,\lambda}\right)\right]\notag\\
&\quad +\left|E\left[\left|J_{0,1}^{\e_n,a_n}(\cdot+uh)-J_{0,1}^{\e_n}(\cdot+uh)\right|\exp\left(-\mathcal{J}_{n,\lambda}'\right)\right] \right.\notag \\
&\quad \hspace{3cm}\left .-E\left[\left|J_{0,1}^{\e_n,a_n}(\cdot+uh)-J_{0,1}^{\e_n}(\cdot+uh)\right|\exp\left (-\overline{J}_{0,1}^{\e_n,\lambda}\right)\right]\right|	\notag\\
&\leq E\left[\left|J_{0,1}^{\e_n,a_n}(\cdot+uh)-J_{0,1}^{\e_n}(\cdot+uh)\right|\exp\left (-\overline{J}_{0,1}^{\e_n,\lambda}\right)\right]\notag\\
&\quad+2\lambda 2^{-\widetilde{\delta}_1n}E\left[\left|J_{0,1}^{\e_n,a_n}(\cdot+uh)-J_{0,1}^{\e_n}(\cdot+uh)\right|\exp\left (-2\overline{J}_{0,1}^{\e_n,\lambda}\right)\right]\notag\\
&\quad +E\left[\left|J_{0,1}^{\e_n,a_n}(\cdot+uh)-J_{0,1}^{\e_n}(\cdot+uh)\right|\left(1+\exp\left (-\overline{J}_{0,1}^{\e_n,\lambda}\right)\right)\right].\label{eq:PsiJJ}
\end{align}
Then, 
we see  from Corollaries \ref{cor:JaJ} and \ref{cor:jphilp} that there exists $\widetilde{\delta}_{7}>0$ such that
\begin{align*}
&\textrm{the right-hand side of }\eqref{eq:PsiJJ}\leq C2^{-\widetilde{\delta}_{7}n}.
\end{align*}
Hence, there exist $C>0 $ and $\widetilde{\delta}_{7}>0$ such that
\begin{align*}
&\mu_{n,\lambda}(A_n^c(\cdot+uh))\leq {C2^{-{\widetilde{\delta}_{7}n}}},
\end{align*}
for $n\geq 1$.
\end{proof}

For $A\in \mathcal{B}$, $h\in K_0$, we set $\{\omega+h\in A\}$ by $A(\cdot-h)$.

\begin{lemma}\label{lem:step4}
For $M>1$  large enough, any given $u\in\R$, $h\in K_0$, and any $\lambda\in (0,\infty)$, there exist $C>0$ and $\widetilde{\delta}_{8}>0$ such that
 \begin{align}
\left|\int_{B_n^c(M)\cap  \left(A_n(\cdot-uh)\cap A_n(\cdot)\right)^c}f(\omega)e^{-\lambda\widetilde{J}_{\e_n,a_n}(-u,h)+V(u,h)}\dd \mu_{n,\lambda}\right|\leq C2^{-\widetilde{\delta}_{8} n}M^2\max|f|.\label{eq:munlambdaanan}
\end{align}
\end{lemma}

\begin{proof}
We have
\begin{align*}
&\left|\int_{B_n^c(M)\cap  \left(A_n(\cdot-uh)\cap A_n(\cdot)\right)^c}f(\omega)e^{-\lambda\widetilde{J}_{\e_n,a_n}(-u,h)+V(u,h)}\dd \mu_{n,\lambda}\right|\\
&\leq \max |f|M^2 \left|\int_{  \left(A_n(\cdot-uh)\cap A_n(\cdot)\right)^c}\dd \mu_{n,\lambda}\right| ,
\end{align*}
and hence, the statement follows from Lemma \ref{lem:step3}.
\end{proof}

Now we prove Lemma \ref{lem:step1} by using Lemmas \ref{lem:step2}, \ref{lem:step3} and \ref{lem:step4}.

\begin{proof}[Proof of Lemma \ref{lem:step1}]

We note that $B_n(M)^c$, the complement of the event $B_n(M)$ defined in \eqref{eq:Bndfn}, can be written as
\begin{align*}
&\left\{|V(u,h)|\leq  {\log M}\right\}\cap \left\{\left|\lambda\widetilde{J}_{\e_n,a_n}(-u,h)\right| \leq \log M\right\} .
\end{align*}
Let $\varphi_M\in C_b^2(\R^2)$ be a continuous function which satisfies
\begin{alignat*}{2}
&0\leq \varphi_M(x_1,x_2)\leq 1,\quad &&\textrm{if }(x_1,x_2)\in \R^2,\\
&\varphi_M(x_1,x_2)=1,\quad &&\textrm{for} \ |x_1|\leq \log M, \ |x_2|\leq \log M,\\
&\varphi_M(x_1,x_2)=0,\quad &&\textrm{for}\ |x_1|\geq 2\log M, \text{ or }\ |x_2|\geq 2\log M,
\end{alignat*}
and we set $\widetilde{\varphi}_M(\omega) := \varphi_M(\lambda \widetilde{J}_{n}(-u,h),V(u,h))$, where we recall the definition of  $D_{n,\lambda}$ \eqref{eq:Dnlambda}. 
Then, we have
\begin{align*}
\int f(\omega)D_{n,\lambda}\dd \mu_{n,\lambda}=\int \widetilde{\varphi}_M(\omega)f(\omega)D_{n,\lambda}\dd \mu_{n,\lambda} +\int \left(1-\widetilde{\varphi}_M(\omega)\right)f(\omega)D_{n,\lambda}\dd \mu_{n,\lambda}. 
\end{align*}
for $f(\omega)=f(\omega_{t_1},\dots,\omega_{t_p})$.

It is easy to see from  Lemma \ref{lem:step2} that for some constant $C>0$
\begin{align}
\left|\int \left(1-\widetilde{\varphi}_M(\omega)\right)f(\omega)D_{n,\lambda}\dd \mu_{n,\lambda}\right|\leq \sup|f| \mu_{n,\lambda}\circ \tau_{uh}^{-1}\left(B_n(M)\right)\leq \frac{C}{\log M}\max |f|. \label{eq:pfLem5.3-01}
\end{align}
We remark that
\begin{align*}
D_{n,\lambda}&=\exp\left(-\lambda\widetilde{J}_{n}(-u,h)1_{A_{n}(\omega-uh)\cap A_{n}(\omega)}+V(u,h)\right)\\
&\times \exp\left(-\left(\lambda{J}_{0,1}^{\e_n,a_n}(\omega-uh)-\lambda \kappa_1(\e_n)+\lambda^2 \kappa_2(\e_n)\right)1_{A_{n}(\omega-uh)\backslash A_{n}(\omega)}\right)\\
&\times \exp\left(\left(\lambda{J}_{0,1}^{\e_n,a_n}(\omega)-\lambda \kappa_1(\e_n)+\lambda^2 \kappa_2(\e_n)\right)1_{A_{n}(\omega)\backslash A_{n}(\omega-uh)}\right) ,
\end{align*}
and hence on $A_n(\cdot-uh)\cap A_n(\cdot)$,
\begin{align*}
D_{n,\lambda}&=\exp\left(-\lambda\widetilde{J}_{\e_n,a_n}(-u,h)+V(u,h)\right).
\end{align*}
Applying Lemma \ref{lem:step3} for $v=u,0$ and Lemma \ref{lem:step4}, we obtain that \begin{align}
\left|\int \widetilde{\varphi}_{M}(\omega)f(\omega){D_{n,\lambda}}\dd \mu_{n,\lambda}-\int \widetilde{\varphi}_{M}(\omega)f(\omega)e^{-\lambda\widetilde{J}_{n}(-u,h)+V(u,h)}\dd \mu_{n,\lambda}\right|\to 0\label{eq:ezetagap}
\end{align}
as $n\to\infty$.

Let $F_M:=\varphi_{M}(x_1,x_2)e^{-x_1+x_2}f(x_3,\cdots,x_{p+2})\in C_b^1(\R^{p+2})$.
Note that $F_M$ is Lipschitz continuous for each $M$.
Then, we can see  from the same argument as in the proof of \eqref{eq:jlambdaphi} and Corollary \ref{cor:lemJcauchy} that
\begin{align*}
&\left| \int \left( F_M\left( \lambda \widetilde{J}_{\e_n,a_n}(-u,h),V(u,h),\omega_{t_1},\cdots,\omega_{t_p})\right) \right. \right. \\
&\quad \left. \left. \hspace{4cm} - F_M\left( \lambda \widetilde{J}_{\e_{m(n)},a_{m(n)}} (-u,h),V(u,h),\omega_{t_1},\cdots,\omega_{t_p})\right)\right)\dd \mu_{n,\lambda} \right| \\
&\leq C_{M,f} \int \left| \widetilde{J}_{\e_n,a_n}(-u,h) - \widetilde{J}_{\e_{m(n)},a_{m(n)}} (-u,h) \right| \dd \mu_{n,\lambda} \to 0.
\end{align*}
Moreover, similar arguments  to \eqref{eq:cauchylambda} and \eqref{eq:nulambdaJphi} yields that \begin{align*}
&\left|\int F_M\left( \lambda \widetilde{J}_{\e_{m(n)},a_{m(n)}}(-u,h),V(u,h),\omega_{t_1},\cdots,\omega_{t_p})\right)\dd \mu_{n,\lambda} \right. \\
&\quad \left. \hspace{3cm} -\int F_M\left( \lambda \widetilde{J}_{\e_{m(n)},a_{m(n)}}(-u,h),V(u,h),\omega_{t_1},\cdots,\omega_{t_p})\right)\dd \nu_{\lambda}\right| \to0 
\end{align*}
as  $n\to\infty$.
Theorem \ref{thm:rho} (see Remark \ref{rem:thmrho}) implies that
\begin{align*}
& \int F_M\left( \lambda \widetilde{J}_{\e_{m(n)},a_{m(n)}}(-u,h),V(u,h),\omega_{t_1},\cdots,\omega_{t_p})\right)\dd \nu_{\lambda} \\
&\to \int F_M( \lambda \rho_\lambda(-u,h),V(u,h),\omega_{t_1},\cdots,\omega_{t_p})\dd \nu_{\lambda},
\end{align*}
as $N\to\infty$.
Putting these together, we obtain
\begin{align*}
&\varlimsup_{n\to \infty}\left|\int F_M\left( \lambda \widetilde{J}_{\e_n,a_n}(-u,h),V(u,h),\omega_{t_1},\cdots,\omega_{t_p})\right))\dd \mu_{n,\lambda} \right. \\
&\quad \hspace{3cm} \left. -\int F_M\left( \lambda \rho_\lambda(-u,h),V(u,h),\omega_{t_1},\cdots,\omega_{t_p}\right)\dd \nu_{\lambda}\right|=0.
\end{align*}
This inequality, \eqref{eq:pfLem5.3-01} and \eqref{eq:ezetagap} yield
\begin{align*}
&\varlimsup_{n\to \infty}\left|\int
f(\omega_{t_1},\dots,\omega_{t_p}) D_{n,\lambda} \dd \mu_{n,\lambda} \right. \\
&\quad \hspace{3cm} \left. - \int \widetilde{\varphi}_M(\omega)e^{-\lambda \rho_\lambda(-u,h)+V(u,h)}f(\omega_{t_1},\dots,\omega_{t_p})\dd \nu_{\lambda} \right|\leq \frac{C}{\log M}  \sup |f| .
\end{align*}
Since \begin{align*}
\lim_{M\to\infty}\int \widetilde{\varphi}_M(\omega)e^{-\lambda \rho_\lambda(-u,h)+V(u,h)}f(\omega_{t_1},\dots,\omega_{t_p})\dd \nu_{\lambda}=\int e^{-\lambda \rho(-u,h)+V(u,h)}f(\omega_{t_1},\dots,\omega_{t_p})\dd \nu_{\lambda},
\end{align*} we then obtain Lemma \ref{lem:step1}.

\end{proof}

This implies that also Theorem \ref{thm:thm3} has a complete proof.

\appendix

\section{Some estimates of integrals }\label{app}

In this appendix, we give some calculations of integrals, that have been used in the course of various proofs given above.

\subsection{Proof of \eqref{eq:I3estimate}}\label{app:I3esti}
Observe that 
\begin{align*}
& \iint _{0<s_2<s_1 <\infty, 0<t_1 <t_2 < 1}\dd s_1 \dd s_2 \dd t_1 \dd t_2 \notag\\
&\quad \hspace{4cm} [ (s_1-s_2)(t_2-t_1)(1-t_2+s_2) +(s_1-s_2+t_2-t_1)(1-t_2)s_2 \notag\\
&\quad \hspace{8cm}+(s_1-s_2+t_2-t_1)(1-t_2+s_2)t_1]^{-\frac{3}{2}}\\
&=  2\iint _{0<s_2<\infty, 0<t_1 <t_2 < 1} \dd s_2 \dd t_1 \dd t_2 \frac{1}{s_2+t_2(1-t_2)}\frac{1}{\sqrt{(t_2-t_1)((1-t_2+s_2)t_1+(1-t_2)s_2)}}\\
&\leq C\iint _{0<s_2<\infty, 0<t_2 < 1} \dd s_2  \dd t_2 \frac{1}{s_2+t_2(1-t_2)}\frac{1}{\sqrt{1-t_2+s_2}}\\
&\leq C\int _{0<t_2 < 1} \dd t_2 \frac{1}{\sqrt{t_2(1-t_2)}}<\infty,
\end{align*}
where in the first inequality we have used 
\begin{align*}
\int_a^b \frac{1}{\sqrt{(x-s)(t-x)+u}}\dd x=\arcsin \frac{b-\frac{(s+t)}{2}}{\sqrt{u+\frac{(t-s)^2}{4}}}-\arcsin \frac{a-\frac{(s+t)}{2}}{\sqrt{u+\frac{(t-s)^2}{4}}}
\end{align*}
when $s<t$, $u>0$ and $a,b\in [s,t]$. 

Similarly, 
\begin{align*}
& \iint _{0<s_2<s_1<\infty, 0<t_2<t_1<1} \dd s_1 \dd s_2 \dd t_1 \dd t_2 \notag\\
& \hspace{4cm} [ (t_2+s_2)(s_1-s_2)(t_1-t_2)+(s_1-s_2+t_1-t_2) s_2 t_2 \notag\\
& \hspace{7cm} \phantom{\int}+(t_2+s_2)(s_1-s_2+t_1-t_2)(1-t_1)]^{-\frac{3}{2}}\\
& =2\iint _{0<s_2<\infty, 0<t_2<t_1<1}  \dd s_2 \dd t_1 \dd t_2 \frac{1}{s_2+t_2(1-t_2)}\frac{1}{\sqrt{(t_1-t_2)(s_2t_2+(s_2+t_2)(1-t_1))}}\notag\\
&\leq C\iint _{0<s_2<\infty, 0<t_2<1}  \dd s_2 \dd t_2 \frac{1}{s_2+t_2(1-t_2)}\frac{1}{\sqrt{s_2+t_2}}\\
&\leq C\int_{0<t_2<1}\dd t_2 \frac{1}{\sqrt{t_2(1-t_2)}}<\infty.
\end{align*}

\subsection{Proof of \eqref{eq:A3}}\label{app:A3}
Observe that
\begin{align*}
&E\left[Y_{s,s+\e}^2e^{-\overline{J}_{0,1-\e}}\Phi_{0,1-\e}\right]\\
&\leq 
C\max|f|\int_{\begin{smallmatrix}
\sigma_1<\sigma_2<s<\tau_1<\tau_2<s+\e\\
\tau_1-\sigma_1\leq \e, \tau_2-\sigma_2\leq \e
\end{smallmatrix}}\dd \sigma_1\dd \tau_1\dd \sigma_2\dd \tau_2\\
&\hspace{7em}((\tau_2-\tau_1)(\tau_1-\sigma_1)+(\tau_1-\sigma_2)(\sigma_2-\sigma_1))^{-\frac{3}{2}}\\
&+C\max|f|\iint_{\begin{smallmatrix}
\sigma_1<\sigma_2<s<\tau_2<\tau_1<s+\e\\
\tau_1-\sigma_1\leq \e,
\tau_2-\sigma_2\leq \e
\end{smallmatrix}
}\dd \sigma_1\dd \tau_1\dd \sigma_2\dd \tau_2\\
&\hspace{7em}((\tau_1-\tau_2+\sigma_2-\sigma_1)(\tau_2-\sigma_2))^{-\frac{3}{2}}\\
&=:Y_1+Y_2.
\end{align*}
Taking the integral of $Y_1$  in $\tau_2\in [\tau_1,\sigma_2+\e]$, the integrand is given by \begin{align*}
&\frac{1}{(\tau_1-\sigma_1)}\left(\frac{1}{\sqrt{(\tau_1-\sigma_2)(\sigma_2-\sigma_1)}}-\frac{1}{\sqrt{(\sigma_2+\e-\tau_1)(\tau_1-\sigma_1)+(\tau_1-\sigma_1)(\sigma_2-\sigma_1)}}\right)\\
&\leq C\frac{1}{(\tau_1-\sigma_1)}\frac{1}{\sqrt{(\tau_1-\sigma_2)(\sigma_2-\sigma_1)}}.
\end{align*}
Taking the integral of $Y_1$ in $\tau_1\in [s,\sigma_1+\e]$ gives an upper bound \begin{align*}
\frac{2C}{\sigma_2-\sigma_1}\arctan \frac{\sqrt{\sigma_2-\sigma_1}(\sqrt{\sigma_1+\e -\sigma _2}-\sqrt{s-\sigma _2})}{\sigma_2-\sigma_1+\sqrt{(\sigma_1+\e -\sigma _2)(s-\sigma _2)}}\leq \frac{C}{\sqrt{(\sigma_2-\sigma_1)(s
-\sigma _2)}},\end{align*}
where we have used the integral \begin{align}
\int_d^e \frac{\dd x}{(x+a)\sqrt{bx+c}}=\frac{2}{\sqrt{ab-c}}\arctan\frac{\sqrt{ab-c}(\sqrt{be+c}-\sqrt{bd+c})}{ab-c+\sqrt{(be+c)(bd+c)}}\label{eq:integralarctan}
\end{align}
when $ab-c>0$. Then, taking the integrals in $s-\e<\sigma_1<\sigma_2<s$, we obtain an upper bound of $Y_1$, $C\max|f|\e$.

The integral of $Y_2$  in $\tau_1\in [\tau_2,\sigma_1+\e]$ is equal to \begin{align*}
&\frac{2}{(\sigma_2-\sigma_1)^{\frac{1}{2}}(\tau_2-\sigma_2)^\frac{3}{2}}-\frac{2}{(\sigma_2+\e-\tau_2)^\frac{1}{2}(\tau_2-\sigma_2)^{\frac{3}{2}}}\\
&\leq \frac{2}{(\sigma_2-\sigma_1)^{\frac{1}{2}}(\tau_2-\sigma_2)^\frac{3}{2}}.
\end{align*}
Taking the integral of $Y_2$ in $\tau_2\in [s,\sigma_1+\e]$ gives an upper bound \begin{align*}
\frac{4}{((\sigma_2-\sigma_1)(s-\sigma_2))^\frac{1}{2}}.
\end{align*}
Then, the integration in $s-\e<\sigma_1<\sigma_2<s$ gives the  upper bound $C \max|f|\e$ of $Y_2$. 

\subsection{Proof of \eqref{eq:A212}}\label{A212}

By Lemma \ref{lem:finitetra} and the some argument as in \eqref{eq:elbound}, we have that \begin{align}
&E\left[\int\int_{ 0\leq u\leq s\leq v\leq 1-\e, v-u\leq \e}\dd u\dd v\delta(\omega_{v}-\omega_{u})e^{-\overline{J}_{0,u}^{\e,\lambda}-\overline{J}_{v,1-\e}^{\e,\lambda}-\lambda J_{0,u;v,1-\e}^{\e}}|\Phi_{0,1}|J_{0,u;u,v}^{\e} e^{K_{\e,v-u,\lambda}} \right]\notag\\
&\leq C\max|f|\int_{s}^{s+\e}\dd v\int_{v-\e}^{s}\dd u\int_{u}^v \dd t \int_0^{t-\e}\dd r ((t-u)(u-r)+(v-t)(t-u)+(v-t)(u-r))^{-\frac{3}{2}}\label{eq:A2121}\\
&\leq C\max|f|\int_{s}^{s+\e}\dd v\int_{v-\e}^{s}\dd u\int_{u}^v \dd t \frac{1}{v-u} \frac{1}{\sqrt{\e(v-u)-(t-u)^2}}\notag\\
&\leq C\max|f|\int_{s}^{s+\e}\dd v\int_{v-\e}^{s}\dd u \frac{1}{v-u}\arcsin \sqrt{\frac{v-u}{\e}}\notag\\
&\leq C\max|f|\e,\notag
\end{align}
where we have used that $K_{\e,v-u,\lambda}\leq c$ for some constant $c>0$ if $v-u\leq \e$ in \eqref{eq:A2121}.

The same calculation yields \begin{align*}
&E\left[\int\int_{ 0\leq u\leq s\leq v\leq 1-\e, v-u\leq \e}\dd u\dd v\delta(\omega_{v}-\omega_{u})e^{-\overline{J}_{0,u}^{\e,\lambda}-\overline{J}_{v,1-\e}^{\e,\lambda}-\lambda J_{0,u;v,1-\e}^{\e}}|\Phi_{0,1}|J_{u,v;v,1-\e}^\e e^{K_{\e,v-u,\lambda}} \right]\\
&\leq C\max|f|\e.
\end{align*}

\subsection{Proofs of \eqref{eq:Remainder1}, \eqref{eq:Remainder2}}\label{app:Rem}

For  \eqref{eq:Remainder1}, the region \begin{align*}
T^4_{s,u,v,\e,\alpha,\beta,\gamma}:={T^3_{s,u,v}(\e^\alpha,\e^\beta)}\times \{(\sigma,\tau):0<\sigma<s+\e^\gamma, s-\e^\gamma<\tau<s+\e^\gamma,\tau-\sigma>\e\}
\end{align*}
 is partitioned into 15 regions $R_{s,u,v,\sigma,\tau}^{(i)}$ ($i=1,\dots,15$) due to the order of $\{s,u,v,s+\e,\sigma,\tau\}$ in $[0,s+\e^\gamma]$:
\begin{align*}
&R_{s,u,v,\sigma,\tau}^{(1)}=\{\sigma<\tau<u<s<v<s+\e\}\cap T^4_{s,u,v,\e,\alpha,\beta,\gamma},\\
&R_{s,u,v,\sigma,\tau}^{(2)}=\{\sigma<u<\tau<s<v<s+\e\}\cap T^4_{s,u,v,\e,\alpha,\beta,\gamma},\\
&R_{s,u,v,\sigma,\tau}^{(3)}=\{\sigma<u<s<\tau<v<s+\e\}\cap T^4_{s,u,v,\e,\alpha,\beta,\gamma},\\
&R_{s,u,v,\sigma,\tau}^{(4)}=\{\sigma<u<s<v<\tau<s+\e\}\cap T^4_{s,u,v,\e,\alpha,\beta,\gamma},\\
&R_{s,u,v,\sigma,\tau}^{(5)}=\{\sigma<u<s<v<s+\e<\tau\}\cap T^4_{s,u,v,\e,\alpha,\beta,\gamma},\\
&R_{s,u,v,\sigma,\tau}^{(6)}=\{u<\sigma<\tau<s<v<s+\e\}\cap T^4_{s,u,v,\e,\alpha,\beta,\gamma},\\
&R_{s,u,v,\sigma,\tau}^{(7)}=\{u<\sigma<s<\tau<v<s+\e\}\cap T^4_{s,u,v,\e,\alpha,\beta,\gamma},\\
&R_{s,u,v,\sigma,\tau}^{(8)}=\{u<\sigma<s<v<\tau<s+\e\}\cap T^4_{s,u,v,\e,\alpha,\beta,\gamma},\\
&R_{s,u,v,\sigma,\tau}^{(9)}=\{u<\sigma<s<v<s+\e<\tau\}\cap T^4_{s,u,v,\e,\alpha,\beta,\gamma},\\
&R_{s,u,v,\sigma,\tau}^{(10)}=\{u<s<\sigma<\tau<v<s+\e\}\cap T^4_{s,u,v,\e,\alpha,\beta,\gamma},\\
&R_{s,u,v,\sigma,\tau}^{(11)}=\{u<s<\sigma<v<\tau<s+\e\}\cap T^4_{s,u,v,\e,\alpha,\beta,\gamma},\\
&R_{s,u,v,\sigma,\tau}^{(12)}=\{u<s<\sigma<v<s+\e<\tau\}\cap T^4_{s,u,v,\e,\alpha,\beta,\gamma},\\
&R_{s,u,v,\sigma,\tau}^{(13)}=\{u<s<v<\sigma<\tau<s+\e\}\cap T^4_{s,u,v,\e,\alpha,\beta,\gamma},\\
&R_{s,u,v,\sigma,\tau}^{(14)}=\{u<s<v<\sigma<s+\e<\tau\}\cap T^4_{s,u,v,\e,\alpha,\beta,\gamma},\\
&R_{s,u,v,\sigma,\tau}^{(15)}=\{u<s<v<s+\e<\sigma<\tau\}\cap T^4_{s,u,v,\e,\alpha,\beta,\gamma}.
\end{align*}
We remark that $R^{(10)}$, $R^{(11)}$, $R^{(13)}$ are empty sets due to the restriction of $\tau-\sigma\geq \e$.
Then, we will estimate the integrals
\begin{align}
C\int_{R_{s,u,v,\sigma,\tau}^{(i)}}\dd s \dd u\dd v\dd \sigma\dd \tau Q^{(i)}(s,u,v,\sigma,\tau,\e)^{-\frac{3}{2}},\label{eq:Qintegral}
\end{align}
where $Q^{(i)}(s,u,v,\sigma,\tau,\e)$ is a polynomial function of $(s,u,v,\sigma,\tau,\e)$ on $R_{s,u,v,\sigma,\tau}^{(i)}$ for each $i=1,\dots,15$.

For each region, we can find an associated electrical circuit given in figure \ref{Fig:TypeCir} as follows:
\begin{itemize}
\item (1): $R_{s,u,v,\sigma,\tau}^{(5)}$.
\item (2): $R_{s,u,v,\sigma,\tau}^{(2)}$, $R_{s,u,v,\sigma,\tau}^{(3)}$, $R_{s,u,v,\sigma,\tau}^{(4)}$, $R_{s,u,v,\sigma,\tau}^{(7)}$, $R_{s,u,v,\sigma,\tau}^{(8)}$, $R_{s,u,v,\sigma,\tau}^{(9)}$, $R_{s,u,v,\sigma,\tau}^{(11)}$, $R_{s,u,v,\sigma,\tau}^{(12)}$, $R_{s,u,v,\sigma,\tau}^{(14)}$.
\item (3): $R_{s,u,v,\sigma,\tau}^{(1)}$, $R_{s,u,v,\sigma,\tau}^{(15)}$.
\item (4): $R_{s,u,v,\sigma,\tau}^{(6)}$, $R_{s,u,v,\sigma,\tau}^{(10)}$, $R_{s,u,v,\sigma,\tau}^{(13)}$.
\end{itemize}

\begin{figure}[ht]
\begin{center}
\includegraphics[width=4in,pagebox=cropbox,clip]{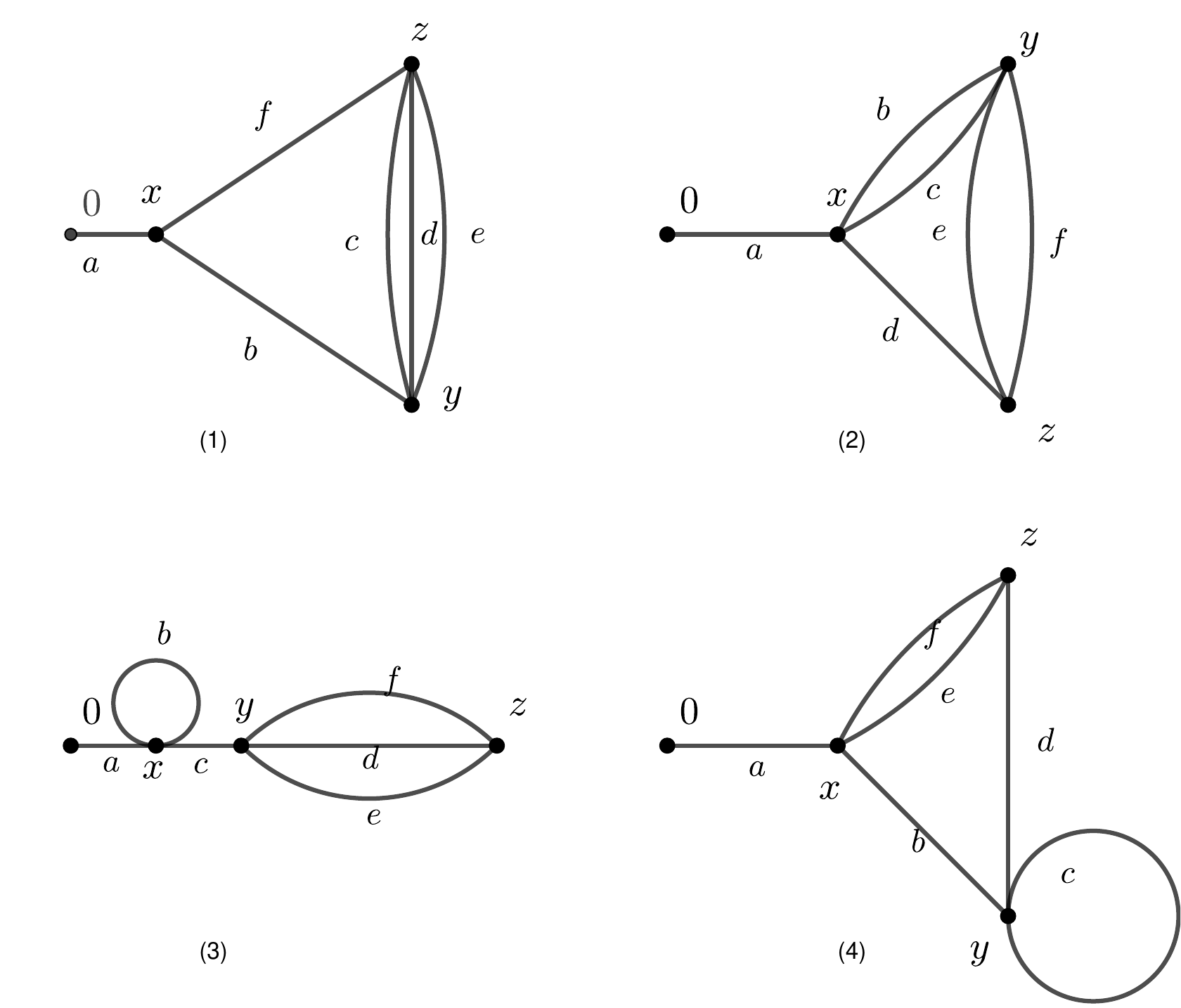}
\caption{$Q^{(i)}$ can be computed from these types of electrical circuits.  The associated electrical circuit for $R^{(15)}$ has the other version of (3) where the tail from $0$ is adjusted to $z$. Even for such a version, $Q^{(15)}$ has the same form. }\label{Fig:TypeCir}
\end{center}
\end{figure}
In particular, $Q^{(i)}$ has the form \begin{align*}
\begin{cases}
b(cd+de+ec)+f(cd+de+ec)+cde\quad &(1) \text{ in Figure }\ref{Fig:TypeCir}, \\
bc(e+f)+ef(b+c)+d(b+c)(e+f)\quad &(2) \text{ in Figure }\ref{Fig:TypeCir},\\
b(de+ef+fd)\quad &(3)\text{ in Figure }\ref{Fig:TypeCir},\\
c((b+d)e+(b+d)f+ef)&(4)\text{ in Figure }\ref{Fig:TypeCir}.
\end{cases}
\end{align*}
Applying this to each case and simplifying the polynomial, $Q^{(i)}$ is given as follows.

\begin{align*}
&Q^{(1)}=(\tau-\sigma)((s+\e-v)(v-s)+\e(s-u)),\\
&Q^{(2)}=(u-\sigma)(\tau-u)[(s+\e-v)+(v-s)]\\
&\hspace{5em}+(v-s)(s+\e-v)[(\tau-u)+(u-\sigma)]\\
&\hspace{5em}+(s-\tau)[(\tau-u)+(u-\sigma)][(s+\e-v)+(v-s)],\\
&Q^{(3)}=(u-\sigma)(v-\tau)[(s+\e-v)+(s-u)]\\
&\hspace{5em}+(s-u)(s+\e-v)[(v-\tau)+(u-\sigma)]\\
&\hspace{5em}+(\tau-s)[(v-\tau)+(u-\sigma)][(s+\e-v)+(s-u)],\\
&Q^{(4)}=(s+\e-\tau)[(\tau-v)+(u-\sigma)][(v-s)+(s-u)]\\
&\hspace{5em}+(s-u)(v-s)[(\tau-v)+(u-\sigma)]\\
&\hspace{5em}+(\tau-v)(u-\sigma)[(v-s)+(s-u)],\\
&Q^{(5)}=(\tau-s-\e)[(s-u)(v-s)+(v-s)(s+\e-v)+(s+\e-v)(s-u)]\\
&\hspace{5em}+(u-\sigma)[(s-u)(v-s)+(v-s)(s+\e-v)+(s+\e-v)(s-u)]\\
&\hspace{5em}+(s-u)(v-s)(s+\e-v),\\
&Q^{(6)}=(\tau-\sigma)((s+\e-v)(v-s)+\e(s-\tau+\sigma-u)),\\
&Q^{(7)}=(s+\e-v)[(v-\tau)+(\sigma-u)][(\tau-s)+(s-\sigma)]\\
&\hspace{5em}+(v-\tau)(\sigma-u)[(\tau-s)+(s-\sigma)]\\
&\hspace{5em}+(\tau-s)(s-\sigma)[(v-\tau)+(\sigma-u)],\\
&Q^{(8)}=(v-s)[(\tau-v)+(\sigma-u)][(s-\sigma)+(s+\e-\tau)]\\
&\hspace{5em}+(\tau-v)(\sigma-u)[(s-\sigma)+(s+\e-\tau)]\\
&\hspace{5em}+(s-\sigma)(s+\e-\tau)[(\tau-v)+(\sigma-u)],\\
&Q^{(9)}=(\sigma-u)[(s-\sigma)+(\tau-s-\e)][(s+\e-v)+(v-s)]\\
&\hspace{5em}+(s-\sigma)(\tau-s-\e)[(s+\e-v)+(v-s)]\\
&\hspace{5em}+(s+\e-v)(v-s)[(s-\sigma)+(\tau-s-\e)],\\ 
&Q^{(12)}=(v-\sigma)[(s+\e-v)+(s-u)][(\tau-s-\e)+(\sigma-s)]\\
&\hspace{5em}+(s+\e-v)(s-u)[(\tau-s-\e)+(\sigma-s)]\\
&\hspace{5em}+(\tau-s-\e)(\sigma-s)[(s+\e-v)+(s-u)],\\
&Q^{(14)}=(\sigma-v)[(v-s)+(s-u)][(\tau-s-\e)+(s+\e-\sigma)]\\
&\hspace{5em}+(s+\e-\sigma)(\tau-s-\e)[(v-s)+(s-u)]\\
&\hspace{5em}+(v-s)(s-u)[(\tau-s-\e)+(s+\e-\sigma)],\\
&Q^{(15)}=(\tau-\sigma)((s+\e-v)(v-s)+\e(s-u)).
\end{align*}

Estimates of all integrals in \eqref{eq:Qintegral} are obtained by explicit calculations.
By reversing the time, \eqref{eq:Remainder2} can be reduced to \eqref{eq:Remainder1}.

\subsection{Asymptotics of Expectation and Variance of $J_{0,1}^\e$}\label{kappa2}

In this subsection, we   prove \eqref{eq:kappa1}-\eqref{eq:kappa2}.

\subsubsection{Proof \eqref{eq:kappa1}}
Observe that \begin{align*}
E\left[J_{0,1}^\e\right]=\int_0^{1-\e}\dd s\int_{s+\e}^1\dd t p_{t-s}(0)&=\frac{2}{(2\pi)^\frac{3}{2}}\int_0^{1-\e}\left(\frac{1}{\e^\frac{1}{2}}-\frac{1}{(1-s)^\frac{1}{2}}\right)\dd s\\
&=\frac{2}{(2\pi)^\frac{3}{2}}\left(\frac{1-\e}{\e^\frac{1}{2}}-2(1-\sqrt{\e})\right).
\end{align*}

\subsubsection{Proof of \eqref{eq:kappa2-2}}
It is easy to see that 
\begin{align*}
&E\left[\left(J_{0,1}^\e\right)^2\right]\\
&=2\int_\e^{1-\e}\dd t\int_{t+\e}^1\dd u \int_0^{t-\e}\dd r\int_{r+\e}^t\dd s  \int_{\R^3}\dd x\int_{\R^3}\dd y p_r(0,x)p_{s-r}(x,x)p_{t-s}(x,y)p_{u-t}(y,y)\\
&\quad+2\int_0^{1-\e}\dd t\int_{t+\e}^1\dd u \int_0^t \dd r \int_{(r+\e)\vee t}^{u}\dd s\int_{\R^3}\dd x\int_{\R^3}\dd y p_{r}(0,x)p_{t-r}(x,y)p_{s-t}(y,x)p_{u-s}(x,y) \\
&\quad+2\int_0^{1-\e}\dd t\int_{t+\e}^1\dd u \int_t^{u-\e} \dd r \int_{r+\e}^{u}\dd s\int_{\R^3}\dd x\int_{\R^3}\dd y p_{t}(0,x)p_{r-t}(x,y)p_{s-r}(y,y)p_{u-s}(y,x)\\
&=:2V_1(\e)+2V_2(\e)+2V_3(\e). 
\end{align*}
We should now investigate the asymptotics of each term:

Observe that 
\begin{align*}
(2\pi)^3V_1(\e)&=\int_\e^{1-\e}\dd t\int_{t+\e}^1\dd u \int_0^{t-\e}\dd r\int_{r+\e}^t\dd s (s-r)^{-\frac{3}{2}}(u-t)^{-\frac{3}{2}}\\
&=2\int_\e^{1-\e}\dd t\int_{t+\e}^1\dd u (u-t)^{-\frac{3}{2}}\int_0^{t-\e}\dd r \left(\e^{-\frac{1}{2}}-(t-r)^{-\frac{1}{2}}\right)\\
&=2\int_\e^{1-\e}\dd t\int_{t+\e}^1\dd u (u-t)^{-\frac{3}{2}}\left((t-\e)\e^{-\frac{1}{2}}-2\sqrt{t}+2\sqrt{\e}\right)\\
&=4\int_\e^{1-\e}\dd t\left(\e^{-\frac{1}{2}}-(1-t)^{-\frac{1}{2}}\right)\left((t-\e)\e^{-\frac{1}{2}}-2\sqrt{t}+2\sqrt{\e}\right)\\
&=4\int_\e^{1-\e}\dd t\left(\e^{-1}t-\e^{-\frac{1}{2}}(1-t)^{-\frac{1}{2}}t-2\sqrt{t}{\e}^{-\frac{1}{2}}\right)+C_1(\e),
\end{align*}
where $C_1(\e)$ converges to some constant as $\e\to 0$. Since  \begin{align*}
&\int_\e^{1-\e}\dd t\left(\e^{-1}t-\e^{-\frac{1}{2}}(1-t)^{-\frac{1}{2}}t-2\sqrt{t}{\e}^{-\frac{1}{2}}\right)\\
&=\frac{1}{2}\frac{(1-\e)^2-\e^2}{\e}-\frac{1}{\sqrt{\e}}\left(2\left(\e\sqrt{1-\e}-(1-\e)\sqrt{\e}\right)+\frac{4}{3}\left((1-\e)^\frac{3}{2}-\e^{\frac{3}{2}}\right)\right)-\frac{4}{3}\frac{1}{\sqrt{\e}}\left((1-\e)^\frac{3}{2}-\e^\frac{3}{2}\right),
\end{align*}
there exists a constant $\check{C}_1\in \R$ such that \begin{align}
\lim_{\e\to 0}\left(2V_1(\e)-E[J_{0,1}^\e]^2+\frac{8}{(2\pi)^3}\frac{2}{3}\frac{(1-\e)^\frac{3}{2}}{\sqrt{\e}}\right)=\check{C}_1.\label{eq:Var1}
\end{align}
Similarly, we have \begin{align*}
(2\pi)^3V_3(\e)&=\int_0^{1-\e}\dd t\int_{t+\e}^1\dd u \int_t^{u-\e} \dd r \int_{r+\e}^{u}\dd s(s-r)^{-\frac{3}{2}}(u-s+r-t)^{-\frac{3}{2}}\\
&=\int_0^{1-\e}\dd r \int_{r+\e}^1\dd s(s-r)^{-\frac{3}{2}}\int_0^r\dd t\int_s^1 \dd u(u-s+r-t)^{-\frac{3}{2}}\\
&=2\int_0^{1-\e}\dd r \int_{r+\e}^1\dd s(s-r)^{-\frac{3}{2}}\int_0^r\dd t\left((r-t)^{-\frac{1}{2}}-(1-s+r-t)^{-\frac{1}{2}}\right)\\
&=4\int_0^{1-\e}\dd r \int_{r+\e}^1\dd s(s-r)^{-\frac{3}{2}}\left(\sqrt{r}-\sqrt{1-s+r}+\sqrt{1-s}\right).
\end{align*}
Since \begin{align*}
\int_0^{1-\e}\dd r \int_{r+\e}^1\dd s(s-r)^{-\frac{3}{2}}\sqrt{r}&=\int_0^{1-\e}\dd r 2\sqrt{r}\left(\e^{-\frac{1}{2}}-(1-r)^{-\frac{1}{2}}\right)=\frac{4}{3}\frac{(1-\e)^\frac{3}{2}}{\sqrt{\e}}+C_{3,1}(\e)\\
\int_0^{1-\e}\dd r \int_{r+\e}^1\dd s(s-r)^{-\frac{3}{2}}\sqrt{1-s+r}&=\int_{0}^{1-\e}\dd r \left(2\left(\frac{\sqrt{1-\e}}{\sqrt{\e}}-\frac{\sqrt{r}}{\sqrt{1-r}}\right)-\int_\e^{1-r}\dd \sigma \frac{1}{\sqrt{\sigma(1-\sigma)}}\right)\\
&=2\frac{(1-\e)^\frac{3}{2}}{\sqrt{\e}}+C_{3,2}(\e)\\
\int_0^{1-\e}\dd r \int_{r+\e}^1\dd s(s-r)^{-\frac{3}{2}}\sqrt{1-s}&=2\int_\e^1 \dd s \sqrt{1-s}\left(\e^{-\frac{1}{2}}-s^{-\frac{1}{2}}\right)\\
&=\frac{4}{3}\frac{(1-\e)^\frac{3}{2}}{\sqrt{\e}}+C_{3,3}(\e),
\end{align*}
where $C_{3,i}(\e)$ ($i=1,2,3$) converge to some constants as $\e\to 0$, there exists a constant $\check{C}_3$ such that  \begin{align}
\lim_{\e\to0}\left(2V_3(\e)-\frac{8}{(2\pi)^3}\frac{2}{3}\frac{(1-\e)^\frac{3}{2}}{\sqrt{\e}}\right)=\check{C}_3.\label{eq:Var2}
\end{align}

To calculate $V_2(\e)$, we first divide the integration as follows:\begin{align}
&\int_0^{1-\e}\dd t\int_{t+\e}^1\dd u \int_0^t \dd r \int_{(r+\e)\vee t}^{u}\dd sF(r,s,t,u)\label{eq:V3repre}\\
&=\int_0^{1-\e}\dd r\int_{r+\e}^1\dd s \int_r^s \dd t \int_{(t+\e)\vee s}^{1}\dd uF(r,s,t,u)\notag\\
&=\int_0^{1-\e}\dd r\int_{r+\e}^1\dd s \int_{r}^{s-\e} \dd t \int_{s}^{1}\dd uF(r,s,t,u)+\int_0^{1-\e}\dd r\int_{r+\e}^1\dd s \int_{s-\e}^s \dd t \int_{t+\e}^{1}\dd uF(r,s,t,u)\notag,
\end{align}
where $F(r,s,t,u)=\frac{1}{(2\pi)^3}\left((t-r)(s-t)+(u-s)(s-t)+(u-s)(t-r)\right)^{-\frac{3}{2}}$  is a continuous function on $\{0<r<t<s<u<1\}$. 
In particular, we have \begin{align*}
&(2\pi)^3V_{2,1}(\e)\\
&:=\int_0^{1-\e}\dd r\int_{r+\e}^1\dd s \int_{r}^{s-\e} \dd t \int_{s}^{1}\dd uF(r,s,t,u)\\
&=2\int_0^{1-\e}\dd r\int_{r+\e}^1\dd s \int_{r}^{s-\e} \dd t\frac{1}{s-r}\left(\frac{1}{\sqrt{(t-r)(s-t)}}-\frac{1}{\sqrt{(1-s)(s-r)+(t-r)(s-t)}}\right)
\intertext{and}
&(2\pi)^3V_{2,2}(\e)\\
&:=\int_0^{1-\e}\dd r\int_{r+\e}^1\dd s \int_{s-\e}^s \dd t \int_{t+\e}^{1}\dd uF(r,s,t,u)\\
&=2\int_0^{1-\e}\dd r\int_{r+\e}^1\dd s \int_{s-\e}^s \dd t\frac{1}{s-r}\\
&\hspace{5em}\times \left(\frac{1}{\sqrt{(t+\e-s)(s-r)+(t-r)(s-t)}}-\frac{1}{\sqrt{(1-s)(s-r)+(t-r)(s-t)}}\right).
\end{align*}
Thus, \begin{align*}
&(2\pi)^3\frac{\dd}{\dd \e}(V_{2,1}(\e)+V_{2,2}(\e))\\
&=-\int_0^{1-\e}\dd r\int_{r+\e}^1\dd s \int_{s-\e}^s \dd t \left((t+\e-s)(s-r)+(t-r)(s-t)\right)^{-\frac{3}{2}}\\
&\hspace{1em}-\frac{2}{\e}\int_0^{1-\e}\dd r  \int_{r}^{r+\e} \dd t \\
&\hspace{3em}\left(\frac{1}{\sqrt{(t-r)\e+(t-r)(r+\e-t)}}-\frac{1}{\sqrt{(1-r-\e)\e+(t-r)(r+\e-t)}}\right)\\
&=-\int_\e^{1}\dd s\int_{s-\e}^s\dd t \int_{0}^{s-\e} \dd r \left((t+\e-s)(s-t)+\e(t-r)\right)^{-\frac{3}{2}}\\
&\hspace{1em}-\frac{2}{\e}\int_{0}^{1-\e}\dd r \int_0^{\e}\frac{\dd u}{\sqrt{u(2\e-u)}}+O\left(\e^{-\frac{1}{2}}\right)\\
&=-\frac{2}{\e}\int_\e^{1}\dd s\int_{s-\e}^s\dd t \left(\frac{1}{\sqrt{\e^2-(s-t)^2}}-\frac{1}{\sqrt{\e s-(s-t)^2}}\right) -\frac{\pi}{\e}(1-\e)+O\left(\e^{-\frac{1}{2}}\right)\ \\
&=-\frac{2\pi}{\e}(1-\e)+\frac{2}{\e }\int_\e^{1}\dd s\arcsin \sqrt{\frac{\e}{s}}=-\frac{2\pi}{\e}+O\left(\e^{-\frac{1}{2}}\right).
\end{align*}
Therefore, there exists a constant $\check{C}_2\in \R$ such that \begin{align}
\lim_{\e\to 0}\left(2V_2(\e)+\frac{2}{(2\pi)^2}\log \e\right)=\check{C}_2.\label{eq:Var3}
\end{align}
 \eqref{eq:Var1}-\eqref{eq:Var3} yield \eqref{eq:kappa2-2}.

\subsubsection{Proof of \eqref{eq:kappa2}}

First, we remark that the proof \eqref{eq:kappa2} was already indicated in \cite[p.20]{Bol02} but we shall give more detail by the following calculation.

As in \eqref{eq:V3repre}, we have \begin{align*}
(2\pi)^3\kappa_2(\e)&=\int_0^\e \dd s\int_{s+\e}^1\dd u\int_\e^{u}\dd t \frac{1}{(s(t-s)+(t-s)(u-t)+(u-t)s)^\frac{3}{2}}\\
&+\int_\e^{1-\e} \dd s\int_{s+\e}^1\dd u\int_{s}^{u}\dd t \frac{1}{(s(t-s)+(t-s)(u-t)+(u-t)s)^\frac{3}{2}}.
\end{align*}
Hence, \begin{align*}
(2\pi)^3\frac{\dd }{\dd \e}\kappa_2(\e)=&-\int_0^\e\dd s \int_\e^{s+\e}\dd t(s(t-s)+(s+\e-t)t)^{-\frac{3}{2}}\\
&-\int_\e^{1-\e}\dd s \int_{s}^{s+\e}\dd t (s(t-s)+(s+\e-t)t)^{-\frac{3}{2}}\\
&-\int_0^\e\dd s\int_{s+\e}^1\dd u (s(\e-s)+(u-\e)\e)^{-\frac{3}{2}}.
\end{align*}
Using the integral \begin{align*}
\int_s^t (ax^2+b x+c)^{-\frac{3}{2}}\dd x=\left[\frac{4ax+2b}{(4ac-b^2)\sqrt{ax^2+bx+c}}\right]_{s}^t,
\end{align*}
we obtain that
\begin{align*}
&\int_0^\e\dd s \int_\e^{s+\e}\dd t(s(t-s)+(s+\e-t)t)^{-\frac{3}{2}}\\
&=\int_0^\e\dd s \left(\frac{2}{(4s+\e)\sqrt{\e s}}-\frac{2}{\e\sqrt{2\e s-s^2}}+\frac{12s}{\e(4s+\e)\sqrt{2\e s-s^2}}\right),\\
&\int_\e^{1-\e}\dd s \int_{s}^{s+\e}\dd t (s(t-s)+(s+\e-t)t)^{-\frac{3}{2}}\\
&=\int_\e^{1-\e}\dd s \frac{4}{(4s+\e)\sqrt{\e s}},\\
&\int_0^\e\dd s\int_{s+\e}^1\dd u (s(\e-s)+(u-\e)\e)^{-\frac{3}{2}}\\
&=\int_0^\e \dd s\frac{2}{\e}\left(\frac{1}{\sqrt{2\e s-s^2}}-\frac{1}{\sqrt{\e(1-\e)+s(\e-s)}}\right).
\end{align*}
Thus, we find that
\begin{align*}
(2\pi)^3\frac{\dd }{\dd \e}\kappa_2(\e)&=-\int_0^\e \dd s\left(\frac{2}{(4s +\e)\sqrt{\e s}}+\frac{12 \sqrt{s}}{\e (4s +\e)\sqrt{2\e -s}}\right)-\int_\e^{1-\e}\dd s\frac{4}{(4s +\e)\sqrt{\e s}}\\
&\quad+\int_0^\e\dd s \frac{2}{\e}\frac{1}{\sqrt{\e(1-\e)+s(\e-s)}}\\
&=-\int_0^{1-\e}\dd s\frac{4}{(4s +\e)\sqrt{\e s}}+\int_0^\e \dd s\frac{2}{(4s +\e)\sqrt{\e s}}-\int_0^\e \dd s \frac{12 \sqrt{s}}{\e (4s +\e)\sqrt{2\e -s}}\\
&\quad +O\left(\e^{-\frac{1}{2}}\right).
\end{align*}
It is easy to see from \eqref{eq:integralarctan} that \begin{align*}
-\int_0^{1-\e}\dd s\frac{4}{(4s +\e)\sqrt{\e s}}+\int_0^\e \dd s\frac{2}{(4s +\e)\sqrt{\e s}}&=-\frac{4}{\e}\arctan \left(2\sqrt{\frac{1-\e}{\e}}\right)+\frac{2}{\e}\arctan 2\\
&=-\frac{ 2\pi}{\e}+O\left(\e^{-\frac{1}{2}}\right)+\frac{\pi}{\e}-\frac{2}{\e}\arctan \frac{1}{2}.
\end{align*}
Also, we have \begin{align*}
\int_0^\e \dd s \frac{12 \sqrt{s}}{\e (4s +\e)\sqrt{2\e -s}}&=\frac{12}{\e}\int_0^1\dd t \frac{\sqrt{t}}{(4t+1)\sqrt{2-t}} &&\\
&=\frac{12}{\e}\int_0^1\dd u \frac{4u^2}{(9u^2 +1)(u^2 +1)}&&(u=\sqrt{\frac{t}{2-t}})\\
&=\frac{3\pi}{2\e}-\frac{2}{\e}\arctan 3=\frac{\pi}{2\e}+\frac{2}{\e}\arctan \frac{1}{3}.
\end{align*}
Since $\arctan \frac{1}{2}+\arctan \frac{1}{3}=\frac{\pi}{4}$, \begin{align}
(2\pi)^3\frac{\dd }{\dd \e}\kappa_2(\e)=-\frac{2\pi}{\e}+O\left(\e^{-\frac{1}{2}}\right).\label{eq:kappa2order}
\end{align}

We can also compute
\begin{align}
&
\int_{s-\e^\beta}^s\dd u\int_{s\vee (u+\e)}^{s+\e}\dd v\frac{(2\pi)^{-3}}{((s+\e-v)(v-s)+(s+\e-v)(s-u)+(v-s)(s-u))^\frac{3}{2}}\notag\\
&=
\int_{s}^{s+\e}\dd v\left(\frac{2(2\pi)^{-3}}{\e((s+\e-v)(v+\e-s))^\frac{1}{2}}-\frac{2(2\pi)^{-3}}{\e\left(\e^{\beta+1}+(s+\e-v)(v-s)\right)^\frac{1}{2}}\right)\notag\\
&=\frac{(2\pi)^{-3}\pi }{\e}+O(\e^{-\frac{\beta+1}{2}})=-\frac{1}{2}\frac{\dd }{\dd \e}\kappa_2(\e)+O(\e^{-\frac{\beta+1}{2}})\label{eq:kappa2order2}
\end{align}
for $\e^\alpha<s<1-\e^\alpha$ with $\alpha\in (0,\frac{1}{2})$ and $\beta\in (\frac{1}{2},1)$.

\section*{Acknowledgements}
This work was supported by JSPS KAKENHI Grant Numbers 19H00643, 21H00988, 21K03302, 22K03351, 22H00099, 22H01128. The authors thank Professor M.~Hairer for a fruitful discussion. The aithors also thank anonymous referees for their very useful comments.
They are very indebted to Professor Michael R\"ockner and the late Professor Xian Yin Zhou for the unpublished preprint \cite{ARZ96}, which gave the initial spark and basic methods for the working on the present paper.

\end{document}